\documentclass[11pt,reqno]{amsart}
\usepackage{amsmath,amsthm,amsfonts,amssymb,mathrsfs,bm,graphicx,stmaryrd}
\usepackage[export]{adjustbox}

\usepackage{mathtools}
\usepackage{dsfont}
\usepackage{multicol}
\usepackage[colorlinks=true,linkcolor=blue]{hyperref}
\hypersetup{bookmarksdepth=1}
\usepackage{enumitem}
\usepackage{bbm}
\usepackage{mathrsfs}
\usepackage{versions}
\usepackage{subcaption}

\usepackage[letterpaper,hmargin=1.0in,vmargin=1.0in]{geometry}
\parskip 	\smallskipamount

\newcommand*{\rom}[1]{\uppercase\expandafter{\romannumeral #1\relax}}

\newtheorem{theorem}{Theorem}[section]
\newtheorem{corollary}[theorem]{Corollary}
\newtheorem{lemma}[theorem]{Lemma}
\newtheorem{proposition}[theorem]{Proposition}

\theoremstyle{remark}
\newtheorem{remark}[theorem]{\bf Remark}

\numberwithin{equation}{section}

\title{Geodesic trees in last passage percolation and some related problems}

\author[Bal\'azs, Basu, Bhattacharjee]{M\'arton Bal\'azs, Riddhipratim Basu, Sudeshna Bhattacharjee}

\begin{document}

\begin{abstract}
For the exactly solvable model of exponential last passage percolation on $\mathbb{Z}^2$, it is known that given any non-axial direction, all the semi-infinite geodesics starting from points in $\mathbb{Z}^2$ in that direction almost surely coalesce, thereby forming a geodesic tree which has only one end. It is widely understood that the geodesic trees are important objects in understanding the geometry of the LPP landscape. In this paper we study several natural questions about these geodesic trees and their intersections. In particular, we obtain optimal (up to constants) upper and lower bounds for the (power law) tails of the height and the volume of the backward sub-tree rooted at a fixed point. We also obtain bounds for the probability that the sub-tree contains a specific vertex, e.g. the sub-tree in the direction $(1,1)$ rooted at the origin contains the vertex $-(n,n)$, which answers a question analogous to the well-known \emph{midpoint problem} in the context of semi-infinite geodesics. Furthermore, we obtain bounds for the probability that a pair of intersecting geodesics both pass through a given vertex. These results are interesting in their own right as well as useful in several other applications. 
\end{abstract}

\address{M\'arton Bal\'azs, School of Mathematics, University of Bristol, Bristol, UK} 

\email{m.balazs@bristol.ac.uk}

\address{Riddhipratim Basu, International Centre for Theoretical Sciences, Tata Institute of Fundamental Research, Bangalore, India} 

\email{rbasu@icts.res.in}

\address{Sudeshna Bhattacharjee, Department of Mathematics, Indian Institute of Science, Bangalore, India}
\email{sudeshnab@iisc.ac.in }

\maketitle

\tableofcontents
\section{Introduction}
\label{Introduction}
%\noindent

%We consider the exponential last passage percolation model on $\mathbb{Z}^2$ where we assign i.i.d. Exp(1) random variables to each vertex of $\mathbb{Z}^2$. 
Both finite and semi-infinite geodesics (optimum weight attaining random paths) are important objects in the study of random growth models such as first and last passage percolation. In this paper we study the landscape of geodesics for exponential last passage percolation (LPP) on $\mathbb{Z}^2$, a canonical model of planar random growth in the Kardar-Parisi-Zhang (KPZ) universality class. Owing to connections to several other interesting processes such as the Totally Asymmetric Simple Exclusion Process (TASEP), and this model being one of the few exactly solvable models in the KPZ class, exponential LPP is one of the most extensively studied model of last passage percolation. In \cite{FP05} Ferrari and Pimentel initiated the study of semi-infinite geodesics in exponential last passage percolation model on $\mathbb{Z}^2$ and proved statements regarding existence, uniqueness and coalescence of semi-infinite geodesics in different directions. It is known that for a fixed $\alpha \in [0, \frac{\pi}{2}]$ (a direction in first quadrant), almost surely starting from each vertex of $\mathbb{Z}^2$, there exists a unique semi-infinite geodesic in the direction $\alpha$ \cite{FP05,C11}. Moreover, for a fixed non-axial 
%(\textcolor{red}{non-axial?})
direction $\alpha$ (i.e., $\alpha \in (0, \frac{\pi}{2})$) almost surely any two semi-infinite geodesics in that direction coalesce \cite{FP05,C11}. Further, it is also proved now that almost surely there does not exist any bi-infinite geodesics except the trivial bi-geodesics in the axial directions \cite{BBS20,BHS21}. Hence, for each fixed non axial direction almost surely, the geodesics form a geodesic tree with only one end (see Figure \ref{fig: Geodesic tree}). It has also been shown \cite{bal-bus-sepp_stabi,bus-fer_univ_geo_tree} that such semi-infinite geodesic trees locally coincide with the point-to-point trees to a far, but finite destination.
%(\textcolor{red}{this is not enough, what if there are bigeodesics?}) 
The semi-infinite geodesic trees will be our main object of interest in this paper.

\begin{figure}[t!]
     \centering
     \includegraphics[width=12 cm]{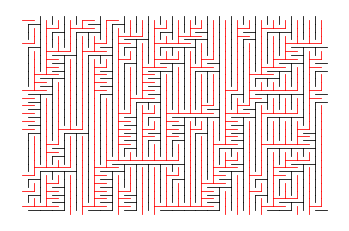}
    \caption{A section of the geodesic tree (in black) and the dual geodesic tree (in red) in the $45^{\circ}$ direction in exponential LPP on $\mathbb{Z}^2$. Simulation by David Harper.} 
    \label{fig: Geodesic tree} 
     \end{figure}

%\textcolor{red}{we need a figure for the sub-tree, perhaps a simulation?}

%The key tools to study the geometry of both finite and semi-infinite geodesics are, asymptotic behaviours of last passage time, transversal fluctuation of geodesics and coalescence of geodesics.

The intricate relation between the understanding of the fluctuations of passage times and the geometry of geodesics has long been known; it is predicted (and proved in some special cases of last passage, as well as under certain unproven hypothesis for first and last passage percolation models with general weight distribution) that the scaling exponent for the transversal fluctuations of the geodesics and the scaling exponent for the fluctuations of passage time are related by the so-called KPZ relation. Since the one point fluctuations of passage times are very well understood for exactly solvable models including exponential LPP, a lot of information regarding the geodesic geometry can be obtained. 
%In very few cases of exactly solvable models including the exponential LPP these are well understood now. 
%In \cite{Jo00} Johansson proved the exact distribution formula for last passage time for expoential LPP which turned out to be the largest eigen value distribution of a certain random matrix. The convergence in distribution to Tracy-Widom limit and $n^{1/3}$ fluctuation was also established in \cite{Jo00}.
Using the $n^{1/3}$ fluctuation results in \cite{BDJ99}, the transversal fluctuation exponent of $2/3$ was first established by Johansson in \cite{J00}, and later with different methods by Cator and Groeneboom \cite{cat-gro_cuber_hammersley} in case of Poissonian LPP. In case of exponential LPP, the exponent $2/3$ has been obtained for both finite and semi-infinite geodesics in \cite{BCS06,BSS14,BSS17B,BGZ21,em-jan-sep_opt_order_exit,em-geo-ort_opt_order_mom} using stationary and coupling arguments as well as the one point estimates of \cite{Jo00} and more quantitative moderate deviation estimates of \cite{LR10}. As mentioned before, another interesting behaviour of geodesics is coalescence. Questions about coalescence time for semi-infinite geodesics and coalescence structure of both finite and semi-infinite geodesics have been studied in \cite{cou-hein_coex,Pim16,GRS15+,BSS17B,Z20,S20,SS19,BHS21,bal-bus-sepp_stabi,bus-fer_univ_geo_tree}. %\textcolor{red}{cite Lingfu Zhang paper}.

%\red{describe the geodesic tree informally with references, give a brief description about the results about coalescence, transversal fluctuations etc. and give some references.}

In this paper we take forward the study of the geodesic tree in a fixed non-axial direction by investigating several natural questions for both finite and semi-infinite geodesics. As already mentioned, the geodesic tree is one-ended almost surely, and therefore the backward sub-tree rooted at a given vertex (say, the origin) is almost surely finite \footnote{By the backward sub-tree here, we mean the sub-tree consisting of all vertices $v$ such that the semi-infinite geodesic started at $v$ passes through the origin.}. It is also not hard to see that the expected size of such sub-trees is infinite. %(\textcolor{red}{check this, perhaps by mass transport?}). 
One would therefore expect the tails for the height and size of such sub-trees to exhibit power law decay. Indeed our first result (see Theorem \ref{first_theorem}) obtains these exponents, and establishes up-to-constant upper and lower bounds of the tail probabilities for sub-trees of the geodesic tree in a non-axial direction.

In our second main result, we consider the probability that the backward sub-tree rooted at a point contains a given point at a certain depth e.g., the probability that the point $-(n,n)$ is contained in the sub-tree (in direction $(1,1)$) rooted at the origin. Clearly, this is the same as computing the probability that the semi-infinite geodesic (in the direction $(1,1)$) from $-(n,n)$ passes through the origin. This problem should be compared to the classical \emph{mid-point problem} in first passage percolation \cite{BKS04} which asks for the probability that the geodesic between $(-n,0)$ and $(n,0)$ passes through the origin. In particular, the question was to show that this probability goes to 0 as $n\to \infty$. This problem has recently been solved under certain hypothesis in \cite{DH17,DEP22} and without any assumption in \cite{AH16}.
%(\textcolor{red}{reference}). 
Quantitatively optimal solution was obtained in the exactly solvable case of exponential LPP in \cite{BB21}.
\begin{comment}
    For semi-infinite geodesics an upper bound was proved in \cite{Sep17}.
    \end{comment}
    We provide up-to-constants solution of this variant of the midpoint problem in Theorem \ref{second_theorem}. 
%Our second main theorem is about midpoint problem.  Historically, midpoint problem is about finding the probability that $(n,n)$ will lie on the geodesic joining $(0,0)$ and $(2n,2n)$. This problem was analysed for finite geodesics in \cite{BHS21,BB21}. We derive a result analogous to the midpoint problem in context of semi-infinite geodesics (see Theorem \ref{second_theorem}). 
Taking this question further, we analyse the probability of a given vertex lying on the intersection of two different geodesics (both finite and semi-infinite). This will be our third main result (see Theorem \ref{main_theorem_1}). In particular, we prove tight (up to logarithmic factors) upper and lower bounds of this probability for geodesics going in macroscopically different directions. 

The main inputs we use are the one point moderate deviations estimates for passage times in exponential LPP from \cite{LR10}, and its consequences (tails estimates for minimum and maximum passage times in a parallelogram) derived in \cite{BSS14,BGZ21}. Although it should be possible to prove all our results using only these inputs, for certain results we shall also use an additional integrable input for convenience. For exponential LPP, it is known that the Busemann increments along any anti-diagonal line is distributed as a certain two sided symmetric random walk (precise statement later, \cite[Theorem 4.2]{Sep17}).
%(\textcolor{red}{reference})
We shall use these facts while proving the semi-infinite geodesic lower bounds in Theorems \ref{second_theorem} and \ref{main_theorem_1}. 

By way of the proofs of our main theorems, we prove several auxiliary results that might be of independent interest, including refinements of certain existing results in the literature. We expect the results of this paper to also be useful in further study of the geodesics trees in LPP and their interactions. In particular, properties of jointly realised geodesics are explored and used to model road layouts in a forthcoming paper involving two of the authors of this article. Results we obtain here for interactions between geodesics going to different directions form an essential input for that work.

\subsection{Statement of main results}
%In this subsection, we precisely state the main results of this paper. Before doing so, we need to define the model.\\
Before moving on to the precise statements, we first define the basic set-up. We assign i.i.d. random variables $\{\tau_{v}\}_{v \in \mathbb{Z}^2}$ to each vertex of $\mathbb{Z}^2$, where $\tau_{v}$'s are distributed as Exp(1). Let $u,v \in \mathbb{Z}^2$ be such that $u \leq v$ (i.e., if $u=(u_1,u_2),v=(v_1,v_2)$ then $u_1 \leq v_1$ and $u_2 \leq v_2$). For an up-right path $\gamma$ between $u$ and $v$ we define $\ell(\gamma)$, the \textit{passage time of $\gamma$} and $T_{u,v}$, the \textit{last passage time between $u$ and $v$} by 
\begin{displaymath}
    \ell (\gamma) := \sum_{w \in \gamma \setminus \{u,v\}} \tau_{w}\footnote{Our definition of passage time of a path excludes both the initial and final vertex. This is slightly different from standard definitions where we include both the vertices. We will work with this definition because of some technical reasons which will be clear from the proofs. Also note that excluding initial and final vertices does not change the geodesics. All the tail estimate results for last passage times that we will use in this paper hold for this definition as well.}
\end{displaymath} 
and
\begin{displaymath}
T_{u,v}:=\max\{ \ell(\gamma): \gamma \text{ is an up-right path from } u \text{ to } v\}
\end{displaymath}
respectively.
\begin{comment}Our definition of last passage time of a path excludes both the initial and final vertex. This is slightly different from standard definitions where we include both the vertices. We will work with this definition because of some technical reasons which will be clear from the proofs. Also note that Excluding initial and final vertices does not change the geodesics
\end{comment}
%(\textcolor{red}{make some comment about excluding the end vertices}). 
Clearly, as the number of up-right paths between $u$ and $v$ is finite, the maximum is always attained. Between any two (ordered) points $u,v \in \mathbb{Z}^2$, maximum attaining paths are called \textit{geodesics}. As $\tau_{v}$ has a continuous distribution, almost surely, between any two points $u,v \in \mathbb{Z}^2$, there exists a unique geodesic denoted by $\Gamma_{u,v}$. An infinite up-right path $\{u_0,u_1,u_2...\}$ in $\mathbb{Z}^2$ is called a \textit{semi-infinite geodesic starting from $u_0$}, if every finite segment of the infinite path is a geodesic. \begin{comment}$\theta \in \mathbb{R}^2$ with $|\theta|=1$ is said to be a \textit{direction}.
\end{comment}
A semi-infinite geodesic starting from $u \in \mathbb{Z}^2$ is said to have \textit{direction} $\theta$ if $\lim_{n \rightarrow \infty} \frac{u_n}{\|u_n\|}$ exists and equal to $\theta\in [0,\frac{\pi}{2}]$ (we shall always identify a unit vector with the corresponding angle in the polar coordinate); $\Gamma_{u_0}^{\alpha}$ will denote the semi-infinite geodesic starting from $u_0$ in the direction $\alpha$. From almost sure uniqueness of geodesics, together with planarity, it follows that geodesics are ordered (i.e., two geodesics with pairs of starting and ending points having the same order in the spatial co-ordinate, remain ordered throughout their journeys and cannot cross). We will use this fact multiple times throughout this article. As mentioned before we are interested in the geodesic tree. Before proceeding further we fix some more notations which will be frequently used.
\paragraph{\textbf{Notations}}
For $r \in \mathbb{Z}$ let $\boldsymbol{r}$ denote the vertex $(r,r) \in \mathbb{Z}^2$. For $k \in \mathbb{Z}, \boldsymbol{r}_k$ denotes the vertex $(r-kr^{2/3},r+kr^{2/3}) \in \mathbb{Z}^2$ \footnote{Note that the correct definition should be using floor or ceiling function. But to avoid too many notations we will work with $\boldsymbol{r}_k$ as defined. Throughout the article we will many times avoid floor or ceiling functions to reduce notational overhead. The reader can easily check that this does not affect any argument in a non-trivial way.}. For convenience, sometimes we will work with the rotated axes $x+y=0$ and $x-y=0$. They will be called space axis and time axis respectively. For a vertex $v \in \mathbb{Z}^2$, $\phi(v)$ will denote the time coordinate of $v$ and $\psi(v)$ will denote the space coordinate of $v$. Precisely, we define
\begin{displaymath}
    \phi(u,v):=u+v,\qquad\psi(u,v):=u-v.
\end{displaymath}The line $x+y=T$ will be denoted by $\mathcal{L}_T$. For $u,v \in \mathbb{Z}^2$ with $\phi(u)=\phi(v)$ and $\psi(u) \leq \psi(v)$ we define
\begin{displaymath}
[u,v]:=\{w \in \mathbb{Z}^2: \phi(u)=\phi(w)=\phi(v), \psi(u) \leq \psi(w) \leq \psi(v)\}.
\end{displaymath} 
Throughout this article, $C$ and $c$ will be used to denote generic constants whose values may change from line to line. In particular cases, when a same constant is used multiple times we will use numbered constants (e.g.\ $C_1,C_2$ etc.) to denote them.

Before proceeding further we define some new notations which will be useful. For $u,v \in \mathbb{Z}^2$ define the centred last passage time and a variant of it.
\begin{displaymath}
      \widetilde{T_{u,v}}=T_{u,v}-\mathbb{E}(T_{u,v});
\end{displaymath}
\begin{displaymath}
      \underline{T_{u,v}}=T_{u,v}+\tau_{u};
      \end{displaymath}
      \begin{displaymath}
    \widetilde{\underline{T_{u,v}}}=\underline{T_{u,v}}-\mathbb{E}(\underline{T_{u,v}}).
\end{displaymath}
Note that $\underline{T_{u,v}}$ is similar to the last passage time except for the fact that we include the initial vertex weight. Also note that $\mathbb{E}(\underline{T_{u,v}})=\mathbb{E}(T_{u,v})+1$. Further, all the results about passage times across parallelograms  (see \cite[Theorem 4.2]{BGZ21}) that we will use hold for both definitions of last passage time above. We will not mention the results separately for these two definitions of passage time. We need these variations of the definition as we will be considering different segments of paths. This will be clear from proofs later. We also need a constrained passage time. We now define this. For $G \subset \mathbb{Z}^2$ and $u,v \in G$ consider 
\begin{displaymath}
T_{u,v}^G=\max_{\gamma:u \rightarrow v, \gamma \setminus\{u,v\} \subset G} \ell(\gamma);
\end{displaymath}
\begin{displaymath}
\widetilde{T_{u,v}^G}=T_{u,v}^G-\mathbb{E}(T_{u,v});
\end{displaymath}
\begin{displaymath}
\underline{T_{u,v}^G}=T_{u,v}^G+\tau_{u};
\end{displaymath}
\begin{displaymath}
\widetilde{\underline{T_{u,v}^G}}=\underline{T_{u,v}^G}-\mathbb{E}(\underline{T_{u,v}}).
\end{displaymath}
For $u,v \in \mathbb{Z}^2$ and $\gamma:u \rightarrow v$ we define the centred last passage time of a path and a variant of it.
\begin{displaymath}
\underline{\ell(\gamma)}=\ell(\gamma)+\tau_{u};
\end{displaymath}
\begin{displaymath}
\widetilde{\ell(\gamma)}=\ell(\gamma)-\mathbb{E}(T_{u,v});
\end{displaymath}
\begin{displaymath}
\widetilde{\underline{\ell(\gamma)}}=\underline{\ell(\gamma)}-\mathbb{E}(\underline{T_{u,v}}).
\end{displaymath}
\begin{comment}
Here is an important inequality regarding expected last passage time which we will use frequently (see \cite[Theorem 2]{LR10}).
For each $\delta >0$, there exists a positive constant $C_2$ (depending only on $\delta$) such that for all $m,n$ sufficiently large with $\delta < \frac{m}{n} < \delta^{-1}$ we have
\begin{equation}
\label{expected_passage_time_estimate}
      |\mathbb{E}T_{0,(m,n)}-(\sqrt{m}+\sqrt{n})^2| \leq C_2n^{1/3}.
\end{equation}
$(\sqrt{m}+\sqrt{n})^2$ will be called the \textit{time constant}. The above inequality also holds for passage times when defined including the initial vertex.
\end{comment}

We can now state our main results. Let $\epsilon>0$ be fixed and let $\alpha \in (\epsilon, \frac{\pi}{2}-\epsilon)$ (henceforth this $\epsilon$ will always be fixed and all the constants in the statement of the results will only depend on this $\epsilon$). Let $\mathcal{T}_{\alpha}$ denote the almost sure one ended geodesic tree on $\mathbb{Z}^2$ formed by the semi-infinite geodesics in direction $\alpha$. %(\textcolor{red}{define the sub-tree, its depth and volume more formally}.) 
For a fixed vertex $u \in \mathbb{Z}^2$, let us consider the sub-tree which consists all the vertices $v$ (including $u$) such that $u \in \Gamma^{\alpha}_v$. This will called the \emph{backward sub-tree rooted at $u$} and will be denoted by $\mathcal{T}_{\alpha,u}^{\downarrow}$. Next we define the depth of this sub-tree.
\begin{displaymath}
    D^{\alpha}_{u}:=\max \{\phi(u)-n : n \leq \phi(u) \text{ and } \exists v \in \mathcal{L}_{n} \cap \mathcal{T}_{\alpha,u}^{\downarrow}\}.
\end{displaymath}
Let $N^{\alpha}_u$ denote the cardinality of  $\mathcal{T}_{\alpha,u}^{\downarrow}$. When $u=\boldsymbol{0}$ we will omit the subscript $u.$
\begin{comment}We will omit the superscript when dealing with the case $\alpha=\frac{\pi}{4}$
\end{comment}
%(\textcolor{red}{why make this comment here?}.) 
The following theorem is our first main result.
\begin{theorem}
    \label{first_theorem}
    For every $\epsilon>0$ there exist constant $C,c>0$ (depending only on $\epsilon$) such that for sufficiently large $n$ we have 
    \begin{enumerate}[label=(\roman*), font=\normalfont]
        \item $\frac{c}{n^{2/3}} \leq \mathbb{P}(D^\alpha \geq n) \leq \frac{C}{n^{2/3}}$,
        \item $\frac{c}{n^{2/5}} \leq \mathbb{P}(N^\alpha \geq n) \leq \frac{C}{n^{2/5}}.$
    \end{enumerate}
\end{theorem}
%(\textcolor{red}{I prefer using (i), (ii) etc. while enumerating within the theorem, but it's ok to keep it as it is})
    In stationary exponential LPP model competition interfaces have been studied in \cite{FP05,FMP06}. Using this we can talk about the backward dual geodesic tree (see Figure \ref{fig: Geodesic tree}) which has the same distribution as the original tree (up-to a rotation of $180^{\circ}$, see \cite{Pim16}). In this context the dual implication of Theorem \ref{first_theorem} is precisely the following corollary.\\
    \begin{figure}[t!]
    \includegraphics[width=8 cm]{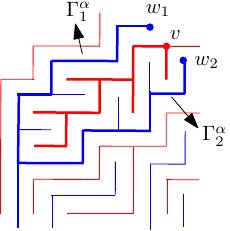}
    \caption{In the above figure the red coloured tree is the geodesic tree on the original lattice and the blue coloured tree is the dual tree on the dual lattice. Consider the backward sub-tree rooted at the red marked vertex ($v$) and consider the two geodesics ($\Gamma_1^{\alpha}, \Gamma_2^{\alpha}$) starting from the two blue marked vertices ($w_1,w_2$). If $D^{\alpha}$ denotes the depth of the backward sub-tree rooted at $v$ and $L^{\alpha}$ denotes the length of coalesce of the geodesics starting from $w_1$ and $w_2$ then $L^{\alpha}=D^{\alpha}+1$. Further, if $N^{\alpha}$ denotes the volume of the backward sub-tree rooted at $v$ and $S^{\alpha}$ denotes the number of faces sandwiched between $\Gamma_1^{\alpha}$ and $\Gamma_{2}^{\alpha}$ then $N^{\alpha}=S^{\alpha}$.} 
    \label{fig: dual_tree} 
\end{figure}
    Let $v_1=\boldsymbol{0}$ and $v_2=(-1,1)$. Consider the semi-infinite geodesics $\Gamma_{v_1}^{\alpha}$ and $\Gamma_{v_2}^{\alpha}$. Let $C^{\alpha}$ be the almost sure first coalesce point of $\Gamma_{v_1}^{\alpha}$ and $\Gamma_{v_2}^{\alpha}$ and $L^{\alpha}:=\phi(C^{\alpha})$. Further, let $S^{\alpha}$ denote the number of faces sandwiched between  $\Gamma_{v_1}^{\alpha}$ and $\Gamma_{v_2}^{\alpha}$. Then we have the following corollary. 
    \begin{corollary}
        \label{dual_implication}
        For every $\epsilon>0$ there exists constant $C,c>0$ (depending only on $\epsilon$) such that for sufficiently large $n$ we have
        \begin{enumerate}[label=(\roman*), font=\normalfont]
            \item $\frac{c}{n^{2/3}} \leq \mathbb{P}(L^{\alpha} \geq n) \leq \frac{C}{n^{2/3}}$.
            \item $\frac{c}{n^{2/5}} \leq \mathbb{P}(S^{\alpha} \geq n) \leq \frac{C}{n^{2/5}}$.
        \end{enumerate}
    \end{corollary}
    \begin{proof}[Proof of Corollary \ref{dual_implication}]
        As discussed above, using competition interfaces we obtain a dual backward geodesic tree. The random variable $L^{\alpha}$ (resp.\ $S^{\alpha}$) in the dual tree is equal to the random variable $D^{\alpha}+1$ (resp.\ $N^{\alpha}$) in the original tree (see Figure \ref{fig: dual_tree}). As the dual tree has same distribution as the original tree the corollary follows.
    \end{proof}
    \begin{comment}Theorem \ref{first_theorem}(1) implies that the probability that coalescence time of $\Gamma_{v_1}^{\alpha}$ and $\Gamma_{v_2}^{\alpha}$ is at least $n$ is like $\frac{1}{n^{2/3}}$.
    \end{comment}
    Note that Corollary \ref{dual_implication}(i) also follows from \cite[Theorem 2]{BSS17B}. Here we obtain an alternative proof of this fact from Theorem \ref{first_theorem}(i). 
    \begin{comment}Theorem \ref{first_theorem}(2) implies that the event that there are at least $n$ vertices trapped between $\Gamma_{v_1}^{\alpha}$ and $\Gamma_{v_2}^{\alpha}$ has probability of order $\frac{1}{n^{2/5}}$. 
    \end{comment}
    %(\textcolor{red}{Write the dual implication separately as a corollary})

Our second main result is an analogue of the midpoint problem in the context of semi-infinite geodesics. For $\alpha \in (\epsilon, \frac{\pi}{2}-\epsilon)$ let $-n_{\alpha}$ denote the intersection point of $y=\tan(\alpha)x$ and $\mathcal{L}_{-2n}$. 
\begin{comment}For $|k| \leq (1-\epsilon)n^{1/3}$, let $\theta_k$ denotes the angle $\tan^{-1} \left( \frac{n^{1/3}-k}{n^{1/3}+k}\right)$.
\end{comment}
Then we have following theorem.
\begin{theorem}
\label{second_theorem}
    For all $\alpha \in (\epsilon, \frac{\pi}{2}-\epsilon)$ there exist constants $C,c>0$ (depending only on $\epsilon$) such that for sufficiently large $n$ we have 
    \begin{displaymath}
    \frac{c}{n^{2/3}} \leq \mathbb{P}(\boldsymbol{0} \in \Gamma_{-n_{\alpha}}^{\alpha}) \leq \frac{C}{n^{2/3}}.
\end{displaymath}
\end{theorem}
We also answer a similar question for two intersecting geodesics. Our third main result is about the probability of a given vertex lying on two intersecting geodesics. For $|k| \leq (1-\epsilon)n^{1/3}$, let $\theta_k$ denote the angle $\tan^{-1} \left( \frac{n^{1/3}-k}{n^{1/3}+k}\right)$. Further, consider the $\epsilon$ we fixed before. We say that $k \sim n^{1/3}$ if there exists $c>0$ such that $c \leq \frac{|k|}{n^{1/3}} \leq (1- \epsilon)$ for all $n$. %\textcolor{red}{do you want the same $\epsilon$ here as before?}
We have the following theorem.
\begin{theorem}
\label{main_theorem_1}
For every $\epsilon>0$, there exist $C,c>0$ and $n_0 \in \mathbb{N}$ (depending on $\epsilon$) such that for all $n \geq n_0$ and for all $k$ such that $k \sim n^{1/3}$,
\begin{enumerate}[label=(\roman*), font=\normalfont]
\item $\frac{c}{n^{4/3}} \leq \mathbb{P}( \boldsymbol{0} \in \Gamma_{-\boldsymbol{n},\boldsymbol{n}} \cap \Gamma_{-\boldsymbol{n}_k,\boldsymbol{n}_k}) \leq \frac{C \log n}{n^{4/3}}$,
  \item  $\frac{c}{n^{4/3}} \leq \mathbb{P}( \boldsymbol{0} \in \Gamma_{-\boldsymbol{n}}^{\theta_0} \cap \Gamma_{-\boldsymbol{n}_k}^{\theta_k}) \leq \frac{C \log n} {n^{4/3}}$.
\end{enumerate}
\end{theorem}
% \begin{theorem}
% \label{main_theorem_2}
% In the above setup there exist $C >0$ and $n_0 \in \mathbb{N}$ (depending on $\epsilon$) such that for all $n \geq n_0$
% \begin{displaymath}
    
% \end{displaymath}
% \end{theorem}
    To aid in parsing the above result, we recall that $\boldsymbol{0}$ is the intersection point of the straight lines joining $-\boldsymbol{n}$ to $\boldsymbol{n}$ and $-\boldsymbol{n}_k$ to $\boldsymbol{n}_k$. Analogously, $\boldsymbol{0}$ is the intersection point of the straight lines starting from $-\boldsymbol{n}$ (resp. $-\boldsymbol{n}_k$) in the direction $\theta_0$ (resp. $\theta_k$). Also, note that in the above, %when $k\sim n^{1/3}$, then % 
    the probability bounds yield $n^{-4/3+o(1)}=n^{-2/3-2/3+o(1)}$, and this suggests that geodesics in macroscopically different directions behave approximately independently. Theorem \ref{main_theorem_1}(ii) has the following implication for geodesic trees. Let us consider two almost sure geodesic trees $\mathcal{T}_{\alpha}$ and $\mathcal{T}_{\beta}$ for $\alpha \neq \beta$. Then essentially the same proof as Theorem \ref{main_theorem_1}(ii) can be carried out to conclude the following.
    \begin{displaymath}
        \frac{c}{n^{4/3}} \leq \mathbb{P}\left( \{D^{\alpha} \geq n \} \cap \{D^{\beta} \geq n \}\right) \leq \frac{C \log n}{n^{4/3}}.
    \end{displaymath}%In the corresponding lower bounds we are currently unable to obtain the expected $k^{-2+o_k(1)}n^{-2/3}$ lower bound for a general value of $k$. However, in the special case $k\sim n^{1/3}$, we have the optimal lower bound and this corresponds to the interaction between geodesics in macroscopically different directions.
    \subsection{Outline of the proofs}We shall now describe briefly the main ideas that go into the proofs of the main results. All the main results of this paper are obtained by suitable averaging arguments, together with several key auxiliary results that are of independent interest.   
    
\subsubsection{Bounding the sub-tree depth and sub-tree size}
    We start with the upper bounds in Theorem \ref{first_theorem}. We will outline the $\frac{\pi}{4}$ direction case for both the upper and lower bounds. We will prove the bounds for general direction in Section \ref{first_theorem_proof}. We describe the averaging argument in some detail, this is used numerous times throughout the article. First we consider the sub-tree depth upper bound.
    Instead of asking the depth of the sub-tree rooted at origin, we can move $n^{2/3}$ distance in anti-diagonal direction on $\mathcal{L}_{0}$ around origin and consider all the sub-trees rooted on this line segment (call it $V$). By translation invariance, each such sub-tree has the same distribution, and hence it suffices to show that the expected number of vertices on $V$ such that the sub-tree rooted there has depth at least $2n$ (denote this random variable by $\widetilde{D}$) is uniformly bounded away from $0$ and $\infty$.
    
\noindent
\textbf{Upper bounds:} We show that $\mathbb{P}(\widetilde{D}\ge \ell)$ decays stretched exponentially in $\ell$. The idea is that for $\{\widetilde{D} \geq \ell\}$ to hold for large $\ell$, one of the two unlikely events need to happen: (a) there has to be a point $v$ in $\mathcal{L}_{-2n}$ with spatial co-ordinate $\gg n^{2/3}$ in absolute value such that the geodesic started at $v$ contributed to $\widetilde{D}$, which means that this geodesic has large transversal fluctuation, or (b) There has to be large number of geodesics started at an $O(n^{2/3})$ length interval on $\mathcal{L}_{n}$ that remain disjoint until $\mathcal{L}_0$. Both of these events have small probability if $\ell$ is large (see Figure \ref{fig: subtree_depth_upper_bound}). 
    \begin{figure}[t!]
    \includegraphics[width=8 cm]{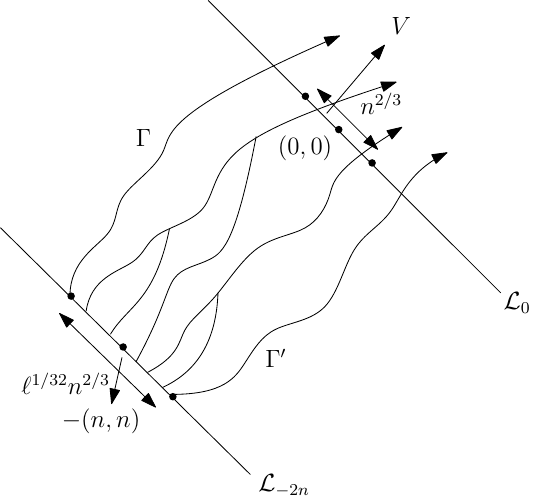}
    \caption{In the figure a semi-infinite geodesic starting from $\mathcal{L}_{-2n}$ in the $\frac{\pi}{4}$ direction can intersect the $n^{2/3}$ line segment $V$ on $\mathcal{L}_{0}$ in two ways. Either the semi-infinite geodesic comes from outside the $\ell^{1/32}n^{2/3}$ line segment on $\mathcal{L}_{-2n},$ which has a small probability due to large transversal fluctuation. For the complement event, we have any semi-infinite geodesic starting from the $\ell^{1/32}n^{2/3}$ line segment can intersect $V$ at at most $\ell$ many vertices with high probability due to Proposition \ref{coalescence_theorem}. As outlined, applying an averaging argument and combining these we get the upper bound for sub-tree depth.} 
    \label{fig: subtree_depth_upper_bound} 
\end{figure}
    
    To make this formal we need two technical results. First, to control the transversal fluctuation of semi-infinite geodesics we have the following proposition which is proved in Section \ref{first_theorem_proof}.

 \begin{comment}Let $\epsilon > 0$ and $|k| < (1-\epsilon)n^{1/3}$. If $\Gamma_n$ is the unique geodesic from $\boldsymbol{0}$ to $\boldsymbol{n_k}$ and $L$ is the straight line joining $\boldsymbol{0}$ to $\boldsymbol{n_k}$, let $\Gamma_n(T)$ denotes the (random) intersection point of $\Gamma_n$ and $\mathcal{L}_{T}$ and $v(T)$ denotes the intersection point of $L$ and $\mathcal{L}_{T}$. We have the following proposition. 
 \end{comment}
 %(\textcolor{red}{why don't we state the result in terms of the infinite geodesics?})
 \begin{comment}
\begin{proposition}
\label{transversal_fluctuation}
In the above setup there exist $C,c>0$ such that for all $T>0,\ell>0, n \geq 1$ the following hold.
       \begin{enumerate}[label=(\roman*)]
           \item  $\mathbb{P}(|\psi(\Gamma_n(T))-\psi(v(T))|\ge  \ell T^{2/3}) \le Ce^{-c \ell^3}$,
           \item $\mathbb{P}( \sup \{ |\psi(\Gamma_n(t))-\psi(v(t))|: 0 \leq t \leq T)\} \geq \ell T^{2/3}) \leq Ce^{-c\ell^3}$.
       \end{enumerate}
\end{proposition} 
\end{comment}
Let $\Gamma^\alpha$ (resp.\ $\mathcal{J}^\alpha$) denote the semi-infinite geodesic (resp.\ the straight line) in the direction $\alpha$ starting from $\boldsymbol{0}$ and $\Gamma^{\alpha}(T)$ (resp.\ $\mathcal{J}^{\alpha}(T)$) denote the intersection points of $\Gamma^\alpha$ (resp.\ $\mathcal{J}^\alpha$) with $\mathcal{L}_{T}.$ We have the following proposition.
\begin{proposition}
\label{transversal_fluctuation_of_semi_infinite_geodesic}
For $\epsilon>0$ there exist $C_1,c_1>0$ such that for all $T >0, \ell >0, n \geq 1$ and for all $\alpha \in (\epsilon, \frac{\pi}{2}-\epsilon)$ we have
\begin{enumerate}[label=(\roman*), font=\normalfont]
    \item 
            $\mathbb{P}(|\psi(\Gamma^{\alpha}(T))-\psi(\mathcal{J}^{\alpha}(T))|\ge  \ell T^{2/3}) \le C_1e^{-c_1 \ell^3}$,
            \item 
            $\mathbb{P}( \sup \{ |\psi(\Gamma^{\alpha}(t))-\psi(\mathcal{J}^{\alpha}(t))|: 0 \leq t \leq T)\} \geq \ell T^{2/3}) \leq C_1e^{-c_1\ell^3}$.
            \end{enumerate}
\end{proposition}
    Similar estimates have been proved for transversal fluctuation for finite geodesics and semi-infinite geodesic in $\frac{\pi}{4}$ direction (for example see \cite{BGZ21,BSS17B,em-jan-sep_opt_order_exit}). A finite variant of the above proposition will be proved in Section \ref{first_theorem_proof} and Proposition \ref{transversal_fluctuation_of_semi_infinite_geodesic} will follow from that.
To control the probability of the second event, we consider the event that there are $\ell$ many vertices on an $\ell^{1/32}n^{2/3}$ line segment that have semi-infinite geodesics going through $V$ and do not coalesce up-to time $n$. We have the following general result. Let $L_n$ (resp. $L_n^*$) denote the line segment on $\mathcal{L}_0$ (resp. $\mathcal{L}_{2n}$) of length $2\ell^{1/32}n^{2/3}$ with midpoint $\boldsymbol{0}$ (resp. $\boldsymbol{n}_k$). For $u,u' \in L_n$ and $v,v' \in L_n^*$ we say that $(u,v) \sim (u',v')$ if the geodesics $\Gamma_{u,v}$ and $\Gamma_{u',v'}$ coincide between the lines $\mathcal{L}_{n/3}$ and $\mathcal{L}_{2n/3}$. It is easy to see that $\sim$ is an equivalence relation. Let $M_n^k$ denote the number of equivalence classes.
\begin{proposition}
\label{coalescence_theorem}
For $\psi < 1$ there exist $C,c > 0$ such that for all $k$ with $|k|+\ell^{1/32} < \psi n^{1/3},$ all $\ell < n^{0.01}$ sufficiently large and all $n \in \mathbb{N}$ sufficiently large we have 
\begin{equation}
\label{tail_bound_for_equivalence_class}
    \mathbb{P}(M_n^k\geq \ell) \leq Ce^{-c\ell^{1/128}}.
\end{equation}
\end{proposition}
This is a generalisation of \cite[Theorem 3.10]{BHS21} and will be proved in Section \ref{first_theorem_proof}. Note that, the only difference between \cite[Theorem 3.10]{BHS21} and Proposition \ref{coalescence_theorem} is that, here the length of the line segment is increasing with $\ell$. We need this while dealing with semi-infinite geodesics. Using Proposition \ref{transversal_fluctuation_of_semi_infinite_geodesic}, it suffices to restrict our attention to semi-infinite geodesics restricted to a parallelogram of width $\ell^{1/32}n^{2/3}$, and get the tail bound on the number of disjoint geodesics using Proposition \ref{coalescence_theorem}.
Note that from the proof of the above proposition it will be clear that the number $1/32$ is arbitrary and any small enough power of $\ell$ will work. %(\textcolor{red}{say something about how 1/32 is arbitrary}). 
Together, these propositions will complete the proof of the upper bound of the sub-tree depth.  

% Both these Propositions show that $\widetilde{D}$ has stretched exponential tail, which proves the upper bound for the tail of the depth. 

For the sub-tree size we will prove that $\mathbb{P}(N \geq n^{5/3})$ is of order $\frac{1}{n^{2/3}}.$ %\textcolor{red}{your notation $\sim$ here does not match the one from before, I suggest to change the previous one and keep this as it is} 
Note that we can break the event $\{N \geq n^{5/3}\}$ in two ways, either all the sub-tree rooted at origin has depth more than $2n$ or the depth is at most $2n$. We already have the upper bound for the first event. So, we will work only with the second event. If $\widetilde{N}$ denotes the number of vertices on $V$ that has sub-tree size at least $n^{5/3}$ and depth at most $2n$, then $\{\widetilde{N} \geq \ell\}$ can happen only when there is some vertex outside of the rectangle $\mathsf{R}:=\{v \in \mathbb{Z}^2: |\psi(v)| \leq \frac{\ell}{4} n^{2/3}, -2n \leq \phi(v) \leq 0\}$ which has semi-infinite geodesic intersecting $V$. This will imply large transversal fluctuation for certain geodesics which in turn gives the stretched exponential tail for $\widetilde{N}$ by Proposition \ref{transversal_fluctuation_of_semi_infinite_geodesic} (see Figure \ref{fig: subtree_upper_bound}). This would complete the proofs of the upper bounds in Theorem \ref{first_theorem}.
\begin{figure}[t!]
    \includegraphics[width=10 cm]{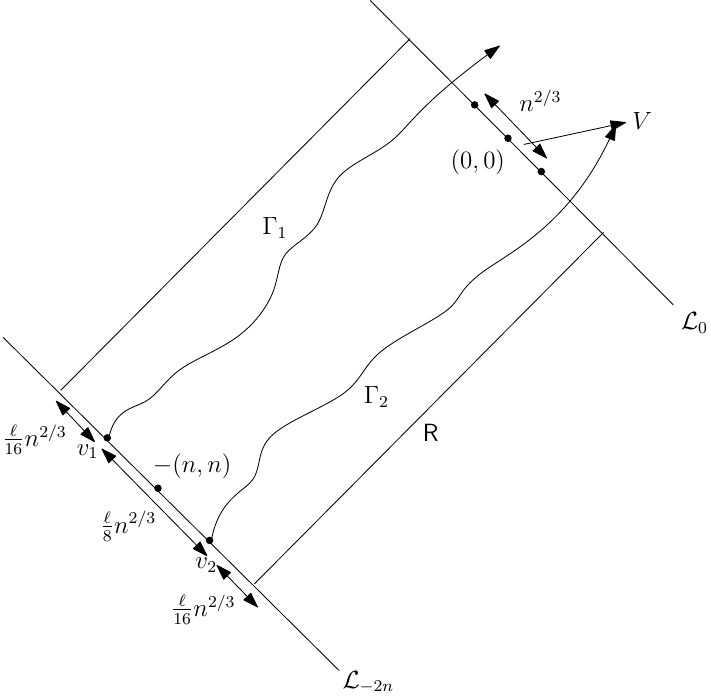}
    \caption{To prove the upper bound for sub-tree size for $\frac{\pi}{4}$ direction consider the above figure. There are at most $\frac{\ell}{2}n^{5/3}$ many vertices inside the rectangle $\mathsf{R}$. So, in order to have at least $\ell$ many vertices on $V$ having sub-tree size $n^{5/3}$ there must be a vertex outside $\mathsf{R}$ which has semi-infinite geodesic intersecting $V$. But by planarity this will imply large transversal fluctuation of either $\Gamma_1$ or $\Gamma_2$. Using this and applying averaging argument as outlined we get the upper bound for sub-tree size.}
    \label{fig: subtree_upper_bound}
\end{figure}

\paragraph{\textbf{Lower bounds:}} To obtain the lower bounds we find a uniform constant lower bound for $\mathbb{E}(\widetilde{D})$ and $\mathbb{E}(\widetilde{N})$. The idea is that, using transversal fluctuation estimate we can get a positive probability event on which geodesics don't have large transversal fluctuation. Precisely, consider two line segment $V_M$ (resp. $\widetilde{V_M}$) of length $Mn^{2/3}$ (resp. $\frac{M}{2}n^{2/3}$) on $\mathcal{L}_0$ (resp. $\mathcal{L}_{-2n}$) with midpoint $\boldsymbol{0}$ (resp. $-\boldsymbol{n}$). We can chose $M$ large enough so that using Proposition \ref{transversal_fluctuation_of_semi_infinite_geodesic}, with a positive probability, any geodesic starting from $\widetilde{V_M}$ will intersect $\mathcal{L}_{0}$ on $V_M$. So, on this event,
\begin{displaymath}
\sum_{v \in V_M}\mathbbm{1}_{\{D_v \geq 2n\}} \geq 1.
\end{displaymath}
Hence,
\begin{displaymath}
    \mathbb{P}(D \geq 2n)=\frac{\mathbb{E}(\sum_{v \in V_M}\mathbbm{1}_{\{D_v \geq 2n\}})}{Mn^{2/3}}\geq \frac{c}{n^{2/3}}.
\end{displaymath}
The lower bound for the sub-tree size is similar and we will prove it in details in Section \ref{first_theorem_proof}. 
%(\textcolor{red}{this I think could be dropped, it really is rather similar})
\begin{comment}Consider the rectangle $\widetilde{R}=\{v \in \mathbb{Z}^2: |\psi(v)| \leq n^{2/3}, -2n < \phi(v) <-n \}$. We can obtain a positive probability event so that any geodesic starting inside the rectangle will intersect $\mathcal{L}_0$ inside $V_M$ and there will be at most $\ell$ (fixed but large) many distinct vertices on $V_M$ that will carry a semi-infinite geodesic starting from $\widetilde{R}$ (using transversal fluctuation estimates and Proposition \ref{coalescence_theorem} we can chose $M$ and $\ell$ large enough so that each of these events has large probability and hence their intersection has positive probability) (see Figure \ref{fig: subtree_lower_bound}). 
\end{comment}
\begin{figure}[t!]
    \includegraphics[width=10 cm]{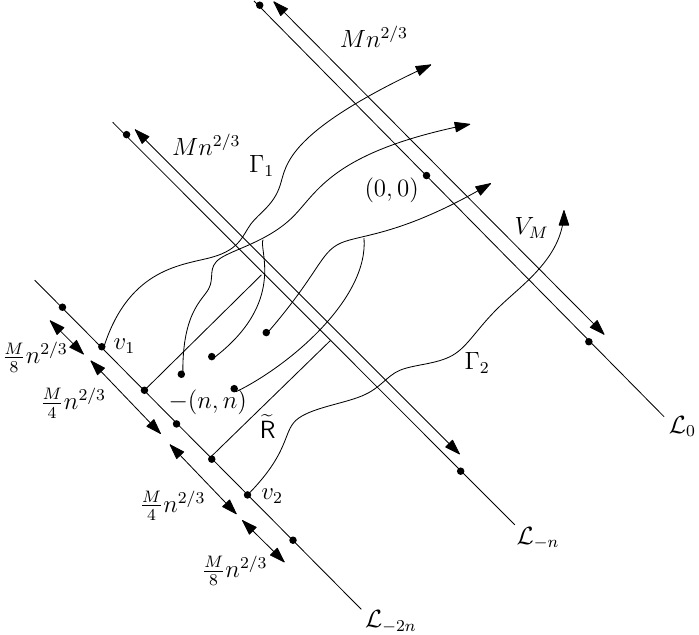}
    \caption{To prove the sub-tree depth and size lower bounds we can chose $M$ large enough so that on a large probability event the geodesics $\Gamma_1$ and $\Gamma_2$ in the above figure do not have large transversal fluctuation. On this event any semi-infinite geodesics starting from the $n^{5/3}$ vertices inside $\widetilde{\mathsf{R}}$ will be a part of sub-tree rooted at some vertex on $V_M$. Hence, on this event $V_M$ will have at least one vertex with depth $2n$. Further, we can fix $\ell$ large enough so that using Proposition \ref{coalescence_theorem} we can get another large probability event on which there are at most $\ell$ many distinct vertices on $V$ carrying a geodesic from $\widetilde{\mathsf{R}}$. So, on a positive probability event, $V_M$ will have at least one vertex with sub-tree size more than $\frac{n^{5/3}}{\ell}$. These two together  prove a constant uniform lower bound for $\mathbb{E}(\widetilde{D})$ and $\mathbb{E}(\widetilde{N})$.}
    \label{fig: subtree_lower_bound}
\end{figure}
\begin{comment}So, on this event
\begin{displaymath}
    \sum_{v \in V_M}\mathbbm{1}_{\{N_v \geq \frac{n^{5/3}}{\ell}\}} \geq 1.
\end{displaymath}
So,
\begin{displaymath}
    \mathbb{P}(N \geq \frac{n^{5/3}}{\ell})=\frac{\mathbb{E}(\sum_{v \in V_M}\mathbbm{1}_{\{N_v \geq \frac{n^{5/3}}{\ell}\}})}{Mn^{2/3}}\geq \frac{c}{n^{2/3}}.
\end{displaymath}
This proves the lower bound for the tail of the sub-tree size.
\end{comment}
\subsubsection{Probability that a semi-infinite geodesic passes through a given point}
%(\textcolor{red}{not a good heading})
The upper bound in Theorem \ref{second_theorem} is a direct consequence of the upper bound in Theorem \ref{first_theorem}(i). We will prove the lower bound using the same idea as used in \cite[Proposition 1.2]{BB21} (where a similar estimate was proved for finite geodesics), by representing the initial segment of a semi-infinite geodesic as a finite geodesic to a certain boundary condition. Note that, by the same averaging argument as before, to prove the lower bound it suffices to show that with positive probability any semi-infinite geodesic starting from a $Mn^{2/3}$ line segment around $-n_{\alpha}$ will intersect $V_M$ at a single point. In \cite[Proposition 1.2]{BB21} precisely this was shown for finite geodesics starting from $O(n^{2/3})$ distance around $\boldsymbol{0}$ and ending at $O(n^{2/3})$ distance around $\boldsymbol{n}$. %(\textcolor{red}{what does this mean?}). 
%The idea was to construct a favourable environment with positive probability on which all the finite geodesics starting from an $Mn^{2/3}$ line segment around $-n_{\alpha}$ will intersect $V_M$ at a single point. 
%We generalise this result to hold for a certain class of sufficiently regular boundary conditions on the line $\mathcal{L}_{0}$. The semi-infinite geodesic case will follow using Busemann function as the boundary condition.
We prove a similar result for semi-infinite geodesics in a general direction.

We consider the points $(-Mn^{2/3},Mn^{2/3}), (Mn^{2/3},-Mn^{2/3})$  and let $\Gamma_1^{\alpha}$ (resp. $\Gamma_2^{\alpha}$) are the semi-infinite geodesics starting from $(-Mn^{2/3},Mn^{2/3})$ (resp. $(Mn^{2/3},-Mn^{2/3})$ in the direction $\alpha$. Then we have the following proposition.
\begin{proposition}
\label{colealscence_probability_for_general_direction}
For all $\alpha \in (\epsilon, \frac{\pi}{2}-\epsilon)$, for sufficiently large $M$(depending on $\epsilon$) and sufficiently large $n$ (depending on $\epsilon$), there exists constant $c>0$ (depending only on $M$) such that 
\begin{displaymath}
    \mathbb{P}\left(\{\Gamma_1^{\alpha}(n)=\Gamma_2^{\alpha}(n)\} \cap \{|\psi(\Gamma_1^{\alpha}(n))| \leq Mn^{2/3}\} \right) \geq c.
\end{displaymath}
\end{proposition}
To prove Proposition \ref{colealscence_probability_for_general_direction} we need the notion of boundary conditions. We consider $\mathbb{Z}^2$ and we have collection of i.i.d. Exp(1) random variables $\{\tau_v\}_{v \in \mathbb{Z}^2}$ associated with each vertices of $\mathbb{Z}^2$. For $r \in \mathbb{Z}$, let us consider the collection of functions (deterministic or random)
\begin{displaymath}
\mathsf{F}=\mathsf{F}_r:\mathcal{L}_{r} \rightarrow \mathbb{R}.
\end{displaymath}
%(\textcolor{red}{why don't you just define the boundary condition on a specific anti-diagonal line?})
With these boundary conditions functions we can define point to line last passage times. 
\begin{comment}For example, Let $u \in \mathcal{L}_0$ and for some constant $M$, $A_M$ is a line segment of length $Mn^{2/3}$ on $\mathcal{L}_{2n}$ then we define the following two last passage times with boundary condition F.\\
\end{comment}
For $u \in \mathbb{Z}^2$ we make the following definitions.
\begin{comment}
\begin{displaymath}
    T_{u,A_M}:=\max_{v \in A}(T_{u,v}+\text{F}(v));
\end{displaymath}
\end{comment}
\begin{displaymath}
    T_{u,\mathcal{L}_r}^{\mathsf{F}}:=\max_{v \in \mathcal{L}_r}(T_{u,v}+\mathsf{F}_r(v)).
\end{displaymath}
%(\textcolor{red}{why don't we do it for the whole line?})
In this new setup also we can talk about point to line geodesics, (i.e., paths attaining $T^{\mathsf{F}}_{u, \mathcal{L}_r}$) and these will be denoted by $\Gamma_{u, \mathcal{L}_r}^{\mathsf{F}}$. 
\begin{comment}Instead of $\mathcal{L}_0$ we can define boundary conditions on $\mathcal{L}_r$ for any $r \in \mathbb{Z}$ similarly.
\end{comment}
\\
Initial segments of semi-infinite geodesics below $\mathcal{L}_r$ are geodesic to a certain boundary condition called Busemann functions. Busemann functions satisfy some sufficiently regular boundary conditions. Using these we will prove Proposition \ref{colealscence_probability_for_general_direction}. We define the Busemann functions now.
\paragraph{\textbf{Busemann Functions}} 
Busemann functions in a fixed direction can be intuitively thought of as differences of last passage times to infinity along that fixed direction. For each fixed $\alpha$,  on a probability 1 set $\Omega^{\alpha}$, it is defined as follows (\cite[Theorem 4.2]{Sep17}).
\begin{comment}and let $v \in \mathbb{S}^1 \cap \mathbb{R}_{>0} \times \mathbb{R}_{>0}$ is the unique element that corresponds to $\alpha$.
\end{comment}
Consider any sequence $\{v_n\}_{n \in \mathbb{N}} \subset \mathbb{Z}^2$ such that $\lim_{n \rightarrow \infty} \frac{v_n}{\|v_n\|}=\alpha$.
   The two valued  Busemann function in the direction $\alpha$ is defined as follows.
   \begin{equation}
   \label{Busemann_Function_Defintion}
       B^{\alpha}_{x,y}:=\lim_{n \rightarrow \infty}[T_{x,v_n}-T_{y,v_n}].
  \end{equation}
The single valued Busemann functions are defined on $\Omega^{\alpha}$ as follows. For a fixed direction $\alpha \in (0, \frac{\pi}{2})$ set $B^{\alpha}_{\boldsymbol{0}}:=0$ and for $(m,-m) \in \mathcal{L}_0 (m>0)$ we define
\begin{comment}intersection point of $y=(\tan \alpha )x$ and $\mathcal{L}_{2r}$ and denote it by $r_{\alpha}$.
For all $v \in \mathbb{Z}$, if $v \in \mathcal{L}_{2r}$ then 
\end{comment}
\begin{displaymath}
B^{\alpha}_{(m,-m)}:=\sum_{i=1}^m B^{\alpha}_{(i,-i),(i-1,-i+1)}.
\end{displaymath}
For $(m,-m) \in \mathcal{L}_0 (m<0)$ define
\begin{displaymath}
    B^{\alpha}_{(m,-m)}:=\sum_{i=m}^{-1} B^{\alpha}_{(i,-i),(i+1,-i-1)}.
\end{displaymath}
For a fixed $\alpha$ the collection $\{B^{\alpha}_v\}_{v \in \mathcal{L}_0}$ is an example of boundary conditions on $\mathcal{L}_0$. We can define Busemann boundary conditions on $\mathcal{L}_r$ for all $r \in \mathbb{Z}$. The only difference will be for the direction $\alpha$ the above sum will start from $r_{\alpha}$ instead of $\boldsymbol{0}$, where $r_{\alpha}$ is the intersection point of $y=\tan(\alpha)x$ and $\mathcal{L}_r$. It follows from \cite[Theorem 4.2]{Sep17} that for each fixed $\alpha$, $\{B^{\alpha}_v\}_{v \in \mathcal{L}_0}$ (resp.\ $\{B^{\alpha}_v\}_{v \in \mathcal{L}_r}$) is  two sided random walk starting from $\boldsymbol{0}$ (resp.\ $r_{\alpha}$), where the increments are distributed as $X_{\alpha}-Y_{\alpha},$ where $X_{\alpha}$ and $Y_{\alpha}$ are independent and they are distributed as follows.
\begin{equation}
\label{busemann_function_increment}
X_{\alpha} \overset{d}{=}\text{Exp} \left(\frac{\sqrt{\cos(\alpha)}}{\sqrt{\cos(\alpha)}+\sqrt{\sin(\alpha)}} \right), Y_{\alpha} \overset{d}{=}\text{Exp}\left(\frac{\sqrt{\sin(\alpha)}}{\sqrt{\sin(\alpha)}+\sqrt{\cos(\alpha)}}\right).
\end{equation}We will use the existence of Busemann functions and the above distributional identities. The reader can find the details about Busemann functions for example in \cite[Theorem 4.2]{Sep17}.
\begin{comment}Here we state some facts about them without proof. The details can be found in \cite{S20}. Let,
\begin{displaymath}
    \rho(\alpha):=\frac{\sqrt{\cos \alpha}}{\sqrt{\sin \alpha}+\sqrt{\cos \alpha}}.
\end{displaymath}
\end{comment}
%\begin{itemize}
 %   \item The collection $\{\mathcal{B}^{\alpha}_v\}_{v \in \mathbb{Z}^2}$ is independent of the collection $\{\tau_v\}_{\{v \in \mathbb{Z}^2: \phi(v) <2r\}}.$
  %  \item Let us consider a random variable $X^{\alpha }$ distributed as sum of two independent exp($\rho(\alpha))$ and $-\exp(1-\rho(\alpha))$ random variables. Then for a fixed $r \in \mathbb{Z}$ and for all $m \in \mathbb{N}$ the following hold.
  %  \begin{displaymath}
   %     \mathcal{B}^{\alpha}_{(r_\alpha-m,r_\alpha+m)}  \overset{\mathcal{D}}{=}\sum_{i=0}^{m}Y_i^{\alpha};
   % \end{displaymath}
   % \begin{displaymath}
   %      \mathcal{B}^{\alpha}_{(r_\alpha+m,r_\alpha-m)}  \overset{\mathcal{D}}{=}-\sum_{i=0}^{m}Y_i^{\alpha};
   % \end{displaymath}
   % where $Y^{\alpha}_0=0$ and for $i \geq 1, Y^{\alpha}_i$'s are i.i.d random variables distributed as $X^{\alpha}$.
%\end{itemize}
For a fixed $\alpha$, and fixed $r$, we consider the exponential last passage percolation model with boundary conditions $\{B^{\alpha}_{v}\}_{v \in \mathcal{L}_r}$. In this case
the collection $\{B^{\alpha}_{v}\}_{v \in \mathcal{L}_r}$ is independent of $\{\tau_v\}_{\{v \in \mathbb{Z}^2: \phi(v)<r\}}$. Further we have the following lemma which connects the Busmann geodeics and semi-infinite geodesics.\\
Let $\alpha \in (0, \frac{\pi}{2})$ and $r \in \mathbb{Z}$ are fixed and $\Gamma^{B_\alpha}_{u,\mathcal{L}_{r}}$ denote the geodesic under the boundary condition $\{B^{\alpha}_{v}\}_{v \in \mathcal{L}_r}$. Further, as mentioned before $\Gamma_u^{\alpha}$ denotes the semi-infinite geodesic starting from $u$ in the direction $\alpha.$ Then the following equality holds.
\begin{comment}and the definition of semi-infinite geodesics implies the following. Any semi-infinite geodesic starting from $u \in {\{v \in \mathbb{Z}^2: \phi(v)<r\}}$ in the direction $\alpha$ will intersect $\mathcal{L}_r$ at a vertex $v$ such that $T_{u,v}+B^{\alpha}_v$ has the maximum value over $\mathcal{L}_r.$ Hence, we have the following.
\end{comment}
\begin{lemma}
\label{busemann function and semi-infinite geodesic lemma}
In the above notations we have
\begin{equation}
\label{semi-infinite_geodesic_and_Busemann_Function}
    \Gamma^{B_\alpha}_{u,\mathcal{L}_{r}}=\Gamma_u^{\alpha}|_{\{v \in \mathbb{Z}^2: \phi(u) \leq \phi(v) \leq r\}}.
\end{equation}
\end{lemma}
We have the following proposition for coalescence of geodesics to the Busemann boundary condition. Same as before, we consider the points $(-Mn^{2/3},Mn^{2/3}), (Mn^{2/3},-Mn^{2/3})$  and let $\Gamma_1^{B_\alpha}$ (resp. $\Gamma_2^{B_\alpha}$) be the geodesics starting from $(-Mn^{2/3},Mn^{2/3})$ (resp. $(Mn^{2/3},-Mn^{2/3})$ to $\mathcal{L}_{2n}$ under the boundary condition $\{B^{\alpha}_v\}_{\{v \in \mathcal{L}_{2n}\}}$. Then we have the following proposition, which by Lemma \ref{busemann function and semi-infinite geodesic lemma}, immediately implies Proposition \ref{colealscence_probability_for_general_direction}.
\begin{proposition}
\label{Busemann_functions_coealsce}
For all $\alpha \in (\epsilon, \frac{\pi}{2}-\epsilon)$, for sufficiently large $M$(depending on $\epsilon$) and sufficiently large $n$ (depending on $\epsilon$), there exists constant $c>0$ (depending only on $M$) such that 
\begin{displaymath}
    \mathbb{P}\left(\{\Gamma_1^{B_\alpha}(n)=\Gamma_2^{B_\alpha}(n)\} \cap \{|\psi(\Gamma_1^{B_\alpha}(n))| \leq Mn^{2/3}\} \right) \geq c.
\end{displaymath}
\end{proposition}
We will prove Lemma \ref{busemann function and semi-infinite geodesic lemma} in Section \ref{Lower Bound for Coalescence of Semi-Infinite Geodesics}.   In Section \ref{Lower Bound for Coalescence of Semi-Infinite Geodesics} we shall use the distribution of Busemann increments to show that they satisfy certain regularity conditions with large probability (we will state and prove it more precisely in Section \ref{Lower Bound for Coalescence of Semi-Infinite Geodesics}) which will let us extend the proof of \cite[Proposition 1.2]{BB21} to establish Proposition \ref{Busemann_functions_coealsce}. A more complicated variant of the same argument will be used to prove Proposition \ref{positive_probability_for_point_to_line} below.

\subsubsection{Bounds for geodesic intersections} We now give outline of the proof of Theorem \ref{main_theorem_1}. Both upper and lower bounds will follow by averaging arguments. \paragraph{\textbf{Intersecting geodesics upper bounds:}} We will give a proof for more general upper bounds in Theorem \ref{main_theorem_1} for all $2 < |k| < (1-\epsilon)n^{1/3}$. %\textcolor{red}{we probably want $|k|>2$ here, because the bound will go to $0$ as $k\to 1$}
We state it as a proposition. 
    \begin{proposition}
        \label{general_upper_bound}
        For every $\epsilon>0$, there exist $C,c>0$ and $n_0 \in \mathbb{N}$ (depending on $\epsilon$) such that for all $n \geq n_0$ and for all $2 <|k| < (1-\epsilon)n^{1/3}$,
\begin{enumerate}[label=(\roman*), font=\normalfont]
\item $\mathbb{P}( \boldsymbol{0} \in \Gamma_{-\boldsymbol{n},\boldsymbol{n}} \cap \Gamma_{-\boldsymbol{n}_k,\boldsymbol{n}_k}) \leq \frac{C \log |k|}{k^2n^{2/3}}$,
  \item  $\mathbb{P}( \boldsymbol{0} \in \Gamma_{-\boldsymbol{n}}^{\theta_0} \cap \Gamma_{-\boldsymbol{n}_k}^{\theta_k}) \leq \frac{C \log |k|} {k^2n^{2/3}}$.
  \end{enumerate}
    \end{proposition}When $k \sim n^{1/3}$, we get the upper bounds in Theorem \ref{main_theorem_1}. Proposition \ref{general_upper_bound} proves the upper bound for intersection probability even if the geodesics are not starting macroscopic distance away from each other.\\ To prove Proposition \ref{general_upper_bound} we again invoke an averaging argument. This time instead of considering only space direction we move both in space and time direction by $n^{2/3}$ and $\frac{n}{100k}$ amount respectively. Specifically, we consider five rectangles $V,A_n,A_n^*,B_n,B_n^*$ around $\boldsymbol{0},-\boldsymbol{n},\boldsymbol{n},-\boldsymbol{n}_k, \boldsymbol{n}_k$ respectively, each of length $n^{2/3}$ in space direction and $\frac{n}{100k}$ in time direction (see Figure \ref{fig: basic_picture}).
    \begin{figure}[t!]
    \includegraphics[width=12 cm]{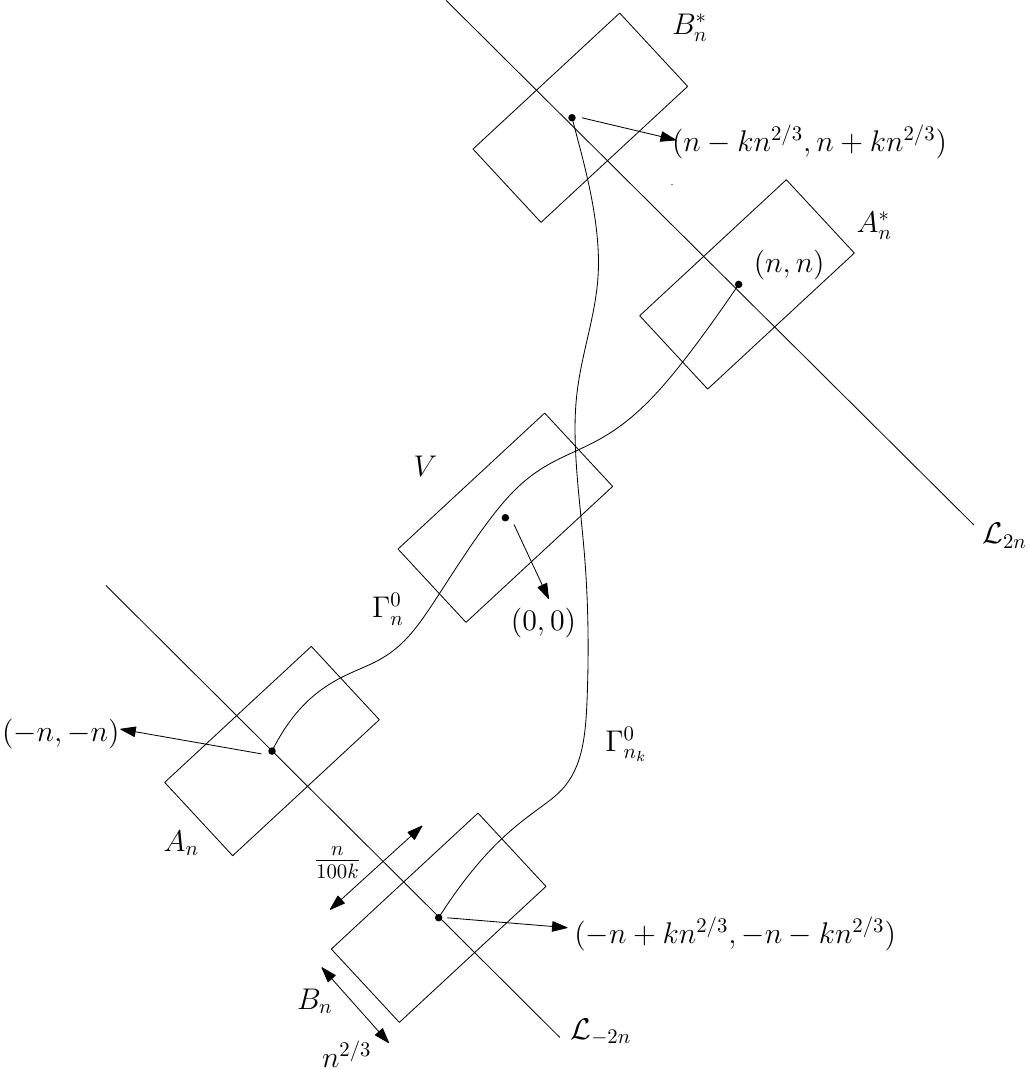}
    \caption{To prove upper bound for intersecting geodesics, we apply averaging method in both space and time direction. We consider rectangles $A_n,B_n,V,A_n^*,B_n^*$ around $-\boldsymbol{n},-\boldsymbol{n}_k,\boldsymbol{0},\boldsymbol{n},\boldsymbol{n}_k$ of equal size whose length along the time direction is $\frac{n}{100k}$ and length along space direction $n^{2/3}$.}
    \label{fig: basic_picture}
\end{figure}Now, if $N^V$ is the number of points $v \in V$ such that $\exists  w_1 \in A_n, w_2 \in A_n^*$ and $w_3 \in B_n$ and $w_4 \in B_n^*$ such that $w \in \Gamma_{w_{1,}w_{2}} \cap \Gamma_{w_{3},w_{4}}$, then the following will conclude the upper bound in Theorem \ref{main_theorem_1}(i). For all $n$ sufficiently large, we will show there exists a positive constant $C>0$ such that 
    \begin{displaymath}
    \mathbb{E}(N^V) \leq C n \log |k|/k^3.
    \end{displaymath}
    To prove this inequality we will again prove that $\frac{N^V k^3}{n \log |k|} $ has stretched exponential tail. Then arguing same way as in the proof of Theorem \ref{first_theorem} will conclude Theorem \ref{main_theorem_1}(i) upper bound. 
    \begin{comment}Let us define the following random variables.\\ Let $w_1,w_2 \in A_n$ and $w_1',w_2' \in A_n^*$. We say $(w_1,w_1')\sim (w_2,w_2')$ if the geodesics $\Gamma_{w_{1},w_{1}'}$ and $\Gamma_{w_{2},w_{2}'}$ coincides inside $V$. Let $M_n^A$ denotes the number of equivalence classes. Similarly we define $M_n^B$ for $B_n$ and $B_n^*$. For two geodesics $\Gamma_n$  and $\Gamma_n'$ starting from $A_n$ and ending at $A_n^*$ (resp. starting at $B_n$ and ending at $B_n^*$) let $I_{\Gamma_n,\Gamma_n'}$ denotes the number of intersection points of these two geodesics inside $V$.
    \end{comment} 
    Let $M_n^A$ (resp.\ $M_n^B$) denote the number of distinct vertices in $V$ so that it has a geodesic from $A_n$ (resp.\ $B_n$) to $A_n^*$ (resp.\ $B_n^*$). Further, let $I$ denote the size of intersection of two geodesics from $A_n$ (resp.\ $B_n$) to $A_n^*$ (resp.\ $B_n^*$). Then we have the following inequality.
\begin{equation}
\label{N^V_as_product}
    N^V \leq M_n^A M_n^B \max I_{\Gamma_n,\Gamma_n'}.
\end{equation} 
Here the maximum is taken over all pairs of geodesics described in the previous paragraph.
Hence, we have the estimate
\begin{equation}
\label{upper_bound_exponential_tail}
\mathbb{P}(N^V \geq \ell n \log |k|/k^3) \leq \mathbb{P}(M_n^A \geq \ell^{1/3})+\mathbb{P}(M_n^B \geq \ell^{1/3})+\mathbb{P}(\max I_{\Gamma_n,\Gamma_n'} \geq \ell^{1/3}n \log |k|/k^3).
\end{equation}
The first two terms on the right hand side will have stretched exponential upper bound due to Proposition \ref{coalescence_theorem}. The only difference will be Proposition \ref{coalescence_theorem} is for line to line but here we need the result for rectangle to rectangle. This can be dealt with using transversal fluctuations. For the last term on right hand side, for two intersecting geodesics large intersection size will imply large transversal fluctuation of either one of the two geodesics. We will prove these in detail in Section \ref{proof_of_upper_bound} and the upper bound will follow. The upper bound for intersecting semi-infinite geodesics reduces to the finite case using a transversal fluctuation argument.\\
From the proof of Proposition \ref{general_upper_bound} it will be clear that in Theorem \ref{main_theorem_1} there is nothing special about the vertex $-\boldsymbol{n}$. The same proof can be carried out for any two different $k_1$ and $k_2$. So, we state the more general version of Proposition \ref{general_upper_bound} without proof.
 \begin{proposition}
 \label{upper_bound_for_general_angles}
       For any $\epsilon>0$, there exists $C,c>0$ and $n_0 \in \mathbb{N}$ (depending on $\epsilon$) such that for all $n \geq n_0$ and all $k_1,k_2$ with $2< |k_1|, |k_2|, |k_1-k_2|\leq (1-\epsilon)n^{1/3}$
  \begin{enumerate}[label=(\roman*), font=\normalfont]
  \item $\mathbb{P}( \boldsymbol{0} \in \Gamma_{-\boldsymbol{n}_{k_1},\boldsymbol{n}_{k_1}} \cap \Gamma_{-\boldsymbol{n}_{k_2},\boldsymbol{n}_{k_2}}) \leq \frac{C\log |k_1-k_2|}{|k_1-k_2|^{2}n^{2/3}}$,
  \item $\mathbb{P}( \boldsymbol{0} \in \Gamma_{-\boldsymbol{n}_{k_1}}^{\theta_{k_1}}\cap \Gamma_{-\boldsymbol{n}_{k_2}}^{\theta_{k_2}}) \leq \frac{C\log |k_1-k_2|}{|k_1-k_2|^{2}n^{2/3}}$.
  \end{enumerate}
  \end{proposition}
\begin{comment}To estimate the right hand sides we will prove two Propositions in Section \ref{proof_of_upper_bound}. 
\begin{proposition}
\label{first_lemma}
In the above setup, there exist constants $C,c>0$ (depending on $\epsilon$) such that for sufficiently large $n$ and $\ell < n^{0.01}$ sufficiently large
\begin{displaymath}\mathbb{P}(M_n^A \geq \ell) \leq Ce^{-c \ell^{1/128}}.
\end{displaymath}
Similar result holds for $M_n^B$.
\end{proposition}
This will be an application of Proposition \ref{transversal_fluctuation} and Proposition \ref{coalescence_theorem}.
Applying Proposition \ref{transversal_fluctuation} we will prove the following Proposition in section \ref{proof_of_upper_bound}.
\begin{proposition}
\label{second_lemma}
In the above setup there exists constant $C,c>0$(depending on $\epsilon$) such that for sufficiently large $n$ and $\ell$ and $k$ sufficiently large  
\begin{displaymath}\mathbb{P}(\max I_{\Gamma_n,\Gamma_n'} \geq \ell n \log k/k^3) \leq Ce^{-c\ell}.
\end{displaymath}
\end{proposition}
\end{comment}
\paragraph{\textbf{Intersecting geodesics lower bounds:}} We prove the lower bound by an averaging argument again. Similar to the upper bound if we apply the averaging argument on some scaled $n^{2/3} \times \frac{n}{k}$ parallelograms then it is enough to prove on a positive probability event all geodesic starting from $A_n$ and ending at $A_n^*$ and all geodesic starting from $B_n$ and ending at $B_n^*$ will intersect at some point inside $V$.
\begin{comment}Consider the random variable $N^V$ defined above. To prove the lower bound we will prove $\mathbb{E}(N^V) \geq c$ for all large enough $n$. We will show this by showing that, on a positive probability event any geodesic starting from $A_n$ and ending at $A_n^*$ and any geodesic starting from $B_n$ and ending at $B_n^*$ will intersect at a single point inside $V$.
\end{comment}
For the finite geodesic case, this will be a consequence of Proposition \ref{positive_probability_for_point_to_line} below. To obtain such positive probability event we need the following technical proposition. It essentially says that in a $n \times n^{2/3}$ parallelogram in any fixed direction, there is a positive probability event that the passage time across the parallelogram are (on scale) small. This will be required to force geodesics to coalesce to a single path with positive probability.\\
Let $U_{m,\Delta}$ denote the parallelogram whose one pair of opposite sides are on $\mathcal{L}_0$ (resp.\ $\mathcal{L}_{2n}$) of length $\Delta n^{2/3}$ with mid points $(-mn^{2/3},mn^{2/3})$ (resp.\ $\boldsymbol{n}$). Let $U_{m,\Delta}^0$ (resp.\ $U_{m,\Delta}^{2n}$) denote $U_{m,\Delta} \cap \mathcal{L}_0$ (resp.\ $U_{m,\Delta} \cap \mathcal{L}_{2n}$). We have,
\begin{proposition}
\label{small_probability_lemma}
For each $\phi <1$, and for any $\Delta,x >1$ and for all sufficiently large $n$ there exists $C,c>0$ (depending only on $\Delta$ and $\phi$) such that for all $|m| < \phi n^{1/3}$
\begin{displaymath}
\mathbb{P}\left(\sup_{u \in U_{m, \Delta}^0, v \in U_{m, \Delta}^{2n}}\widetilde{T_{u,v}} \leq -xn^{1/3} \right) \geq Ce^{-cx^3}.
\end{displaymath}
\end{proposition}
This is a generalisation of \cite[Lemma 4.10]{BGZ21} and will be proved in Section \ref{proof_for_lower_bound}.\\ 
\begin{comment}Consider an exponential last passage percolation model with a collection of continuous random boundary conditions on the lines $\mathcal{L}_r, \{\text{W}_v\}_{v \in \mathcal{L}_r}$ which satisfies the following two assumptions.
\end{comment}
\begin{comment}
\paragraph{\textbf{Assumption 1}}
\label{assumption_1} For $r \in \mathbb{Z}$, $\{\text{W}_v\}_{v \in \mathcal{L}_{2r}}$ is independent of $\{\tau_v\}_{\{v \in \mathbb{Z}^2: \phi(v) < 2r\}}$.
\paragraph{\textbf{Assumption 2}}
\label{assumption_2}We consider the event as defined in \eqref{deterministic_function_condition} and denote it by $\mathcal{B}_M$. Then for sufficiently large $|r|$ there exist constants $\beta(M)>0$ such that 
\begin{equation}
\label{boundary_event_probability}
    \mathbb{P}(\mathcal{B}_M) \geq \beta(M).
\end{equation} 
\end{comment}
Coming back to the proof of lower bound for finite geodesics, consider two line segments $P_n,Q_n$ on $\mathcal{L}_{-2n}$ each of length $2Mn^{2/3}$ ($M$ will be fixed later) with midpoints $-\boldsymbol{n},-\boldsymbol{n}_k$ respectively. Similarly, define
$P_n^*,Q_n^*$ on $\mathcal{L}_{2n}$ each of length $2Mn^{2/3}$ with  midpoints of $P_n,Q_n,P_n^*,Q_n^*$ are $\boldsymbol{n},\boldsymbol{n}_k$ respectively. Let
$V_{M}$ denote the parallelogram $\{v \in \mathbb{Z}^2: |\phi(v)| < \frac{n}{100k}, |\psi(v)| < Mn^{2/3}\}$. Now consider the following events (see also Figure \ref{fig: proof_for_lower_bound}).
\begin{comment}Let $u_1,u_2$ (resp. $u_1',u_2'$) denote the end points of $P_n$ (resp. $P_n^*$) and $w_1,w_2$ (resp. $w_1',w_2'$) denote the end points of $Q_n$ (resp. $Q_n^*$). 
For any point $u \in \mathbb{Z}^2$ and a line segment $T_{u,A}$ attaining path will be denoted by $\Gamma_{u,A}$.
\end{comment}
\begin{itemize}
    \item $\mathcal{E}_1:=\{\Gamma_{a_1,a_1'}(t)=\Gamma_{a_2,a_2'}(t)$,  $\forall a_1,a_2 \in P_n, \forall a_1',a_2' \in P_n^*$ and $ \forall t \in [-\frac{n}{100k},\frac{n}{100k}]$\};
    \item $\mathcal{E}_2:=\{\Gamma_{b_1,b_1'}(t)=\Gamma_{b_2,b_2'}(t)$, $ \forall b_1,b_2 \in Q_n, \forall b_1',b_2' \in Q_n^*$ and $\forall t \in [-\frac{n}{100k},\frac{n}{100k}]$\};
   \item $\mathcal{E}_3:=\{ \Gamma_{a,a'} \cap \Gamma_{b,b'} \subset V_{\frac{M}{2}}, \forall a \in P_n, b \in Q_n, a' \in P_n^*, b' \in Q_n^*$\}
    \item $\mathcal{E}=\mathcal{E}_1 \cap \mathcal{E}_2 \cap \mathcal{E}_3.$
    \end{itemize}
    We will prove the following Proposition in Section \ref{proof_for_lower_bound}.
\begin{proposition}
\label{positive_probability_for_point_to_line}
    For sufficiently large $M$ (depending only on $\epsilon$) there exists a constant $c$ (depending on $M$) such that for all sufficiently large $n$ (depending on $M$) we have 
    \begin{displaymath}
        \mathbb{P}(\mathcal{E}) \geq c.
    \end{displaymath}
\end{proposition}
\begin{remark}
    \begin{comment}When we have the zero boundary condition, the proof of Proposition \ref{positive_probability_for_point_to_line} can be carried out for the finite geodesic setup in Theorem \ref{main_theorem_1}.
    \end{comment} 
    Using the above proposition we will prove the lower bound for Theorem \ref{main_theorem_1}(i) in Section \ref{proof_for_lower_bound}. Theorem \ref{main_theorem_1}(ii) lower bound will follow from the same argument by representing semi-infinite geodesics as finite geodesics to Busemann boundary conditions. To avoid repetition we shall only give a sketch.
    \begin{comment}As mentioned earlier Theorem \ref{second_theorem} lower bound follow if we assume the boundary conditions to be the single valued Busemann functions in $\frac{\pi}{4}$ direction. Because, in this case Assumption 1 and Assumption 2 are satisfied using Proposition \ref{Busemann_Function_loss}. We will not prove the lower bound in Theorem \ref{main_theorem_1}(2). But, considering two different boundary conditions simultaneously, (precisely, Busemann functions in two different directions) we can prove the lower bound. Only difficulty in this case is, although the Busemann functions in a general direction satisfies Assumption 1, it does not satisfy assumption 2. Because, in Proposition \ref{Busemann_Function_loss} the expectations are non zero if the direction is not $\frac{\pi}{4}$. But this drift term can be cancelled out by the weights of the paths below $\mathcal{L}_2n$ and then we can use the same argument to prove a similar result like Proposition \ref{positive_probability_for_point_to_line} and deduce the lower bound in the semi-infinite setting. To avoid too many notations we will not do this proof in this paper.
    \end{comment}
\end{remark}
    \begin{remark}
        As discussed earlier, although we prove the upper bounds in Theorem \ref{main_theorem_1} and Proposition \ref{general_upper_bound} for all values of $k$, we are currently unable to obtain the matching lower bound $k^{-2+o_k(1)}n^{-2/3}$ for a general value of $k$, which we expect to be true. Note that when $k$ is of constant order then we already have the expected lower bounds (see \cite[Theorem 1.1]{Z20}, \cite[Proposition 1.2]{BB21}). In this article we only prove this when $k \sim n^{1/3}.$ When the geodesics are starting from $\mathcal{O}(kn^{2/3})$ distance away from each other, considering transversal fluctuations, we expect that the typical time they will spend together is $\mathcal{O}(\frac{n}{k^3})$, perhaps with a logarithmic factor arising due to the randomness of the first intersection point (the upper bound to this effect is shown in Lemma \ref{intersection size lemma}). Our current lower bound, however, only ensures that the geodesics will intersect in a given region with positive probability, so we get the trivial lower bound of $1$ for the size of the intersection. When $k \sim n^{1/3}$, $\frac{n}{k^3}$ is of order one and we get matching bounds. 
        %(\textcolor{red}{say something about why this is expected})
    \end{remark}
    \begin{remark}
        Theorem \ref{main_theorem_1} is about intersection of geodesics. In a forthcoming article involving two of the authors of the current manuscript, a model for road layouts is considered. In certain terrains roads match first passage geodesics remarkably well. As first passage percolation is expected to be in the KPZ universality class under mild assumptions, and exponential LPP has been proved to belong to this class, LPP predictions will be compared to actual road layouts. As roads do not have a common specific direction they follow, LPP geodesics in various directions naturally arises in this context. Their intersection properties as proved in Proposition \ref{upper_bound_for_general_angles} will be of fundamental importance in understanding the statistics of road layouts under the modeling assumptions.
        %\red{to write}
    \end{remark}
    \paragraph{\textbf{Organisation of this paper}} The remainder of the paper is organised as follows. Section \ref{first_theorem_proof} is devoted to the proof of Theorem \ref{first_theorem}, we shall also establish Proposition \ref{transversal_fluctuation_of_semi_infinite_geodesic} and  Proposition \ref{coalescence_theorem} en route. As mentioned already, this will immediately imply the upper bound in Theorem \ref{second_theorem}. The upper bound for Theorem \ref{main_theorem_1} is presented in Section \ref{proof_of_upper_bound} where we prove the more general Proposition \ref{general_upper_bound}. In Section \ref{proof_for_lower_bound} we prove Proposition \ref{small_probability_lemma} and Proposition \ref{positive_probability_for_point_to_line} and use Proposition \ref{positive_probability_for_point_to_line} to establish the lower bound in Theorem \ref{main_theorem_1}(i). The remaining lower bounds for Theorem \ref{second_theorem} and Theorem \ref{main_theorem_1}(ii) are shown in Section \ref{Lower Bound for Coalescence of Semi-Infinite Geodesics}. Since, these are simpler variants of the proof of Proposition \ref{positive_probability_for_point_to_line}, to avoid repitition we do not provide all the details and restrict ourselves to giving an elaborate sketch.

    % We will prove Proposition \ref{transversal_fluctuation_of_semi_infinite_geodesic} and  Proposition \ref{coalescence_theorem} in Section \ref{first_theorem_proof}. Using these two propositions we will prove Theorem \ref{first_theorem} in Section \ref{first_theorem_proof}. Proof of Proposition \ref{Busemann_functions_coealsce} will be outlined in Section \ref{Lower Bound for Coalescence of Semi-Infinite Geodesics}. Lemma \ref{busemann function and semi-infinite geodesic lemma} will be proved in Section \ref{Lower Bound for Coalescence of Semi-Infinite Geodesics}. Proposition \ref{colealscence_probability_for_general_direction} will follow from Proposition \ref{Busemann_functions_coealsce} and Lemma \ref{busemann function and semi-infinite geodesic lemma}. We will prove Proposition \ref{general_upper_bound} (and hence Theorem \ref{main_theorem_1} upper bounds) in Section \ref{proof_of_upper_bound}. We will prove Proposition \ref{small_probability_lemma} and Proposition \ref{positive_probability_for_point_to_line} in Section \ref{proof_for_lower_bound}. In Section \ref{proof_for_lower_bound} we prove Theorem \ref{main_theorem_1}(i) lower bound using Proposition \ref{positive_probability_for_point_to_line}. We will outline the proof of lower bounds in Theorem \ref{second_theorem} and Theorem \ref{main_theorem_1}(ii) in Section \ref{Lower Bound for Coalescence of Semi-Infinite Geodesics}.
    \paragraph{\textbf{Acknowledgements}}The authors are grateful to David Harper for the simulation in Figure \ref{fig: Geodesic tree}. We also thank Manan Bhatia for useful discussions. MB was partially supported by the EPSRC EP/W032112/1 Standard Grant of the UK. RB was partially supported by a MATRICS grant (MTR/2021/000093) from SERB, Govt.~of India, DAE project no.~RTI4001 via ICTS, and the Infosys Foundation via the Infosys-Chandrasekharan Virtual Centre for Random Geometry of TIFR. SB was supported by scholarship from National Board for Higher Mathematics (NBHM) (ref no: 0203/13(32)/2021-R\&D-\rom{2}/13158). This project was initiated at the International Centre for Theoretical Sciences (ICTS), Bengaluru, India during the program "First-passage percolation and related models" in July 2022 (code:~ICTS/fpp-2022/7), the authors thank ICTS for the hospitality. This study did not involve any underlying data.

\section{Depth and volume of the backward sub-tree}
\label{first_theorem_proof}
\subsection{Proof of Proposition \ref{transversal_fluctuation_of_semi_infinite_geodesic}}We first prove a finite variant of Proposition \ref{transversal_fluctuation_of_semi_infinite_geodesic}.\\
    Let $\epsilon > 0$ and $|k| < (1-\epsilon)n^{1/3}$. Let $\Gamma_{n_k}$ be the unique geodesic from $\boldsymbol{0}$ to $\boldsymbol{n}_k$ (recall that $\boldsymbol{n}_k:=(n-kn^{2/3},n+kn^{2/3})$ and $\mathcal{J}_k$ be the straight line joining $\boldsymbol{0}$ to $\boldsymbol{n}_k$, let $\Gamma_{n_k}(T)$ denote the (random) intersection point of $\Gamma_{n_k}$ and $\mathcal{L}_{T}$ and $\mathcal{J}_k(T)$ denote the intersection point of $\mathcal{J}_k$ and $\mathcal{L}_{T}$. We have the following proposition. 
\begin{proposition}
\label{transversal_fluctuation}
In the above setup there exist $C_1,c_1>0$ such that for all $T \leq 2n,\ell>0, n \geq 1$ the following hold.
       \begin{enumerate}[label=(\roman*), font=\normalfont]
           \item  $\mathbb{P}(|\psi(\Gamma_{n_k}(T))-\psi(\mathcal{J}_k(T))|\ge  \ell T^{2/3}) \le C_1e^{-c_1 \ell^3}$,
           \item $\mathbb{P}( \sup \{ |\psi(\Gamma_{n_k}(t))-\psi(\mathcal{J}_k(t))|: 0 \leq t \leq T)\} \geq \ell T^{2/3}) \leq C_1e^{-c_1\ell^3}$.
       \end{enumerate}
\end{proposition} 
    Proposition \ref{transversal_fluctuation} essentially says for $0 \leq T \leq 2n$, on $\mathcal{L}_T$, a geodesic starting from $\boldsymbol{0}$ to some fixed point in $\mathcal{L}_{2n}$ (away from axes) will have large transversal fluctuation on $T^{2/3}$ scale with small probability. 
\begin{proof}[Proof of Proposition \ref{transversal_fluctuation} (i)] We use a simplified version of the idea used in \cite[Theorem 3]{BSS17B}. Let $|k| < (1 -\epsilon)n^{1/3}$ be fixed. Note that it suffices to show that for all  $|k| < (1 -\epsilon)n^{1/3}$,
 \[
  \mathbb{P}(\psi(\Gamma_{n_k}(T))-\psi(\mathcal{J}_k(T))\ge\ell T^{2/3}) \leq e^{-c\ell^3}
 \]
 because by symmetry
 \[
  \mathbb{P}(\psi(\Gamma_{n_k}(T))-\psi(\mathcal{J}_k(T))\le  -\ell T^{2/3})=\mathbb{P}(\psi(\Gamma_{n_{-k}}(T))-\psi(\mathcal{J}_{-k}(T))\ge  \ell T^{2/3}).
 \]
 Further, we can assume $T<\frac{n}{2}$. For other choices of $T$ we already have the required estimate from \cite[Proposition C.9]{BGZ21}\\
First we break the event $\mathbb{P}(\psi(\Gamma_{n_k}(T))-\psi(\mathcal{J}_k(T)) \ge  \ell T^{2/3})$ into smaller events. We do this in the following way. Let $\alpha \in (1 ,\sqrt{2})$ be a fixed constant. The reason for choosing such an $\alpha$ will be clear soon. For $0 \leq j \leq \log_2(\frac{n}{T})-1$, we define the following events.
\begin{itemize}
\item $\mathcal{B}_j':=\{\psi(\Gamma_{n_k}(2^jT))-\psi(\mathcal{J}_k(2^jT)) \geq \ell ((2 \alpha)^jT)^{2/3} \}$,
\item $\mathcal{B}_j'':=\{\psi(\Gamma_{n_k}(2^{j+1}T))-\psi(\mathcal{J}_k(2^{j+1}T)) \leq \ell((2 \alpha)^{j+1}T)^{2/3}\}$,
\item $\mathcal{B}_j:=\mathcal{B}_j' \cap \mathcal{B}_j'',$
\item $\mathcal{B}:=\bigcup \mathcal{B}_j$.
\begin{comment}
   \mathcal{B}_j:=\{\psi(\Gamma_n(2^jT))-\psi(v(2^jT)) \geq \ell ((2 \alpha)^jT)^{2/3} \} \cap \{\psi(\Gamma_n(2^{j+1}T))-\psi(v(2^{j+1}T)) \leq \ell((2 \alpha)^{j+1}T)^{2/3}\}.
   \end{comment}
\end{itemize}
\begin{comment}Then we have 
\begin{displaymath}
\mathbb{P}(\psi(\Gamma_n(T))-\psi(v(T))\ge  \ell T^{2/3})= \mathbb{P}(\{\psi(\Gamma_n(T))-\psi((v(T))\ge  \ell T^{2/3}\} \cap \mathcal{B})+ \\ 
\mathbb{P}(\{\psi(\Gamma_n(T))-\psi(v(T))\ge  \ell T^{2/3}\} \cap \mathcal{B}^{c}).
\end{displaymath}
\end{comment}
We bound probability of the event $\{\psi(\Gamma_{n_k}(T))-\psi((\mathcal{J}_k(T))\ge  \ell T^{2/3}\} \cap \mathcal{B}^c$. Let $n \in \mathbb{N}$. Then there exists $j_0 \geq 0$ such that $2^{j_0}T \leq n < 2^{j_0+1}T$. Note that when both the event $\mathcal{B}^c$ and  $\{\psi(\Gamma_{n_k}(T))-\psi(\mathcal{J}_k(T))\ge  \ell T^{2/3}\}$ happen then for all $0 \leq j \leq \log_2(\frac{n}{T})-1, \psi(\Gamma_{n_k}(2^{j+1}T))-\psi(\mathcal{J}_k(2^{j+1}T))\ge  \ell ((2 \alpha)^{j+1}T)^{2/3}$. Hence, we have the following 
\begin{displaymath}
\mathcal{B}^{c} \cap \{\psi(\Gamma_{n_k}(T))-\psi(\mathcal{J}_k(T))\ge  \ell T^{2/3}\} \subset \{\psi(\Gamma_{n_k}(2^{j_0}T))-\psi(\mathcal{J}_k(2^{j_0}T))\ge  \ell ((2 \alpha)^{j_0}T)^{2/3}\}.
\end{displaymath}
 So in the setup of \cite[Proposition C.9]{BGZ21} if we take $r=n, \phi=\frac{\ell \alpha^{2j_0/3}}{2^{2/3}}$, the right hand side above is a subset of the event that the geodesic from $\boldsymbol{0}$ to $\boldsymbol{n}_k$ goes out of the strip of width $\phi n^{2/3}$ around the straight line joining $\boldsymbol{0}$ and $\boldsymbol{n}_k$ (here we use the condition $n < 2^{j_0+1}T$).
So using \cite[Proposition C.9]{BGZ21}, we get constants $C,c>0$ (note that these constants depend only on $\epsilon$) such that 
\begin{displaymath}
    \mathbb{P}(\mathcal{B}^c \cap \{\psi(\Gamma_{n_k}(T))-\psi(\mathcal{J}_k(T))\ge  \ell T^{2/3}\}) \leq Ce^{-c \phi^3} \leq Ce^{-c \ell^3}.
\end{displaymath}
The last inequality is because $\alpha >1$. So, it only remains to find an upper bound for $\mathbb{P}(\mathcal{B})$. 
\begin{comment}\{\psi(\Gamma_{n_k}(T))-\psi((\mathcal{J}_k(T))\ge  \ell T^{2/3}\} \cap 
\end{comment}
If we can show 
\begin{displaymath}
\mathbb{P}(\mathcal{B}_j) \leq e^{-c\ell^3 \alpha^{2j}}
\end{displaymath}
then summing over $j$ will give the result.\\
To this end we observe that as we have chosen $|k| < (1-\epsilon)n^{1/3}$, $|\psi((0,T))-\psi(\mathcal{J}_k(T))|=m T$ for some $0 < \epsilon < m <1-\epsilon$ (see Figure \ref{fig: Transversal_fluctuation}). Also, if $\ell > mT^{1/3}$ then $\mathbb{P}(\psi(\Gamma_{n_k}(T))-\psi(\mathcal{J}_k(T)) \ge  \ell T^{2/3})=0$. So, we consider only $\ell \leq mT^{1/3}$. For $j$ as before consider the deterministic points $w(2^jT)$ on $\mathcal{L}_{2^jT}$ such that $\psi(\mathcal{J}_k(2^jT))-\psi(w(2^{j}T))=\ell((2\alpha)^{j}T)^{2/3}$. Observe that 
\begin{equation}
\label{condition_checking}
    \ell((2\alpha)^{j+1}T)^{2/3} < \left(\frac{\alpha^{2(j+1)/3}}{2^{{(j+1)/3}}}\right)m2^{j+1}T \leq m2^{j+1}T.
\end{equation}
It follows from \eqref{condition_checking} that, as $m$ is bounded away from $0$ and $1$, $w(2^{j+1}T)$ is uniformly bounded away from the axes. So, if we consider the geodesic between $\boldsymbol{0}$ and $w(2^{j+1}T)$ we are in the setup of \cite[Proposition C.9]{BGZ21}. Also note that the distance between $w(2^jT)$ and the intersection point of the straight line joining $\boldsymbol{0}$ and $w(2^{j+1}T)$ is $\ell((2\alpha)^jT)^{2/3}(1-\frac{\alpha^{2/3}}{2^{1/3}})$(see Figure \ref{fig: Transversal_fluctuation}).\\
\begin{figure}[t!]
    \includegraphics[width=15 cm]{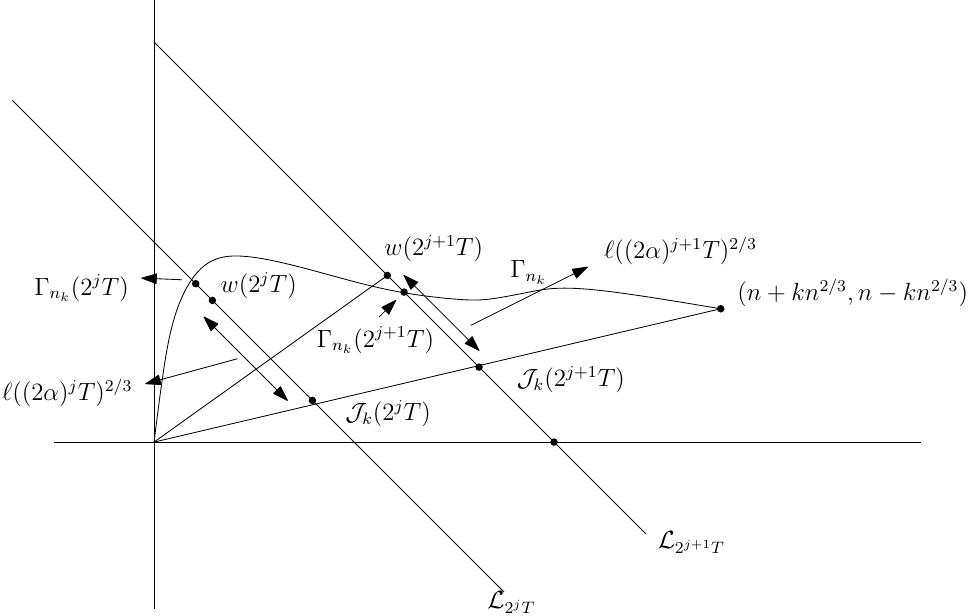}
    \caption{To prove Proposition \ref{transversal_fluctuation} we consider the lines $\mathcal{L}_{2^jT}$ and deterministic points $w(2^{j}T)$ on the lines. The idea of the proof is to show that if $\Gamma_{n_k}$ has large transversal fluctuation on $\mathcal{L}_T$, then with high probability there exists $j \geq 0$ such that $\Gamma_{\boldsymbol{0},w(2^{j+1}T)}$ also has large transversal fluctuation on $\mathcal{L}_{2^{j}T}$. Then applying the global transversal fluctuation result (see \cite[Proposition C.9]{BGZ21}) and taking a union bound proves the result.} 
    \label{fig: Transversal_fluctuation} 
\end{figure}
By construction of the events $\mathcal{B}_j$, we can apply planarity 
\begin{comment}(if two geodesics intersect at a point then they will coalesce for some length before splitting, i.e., geodesics do not form loops as this will contradict the uniqueness of geodesic between two points)
\end{comment}
to conclude $\mathcal{B}_j$ is contained in the event that the geodesic $\Gamma_{\boldsymbol{0},w(2^{j+1}T)}$ has a transversal fluctuation larger than $\ell((2\alpha)^jT)^{2/3}(1-\alpha^{2/3}/2^{1/3})$ on $\mathcal{L}_{2^jT}$. So, applying \cite[Proposition C.9]{BGZ21} we have for sufficiently large $j$
\begin{displaymath}
\mathbb{P}(\mathcal{B}_j) \leq e^{-c\ell^3\alpha^{2j}}.
\end{displaymath}
We take a sum of the right hand side over all $j$ to conclude
\begin{displaymath}
    \mathbb{P}(\mathcal{B}) \leq Ce^{-c \ell^3}.
\end{displaymath}
combining all the arguments above we have we have the required upper bound.
\end{proof}
\begin{proof}[Proof of Proposition \ref{transversal_fluctuation}(ii)] The proof is an application of first part and \cite[Proposition C.9]{BGZ21}. We consider  $\mathbb{P}( \sup \{ \psi(\Gamma_{n_k}(t))-\psi(\mathcal{J}_k(t)): 0 \leq t \leq T)\} \geq \ell T^{2/3})$. As before if $\ell \geq mT^{1/3}$ then this probability is 0. So, we need to consider only the case where $\ell < mT^{1/3}$. Let $\mathcal{B}'$ denote the event $\{\psi(\Gamma_{n_k}(T))-\psi(\mathcal{J}_k(T)) \geq \frac{\ell}{2}T^{2/3}\}$. Then by the first part 
\begin{displaymath}
   \mathbb{P}(\mathcal{B}') \leq e^{-c\ell^3}.
\end{displaymath}
We consider the deterministic point $w'(T)$ on $\mathcal{L}_T$ such that $\psi(\mathcal{J}_k(T))-\psi(w'(T))=\frac{\ell}{2}(T)^{2/3}$. Then the event $(\mathcal{B}')^{c} \cap \{\sup \{ \psi(\Gamma_{n_k}(t))-\psi(\mathcal{J}_k(t)): 0 \leq t \leq T)\} \geq \ell T^{2/3}\}$ is contained in the event that the geodesic $\Gamma_{\boldsymbol{0},w'(T)}$ goes out of the strip of width $\frac{\ell}{2}T^{2/3}$ around the straight line joining $\boldsymbol{0}$ and $w'(T)$ (by planarity of geodesics). By \cite[Proposition C.9]{BGZ21} this event has probability less than $e^{-c\ell^3}$ (observe that to apply \cite[Proposition C.9]{BGZ21} we need $\ell < mT^{1/3}$). Taking a union bound we get the upper bound $\mathbb{P}(\sup \{ \psi(\Gamma_{n_k}(t))-\psi(\mathcal{J}_k(t)): 0 \leq t \leq T)\} \geq \ell T^{2/3}) \leq Ce^{-c \ell^3}$. As argued before this completes the proof of (ii).
\end{proof}
A consequence of Proposition \ref{transversal_fluctuation} is Proposition \ref{transversal_fluctuation_of_semi_infinite_geodesic}. We first prove a lemma. Recall the notations in Proposition \ref{transversal_fluctuation_of_semi_infinite_geodesic}.
\begin{lemma}
\label{finite_approximation_lemma}
 Consider an increasing sequence of points $u_n$ on the line $L^{\alpha}$ diverging to $\infty$ and define $\Gamma_n^{\alpha}=\Gamma_{\boldsymbol{0},u_n}$. Then $\Gamma^{\alpha}$ is a (random) subsequential limit of $\Gamma_n^{\alpha}$. Further, for fixed $T>0,$ there exists (random) $n_0 \in \mathbb{N}$ such that for all $n \geq n_0$ and for all $0 \leq t \leq T$ we have $\Gamma(t)=\Gamma_n(t).$
\end{lemma}
\begin{proof}[Proof of Lemma \ref{finite_approximation_lemma}]
\begin{comment}
The first part of the lemma is clear. By standard construction of semi-infinite geodesic as a subsequential limit of finite geodesics (see for example \cite[Section 4.4]{ADH15}), $\Gamma^{\alpha}$ is (random) subsequential limit of $\Gamma_n^{\alpha}$.\\ 
\end{comment}
We will not prove the first part of the lemma. If we consider the finite geodesics $\Gamma_{n}$ then it is known that (\cite{DN01, FMP06}) there exists a random subsequence $\{n_k(\omega)\}$ such that 
\begin{displaymath}
    \Gamma^{\alpha}(\omega)=\lim_{k \rightarrow \infty} \Gamma^{\alpha}_{n_k(\omega)}(\omega).
\end{displaymath}
We prove the second part by contradiction. If the lemma is not true then for a fixed realization there exists subsequence $\{n_k\}$ such that $\Gamma_{n_k}^{\alpha}$ and $\Gamma^{\alpha}$ do not coincide below $\mathcal{L}_T$. Now as there are only finitely many edges below $\mathcal{L}_T$, there exists $t \in (0,T]$ such that for infinitely many $n_k$'s $\Gamma_{n_k}^{\alpha}(t) \neq \Gamma^{\alpha}(t)$. Let $t_0$ denote the infimum of all such $t$. Then we have a subsequence $\{n_{k_\ell}\}$ of $\{n_k\}$ such that $\Gamma^{\alpha}_{n_{k_\ell}}$ coincides with $\Gamma^{\alpha}$ up to $\mathcal{L}_{[t_0]-1}$ and between $\mathcal{L}_{[t_0]-1}$ and $\mathcal{L}_{[t_0]}$ there exists an edge on $\Gamma^{\alpha}$ such that $\Gamma^{\alpha}_{n_{k_\ell}}$ do not have that edge for all $\ell \geq 1$. So, infinitely many of $\Gamma^{\alpha}_{n_{k_\ell}}$ share a same edge between $\mathcal{L}_{[t_0]-1}$ and $\mathcal{L}_{[t_0]}$ and this edge is not on $\Gamma^{\alpha}$. Now on this subsequence of $\Gamma^{\alpha}_{n_{k_\ell}}$ we can apply the same technique as used in \cite{DN01,FMP06} to get a semi-infinite geodesic $(\Gamma^{\alpha})'$ which has an edge disjoint from $\Gamma^{\alpha}$ and $(\Gamma^{\alpha})'$ be such that it starts from $\boldsymbol{0}$ in the direction of $\alpha$. This contradicts the uniqueness of semi-infinite geodesic in the direction $\alpha$. Hence the lemma is proved.
\end{proof}
\begin{proof}[Proof of Proposition \ref{transversal_fluctuation_of_semi_infinite_geodesic} (i)]Consider the points $u_n$ defined as in Lemma \ref{finite_approximation_lemma}. As mentioned in Lemma \ref{finite_approximation_lemma} we have
\begin{displaymath}
   \Gamma^{\alpha}(\omega)=\lim_{k \rightarrow \infty} \Gamma^{\alpha}_{n_k(\omega)}(\omega).
\end{displaymath}
So,
\begin{displaymath}
    \psi(\Gamma^{\alpha}(T;\omega)) \leq \limsup \psi (\Gamma_n^{\alpha}(T; \omega)).
\end{displaymath}
Hence
\begin{displaymath}
 \mathbbm{1}_{\{\psi(\Gamma^{\alpha}(T; \omega))-\psi(\mathcal{J}^{\alpha}(T)) > \ell T^{2/3}\}} \leq \limsup \mathbbm{1}_{\{\psi(\Gamma_n^{\alpha}(T;\omega))-\psi(\mathcal{J}^{\alpha}(T)) > \ell T^{2/3}\}\}}.
\end{displaymath}
Now by Lemma \ref{finite_approximation_lemma} we have the limsup in the right side is actually a limit (as by the lemma the sequence eventually becomes $\mathbbm{1}_{\{\psi(\Gamma^{\alpha}(T))-\psi(^{\alpha}(T)) > \ell T^{2/3}\}}$). So, by the dominated convergence theorem we have 
\begin{displaymath}
\mathbb{P}(\psi(\Gamma^{\alpha}(T))-\psi(\mathcal{J}^{\alpha}(T)) > \ell T^{2/3}) \leq \lim_{n \rightarrow \infty} \mathbb{P}(\psi(\Gamma_n^{\alpha}(T))-\psi(\mathcal{J}^{\alpha}(T)) > \ell T^{2/3})
\end{displaymath}
From Proposition \ref{transversal_fluctuation}, the limit on the right hand side is bounded by $e^{-c \ell^3}$. This completes the proof of (i).
\end{proof}
\begin{proof}[Proof of Proposition \ref{transversal_fluctuation_of_semi_infinite_geodesic} (ii)] The proof is same as the proof of Proposition \ref{transversal_fluctuation} (ii). So, we omit the details.
\end{proof}
\subsection{Proof of Proposition \ref{coalescence_theorem}}
The only difference between \cite[Theorem 3.10]{BHS21} and Proposition \ref{coalescence_theorem} is that in Proposition \ref{coalescence_theorem} the lengths of the line segments are varying with $\ell$. But due to \cite[Proposition 3.1]{BHS21} the same proof as \cite[Theorem 3.10]{BHS21} goes through. So, we will only give an outline of the proof of Proposition \ref{coalescence_theorem}. More details can be found in \cite[Theorem 3.10]{BHS21}.
\begin{proof}[Proof of Proposition \ref{coalescence_theorem}]
Recall that $L_n$ (resp.\ $L_n*$) are line segments of length $\ell^{1/32}n^{2/3}$ on $\mathcal{L}_0$ (resp.\ $\mathcal{L}_{2n}$) with midpoints $\boldsymbol{0}$ (resp.\ $\boldsymbol{n}_k$). Using transversal fluctuation and planarity arguments it can be seen easily that we can restrict to the event that no geodesic starting from $L_n$ and ending at $L_n^*$ go outside a strip of width $2 \ell^{1/16}n^{2/3}$ around the straight line joining $\boldsymbol{0}$ to $\boldsymbol{n}_k$. Assume there are $\ell$ distinct equivalence classes and $\{\Gamma_{u_i,v_i}\}_{i=1}^\ell$ are the $\ell$ geodesics such that $(u_i,v_i)$ form different equivalence classes.\\ Using planarity of geodesics and a combinatorial argument, it can be seen that at least one of the following three cases happen. (a) there exists $I \subset \{1,2,...\ell\}$ with $|I| \geq \ell^{1/4}$ such that for $i \in I$ the restriction of $\{\Gamma_{u_i,v_i}\}_{i \in I}$ between $\mathcal{L}_0$ and $\mathcal{L}_{n/3}$ are ordered and for $i \neq i' \in I$, $\Gamma_{u_i,v_i}$ and $\Gamma_{u_i',v_i'}$ are disjoint between $\mathcal{L}_0$ and $\mathcal{L}_{n/3}$. \cite[Proposition 3.1]{BHS21} is applied to estimate this event. (b) there exists $I \subset \{1,2,...\ell\}$ with $|I| \geq \ell^{1/4}$ such that for $i \in I$ the restriction of $\{\Gamma_{u_i,v_i}\}_{i \in I}$ between $\mathcal{L}_{2n/3}$ and $\mathcal{L}_{n}$ are ordered and for $i \neq i' \in I$ $\Gamma_{u_i,v_i}$ and $\Gamma_{u_i',v_i'}$ are disjoint between $\mathcal{L}_{2n/3}$ and $\mathcal{L}_{n}$. We can again apply \cite[Proposition 3.1]{BHS21} in this case. (c) there exists an $I' \subset \{1,2,...\ell\}$ with $|I'| \geq \frac{\ell^{1/32}}{100}$ such that the restrictions of $\{\Gamma_{u_i,v_i}\}_{i \in I'}$ are pairwise disjoint either between $\mathcal{L}_0$ and $\mathcal{L}_{n/6}$ or between $\mathcal{L}_{n/6}$ and $\mathcal{L}_{n/3}$ or between $\mathcal{L}_{n/3}$ and $\mathcal{L}_{2n/3}$ or between $\mathcal{L}_{2n/3}$ and $\mathcal{L}_n$. We can apply \cite[Proposition 3.1]{BHS21} in each of these cases. This completes the proof.
\end{proof}
\subsection{Proof of Theorem \ref{first_theorem} upper bounds}
\paragraph{\textit{Proof of Theorem \ref{first_theorem}(i) upper bound}}
We formally describe the averaging argument as outlined. For a fixed $\alpha \in (\epsilon, \frac{\pi}{2}-\epsilon)$ consider the geodesic tree $\mathcal{T}_{\alpha}$. We consider a line segment $V$ on $\mathcal{L}_{0}$ of length $n^{2/3}$ with midpoint $\boldsymbol{0}.$ For $v \in V$, let $D^{\alpha}_v$ denote the depth of the sub-tree rooted at $v$. We have
    \begin{equation}
    \label{outline_averaging_equation}
        n^{2/3}\mathbb{P}(D^{\alpha}\geq n)=\sum_{v \in V}\mathbb{P}(D^{\alpha}_v \geq n)=\mathbb{E}(\widetilde{D^\alpha}),
    \end{equation}
    where $\widetilde{D^\alpha}=\sum_{v \in V} \mathbbm{1}_{\{D^{\alpha}_v \geq n\}}$. We find an upper bound for $\mathbb{P}(\widetilde{D^\alpha} \geq \ell).$ Let $\ell<n^{0.01}$. Let us consider the line segment $\widetilde{V^\alpha}$ of length $\ell^{1/32}n^{2/3}$ with midpoint $-\boldsymbol{n}_{mn^{1/3}}$ on $\mathcal{L}_{-2n},$ where $\alpha=\tan^{-1}\left(\frac{1-m}{1+m}\right)$. Let $\mathsf{LTF}$ be the event that there exists $v \in \mathcal{L}_{-2n} \setminus \widetilde{V^\alpha}$ such that the semi-infinite geodesic starting from $v$ intersects $V$ (see Figure \ref{fig: subtree_depth_upper_bound}). We have 
    \begin{displaymath}
        \mathbb{P}(\widetilde{D^\alpha} \geq \ell) \leq \mathbb{P}(\mathsf{LTF})+\mathbb{P}(\{\widetilde{D^\alpha} \geq \ell \} \cap \mathsf{LTF}^c).
    \end{displaymath}
    Note that if $v_1,v_2$ denote the end points of $\widetilde{V^\alpha}$ then on $\mathsf{LTF}$, by planarity we have either $\Gamma_{v_1}^{\alpha}$ or $\Gamma_{v_2}^{\alpha}$ intersects $V$. Using Corollary \ref{transversal_fluctuation_of_semi_infinite_geodesic}(i) we have 
    \begin{displaymath}
        \mathbb{P}(\mathsf{LTF}) \leq Ce^{-c \ell^{\frac{3}{32}}}.
    \end{displaymath}
    Now, on the event $\{\widetilde{D^\alpha} \geq \ell \} \cap \mathsf{LTF}^c$ we have there are $\ell$ distinct vertices on $V$ that has a semi-infinite geodesic coming from $\widetilde{V^\alpha}$. From Proposition \ref{coalescence_theorem} we have for sufficiently large $n$ and $\ell <n^{0.01}$
    \begin{displaymath}
        \mathbb{P}(\{\widetilde{D^\alpha} \geq \ell \} \cap \mathsf{LTF}^c) \leq e^{-c \ell^{\frac{1}{128}}}.
    \end{displaymath}
    So, for $\ell <n^{0.01}$ sufficiently large and $n$ sufficiently large we have 
    \begin{displaymath}
        \mathbb{P}(\widetilde{D^\alpha} \geq \ell) \leq Ce^{-c \ell^{\frac{1}{128}}}.
    \end{displaymath}
    We have 
    \begin{displaymath}
    \mathbb{E}(\widetilde{D^\alpha}) \leq \sum_{1 \leq \ell < n^{0.01}}\mathbb{P}(\widetilde{D^\alpha} \geq \ell)+ \sum_{\ell \geq n^{0.01}}\mathbb{P}(\widetilde{D^\alpha} \geq \ell)
    \end{displaymath}
    For the second sum on the right hand side observe that choice of $\ell$ can be at most $n^{2/3}$. When $\ell \geq n^{0.01}$ we have 
\begin{displaymath}
    \mathbb{P}(\widetilde{D^\alpha} \geq \ell) \leq \mathbb{P}(\widetilde{D^\alpha} \geq \frac{n^{0.01}}{2}) \leq Ce^{-c n^{\frac{0.01}{128}}}.
\end{displaymath}
So, there exists $\ell_0$ such that 
\begin{displaymath}
    \mathbb{E}(\widetilde{D^\alpha}) \leq \ell_0+\sum_{\ell_0 \leq \ell <n^{0.01}} Ce^{-c \ell^{\frac{1}{128}}}+n^{2/3}Ce^{-c n^{\frac{0.01}{128}}}.
\end{displaymath}
Hence, there exists a constant $\widetilde{C}>0$ such that
\begin{displaymath}
    \mathbb{E}(\widetilde{D^\alpha}) <\widetilde{C}.
\end{displaymath}
This completes the proof for sub-tree depth upper bound. \qed
\paragraph{\textit{Proof of Theorem \ref{first_theorem}(ii) upper bound}} As outlined we will prove the upper bound for $\mathbb{P}(N^{\alpha} \geq n^{5/3})$. Also,
\begin{displaymath}
    \mathbb{P}(N^{\alpha} \geq n^{5/3}) \leq \mathbb{P}(D^{\alpha} \geq 2n)+\mathbb{P}(\{N^{\alpha} \geq n^{5/3}\} \cap D^{\alpha} < 2n).
\end{displaymath}
We already proved the upper bound for the first term. So, we will only consider the second term. Let us consider the line segment $V$ defined as before. For $v \in V$, $N^{\alpha}_v$ denotes the volume of the sub-tree rooted at $v$. Then arguing similarly as before, it suffices to prove a uniform upper bound for the expectation of the following random variable. 
\begin{displaymath}
    \widetilde{N^{\alpha}}:=\sum_{v \in V} \mathbbm{1}_{\{N^{\alpha}_{v} \geq n^{5/3}\} \cap \{D^{\alpha}_v < 2n\}}.
\end{displaymath}
\begin{comment}\begin{figure}[h!]
    \includegraphics[width=10 cm]{subtree_upper_bound.pdf}
    \caption{To prove the upper bound we consider the line segment $\widetilde{V}$ of length $\frac{\ell}{4}n^{2/3}$ on $\mathcal{L}_{-2n}$ and the rectangle R. The semi-infinite geodesics $\Gamma_1$ and $\Gamma_2$ will lie inside R and will not intersect $V$ with high probability. Hence, using planarity, with high probability there will be at most $\frac{\ell}{\sqrt{2}}n^{5/3}$ many vertices between $\mathcal{L}_{-2n}$ and $\mathcal{L}_{0}$ so that they lie on some subtree rooted on $V$. We use this to find an upper bound for $\mathbb{E}(\widetilde{N})$}
    \label{fig: subtree_upper_bound}
\end{figure}
We will show, there exists constant $C''>0$ such that for sufficiently large $n$ we have
\begin{displaymath}
    \mathbb{E}(\widetilde{N}) \leq C''.
\end{displaymath}
This will complete the proof. We show this now.
\end{comment}
For $\ell \geq 1$, we will find an upper bound for $\mathbb{P}(\{\widetilde{N^{\alpha}} \geq \ell\})$. Let $\ell<n^{\frac{1}{6}}$. On this event there are at least $\ell n^{5/3}$ many vertices such that, all these vertices $v$ have $\phi(v) \in [-2n,0]$ and they lie on some sub-tree which has its root on $V$. Let $\widetilde{V^{\alpha}}$ denote the line segment of length $\frac{\ell}{8}n^{2/3}$ on $\mathcal{L}_{-2n}$ with midpoint $-\boldsymbol{n}_{mn^{1/3}}$. Let $v_1,v_2$ are the end points of $\widetilde{V^{\alpha}}$ with $\psi(v_1) \leq \psi(v_2)$. We consider the semi-infinite geodesics $\Gamma_{v_1}^{\alpha}$ and $\Gamma_{v_2}^{\alpha}$. Let $\mathcal{J}_1$ (resp.\ $\mathcal{J}_2$) be the straight lines starting from $v_1$ (resp.\ $v_2$) in the direction $\alpha$ and for $T \in [-2n,0]$ let $v_1(T)$ (resp.\ $v_2(T)$) denote the intersection points of $\mathcal{J}_1$ (resp.\ $\mathcal{J}_2$) with $\mathcal{L}_{T}$ (see Figure \ref{fig: subtree_upper_bound}).
\begin{itemize}
    \item Let $\mathsf{TF}$ be the event such that for all $T \in [-2n,0]$, $|\psi(\Gamma_{v_1}^{\alpha}(T))-\psi(v_1(T))| \leq \frac{\ell}{16}n^{2/3}$ and $|\psi(\Gamma_{v_2}^{\alpha}(T))-\psi(v_2(T))| \leq \frac{\ell}{16}n^{2/3}$.
\end{itemize}
Now using Corollary \ref{transversal_fluctuation_of_semi_infinite_geodesic}(ii) we have for sufficiently large $\ell$ and sufficiently large $n$
\begin{equation}
\label{transversal_fluctuation_upper_bound}
    \mathbb{P}((\mathsf{TF})^c) \leq Ce^{-c \ell^3},
\end{equation}
Consider the following parallelogram.
\begin{itemize}
   \item $\mathsf{R}$ is the parallelogram whose one pair of opposite sides lie on $\mathcal{L}_0$ (resp.\ on $\mathcal{L}_{-2n}$) each of length $\frac{\ell}{4}n^{2/3}$ and midpoints $\boldsymbol{0}$ (resp.\ $-\boldsymbol{n}_{mn^{1/3}}$).
\end{itemize}
On the event $\mathsf{TF}$ both $\Gamma_{v_1}^{\alpha}$ and $\Gamma_{v_2}^\alpha$ lies within $\mathsf{R}$. Also, using planarity of geodesics, on the event $\mathsf{TF}$, any semi-infinite  geodesic starting outside the rectangle $\mathsf{R}$ will not intersect $V$. So, in this case there are at most $\frac{\ell}{2}n^{5/3}$ (number of vertices in $\mathsf{R}$ can be at most $\frac{\ell}{2}n^{5/3}$) many vertices $v$ such that $\phi(v) \in [-2n,0]$ and the semi-infinite geodesic starting from $v$ will lie on a sub-tree rooted on $V$ (see Figure \ref{fig: subtree_upper_bound}).\\
So, from the above argument it is clear that
\begin{displaymath}
    \{\widetilde{N^\alpha} \geq \ell\} \subset (\mathsf{TF})^c.
\end{displaymath}
Hence from \eqref{transversal_fluctuation_upper_bound} for sufficiently large $\ell<n^{\frac{1}{6}}$ and sufficiently large $n$ we have 
\begin{displaymath}
    \mathbb{P}( \{\widetilde{N} \geq \ell\}) \leq Ce^{-c\ell^3}.
\end{displaymath}
A similar calculation as done in the proof of sub-tree depth upper bound shows that $\mathbb{E}(\widetilde{N^\alpha})$ is uniformly bounded by a constant. This completes the proof for upper bound. \qed\\
\subsection{Proof of Theorem \ref{first_theorem} lower bounds} 
\paragraph{\textit{Proof of Theorem \ref{first_theorem}(i) lower bound}}We fix some large $M$ (to be chosen later) and consider the line segment $V_M$ of length $Mn^{2/3}$ on $\mathcal{L}_{0}$ with midpoint $\boldsymbol{0}$. Again let $\widetilde{D^{\alpha}_M}$ denote the number of vertices on $V_M$ that has sub-tree depth at least $n$. Let $V_M^{\alpha}$ be the line segment with end points $v_1,v_2$ of length $\frac{M}{2}n^{2/3}$ on $\mathcal{L}_{-2n}$ with midpoint $-\boldsymbol{n}_{mn^{1/3}}$. For $T \in [-2n,0], \Gamma_{v_1}^{\alpha}(T),\Gamma_{v_2}^{\alpha}(T),v_1(T),v_2(T)$ are defined as before. We consider the following event (see also Figure \ref{fig: subtree_lower_bound}).
\begin{itemize}
    \item $\widetilde{\mathsf{TF}}$ is the event that for all $T \in [-2n,2n], |\psi(\Gamma_{v_1}^{\alpha}(T))-\psi(v_1(T))| \leq \frac{M}{8}n^{2/3}$ and $|\psi(\Gamma_{v_2}^{\alpha}(T))-\psi(v_2(T))| \leq \frac{M}{8}n^{2/3}$.
\end{itemize}
By Corollary \ref{transversal_fluctuation_of_semi_infinite_geodesic} we can chose $M$ large enough so that for all $n$ we have $\mathbb{P}(\widetilde{\mathsf{TF}}) \geq 0.99$. Clearly, on $\widetilde{\mathsf{TF}}$ by planarity any geodesic starting from $\widetilde{V}_M^{\alpha}$ will not intersect $\mathcal{L}_{0}$ outside $V_M$. Hence, on $\widetilde{\mathsf{TF}}$, $\widetilde{D^{\alpha}_M} \geq 1$. Hence,
\begin{displaymath}
    \mathbb{E}(\widetilde{D^{\alpha}_M}) \geq \mathbb{P}(\widetilde{\mathsf{TF}}) \geq 0.99.
\end{displaymath}
So,
\begin{displaymath}
    \mathbb{P}(D^{\alpha} \geq n) \geq \frac{0.99}{Mn^{2/3}}.
\end{displaymath}
This proves the lower bound in Theorem \ref{first_theorem}(i). \qed
\paragraph{\textit{Proof of Theorem \ref{first_theorem}(ii) lower bound}} We fix $M$ as in the last proof and large $\ell$ (this will be chosen later). Observe that using planarity, on the event $\widetilde{\mathsf{TF}}$ any semi-infinite geodesic starting from the parallelogram $\widetilde{\mathsf{R}},$ whose one pair of opposite sides lie on $\mathcal{L}_{-2n}$ (resp.\ $\mathcal{L}_{-n}$) each of length $n^{2/3}$ and midpoints $-\boldsymbol{n}_{mn^{1/3}}$ (resp.\ $-(\boldsymbol{\frac{n}{2}})_{mn^{1/3}}$), will intersect $\mathcal{L}_{0}$ on $V_M$. Hence, on $\widetilde{\mathsf{TF}}$ there are at least $n^{5/3}$ many vertices that lie on some sub-tree rooted at $V_M.$ Further, let $\widetilde{V^{\alpha}_M}$ denote the line segment of length $Mn^{2/3}$ on $\mathcal{L}_{-n}$ with midpoint $-\boldsymbol{(\frac{n}{2})}_{mn^{1/3}}$ (see Figure \ref{fig: subtree_lower_bound}). On $\widetilde{\mathsf{TF}}$ any geodesic starting from $\widetilde{\mathsf{R}}$ will intersect $\mathcal{L}_{-n}$ on $\widetilde{V ^\alpha_M}$.
Let $M^{\alpha}$ denote the number of points on $V_M$ which has a semi-infinite geodesic starting from $\widetilde{V^\alpha_M}$ passing through it. Using Proposition \ref{coalescence_theorem} we get that, depending on $M$ we can chose an $\ell$ large enough so that for all sufficiently large $n$
\begin{displaymath}
    \mathbb{P}(M^{\alpha} \leq \ell) \geq 0.99.
\end{displaymath}
We observe that 
\begin{displaymath}
    \mathbb{P}(\widetilde{\mathsf{TF}} \cap \{M^{\alpha} \leq \ell\}) \geq 0.98
\end{displaymath}
and on $\widetilde{\mathsf{TF}} \cap \{M^{\alpha} \leq \ell\}, \widehat{N^{\alpha}} \geq 1,$  where
\begin{displaymath}
    \widehat{N^{\alpha}}:=\sum_{v \in V}\mathbbm{1}_{\{N^{\alpha}_v \geq \frac{n^{5/3}}{\ell}\}}.
\end{displaymath}
Hence,
\begin{displaymath}
    \mathbb{E}(\widehat{N^{\alpha}}) \geq \mathbb{P}(\widetilde{\mathsf{TF}} \cap \{M^{\alpha} \leq \ell\}) >0.98.
\end{displaymath}
So, we get for sufficiently large $n$.
\begin{displaymath}
    \mathbb{P}(N^{\alpha} \geq \frac{n^{5/3}}{\ell}) > \frac{0.98}{Mn^{2/3}}.
\end{displaymath}
Hence, for sufficiently large $n$,
\begin{displaymath}
    \mathbb{P}(N^{\alpha} \geq n^{5/3}) > \frac{0.98}{M \ell^{2/5}n^{2/3}}=\frac{c}{n^{2/3}}.
\end{displaymath}
This proves the lower bound in Theorem \ref{first_theorem}(ii). \qed
\begin{comment}So, on TF $\widetilde{D^{\alpha}_M} \geq 1$. Hence, $\mathbb{E}(\widetilde{D^{\alpha}_M})$ is bounded below by a uniform constant $c>0$. This proves the lower bound for Theorem \ref{first_theorem}(1). To prove the lower bound in Theorem \ref{first_theorem}(2) observe that in the previous argument we could chose the event TF in such a way that, the semi-infinite geodesics starting from $-\boldsymbol{n_{(mn^{1/3}-M)}}$ and $-\boldsymbol{n_{(mn^{1/3}+M)}}$ do not enter the rectangle $\widetilde{R}:=\{v \in \mathbb{Z}^2: |\psi(v)-\psi(-\boldsymbol{n_{mn^{1/3}}})| \leq n^{2/3}, -2n \leq \phi(v) \leq -n\}$ and do not intersect $\mathcal{L}_{0}$ outside $V_M$. So, on TF by planarity all $n^{5/3}$ vertices in $\widetilde{R}$ will have its semi-infinite geodesics intersecting $V_M$. Due to coalescence, number of vertices lying in $V_M$ having semi-infinite geodesics is small. In particular, using \cite[Theorem 3.10]{BHS21} we can chose large enough $\ell$ so that the event that there are at most $\ell$ many vertices in $V_M$ having a semi-infinite geodesic coming from $\widetilde{R}$ (denoted by COL). As both the event TF and COL are large probability events, their intersection has positive probability and on this event there is at least one vertex in $V_M$ that has at least $\frac{n^{5/3}}{\ell}$ many vertices through it. So, applying averaging argument again $\mathbb{P}(N^{\alpha} \geq \frac{n^{5/3}}{\ell}) \geq \frac{c}{n^{2/3}}.$ This proves the lower bound in  Theorem\ref{first_theorem}(2).
\end{comment}
\section{Upper bound for geodesic intersections}
%Proof of Theorem \ref{main_theorem_1} Upper Bounds}
\label{proof_of_upper_bound}
 \begin{comment}We first explain the outline of the proof. To prove Theorem \ref{main_theorem_1}(1), we apply an averaging argument. Specifically, instead of working with the probability that the origin lies on the intersection of two geodesics, we consider rectangles whose length in the time direction is some fixed constant multiple of $n$ and width in the space direction is some fixed constant multiple of $n^{2/3}$ around $-\boldsymbol{n}, \boldsymbol{n}, -\boldsymbol{n_k}, \boldsymbol{n_k}$ and $\boldsymbol{0}$ (see Figure \ref{fig: basic_picture}). Definitions are as follows.
 \end{comment}
\begin{comment} \begin{figure}[h!]
    \includegraphics[width=10 cm]{Figure_1.pdf}
    \caption{We prove Theorem \ref{main_theorem_1}(1) by an averaging method. We consider rectangles $A_n,B_n,V,A_n^*,B_n^*$ around $-\boldsymbol{n},-\boldsymbol{n_k},\boldsymbol{0},\boldsymbol{n},\boldsymbol{n_k}$ of equal size whose length along the time direction is $\frac{n}{100k}$ and length along space direction $n^{2/3}$. We want to find an upper bound for the expected number of points in $v \in V$ such that it lies in the intersection of $\Gamma_{w_1,w_1'}$ and $\Gamma_{w_2,w_2'}$ for some $w_1 \in A_n,w_1' \in A_n^*,w_2 \in B_n$ and $w_2' \in B_n^*$.}
    \label{fig: basic_picture}
\end{figure}
\end{comment}
\subsection{Proof of Proposition \ref{general_upper_bound}(i)}As outlined, we will prove the upper bounds for Theorem \ref{main_theorem_1} for all $1<|k|<(1-\epsilon)n^{1/3}$. We will also assume $k>0$ and sufficiently large. The negative $k$ case follows by taking $|k|$ in place of $k$ in all the arguments below and the same proofs work. For small $k$ we obtain the result by adjusting the constants. Let us define the following rectangles (also see Figure \ref{fig: basic_picture}).
 \begin{itemize}
   \item  $V:=\{v=(v_1,v_2) \in \mathbb{Z}^2: |\psi(v)|<n^{2/3}$ and $|\phi(v)|< \frac{n}{100k}\}$.
    \item $A_n:=\{-\boldsymbol{n}+v : v \in V\}$,
    \item $B_n:=\{-\boldsymbol{n}_k+v : v \in V\}$,
    \item $A_n^*:=\{\boldsymbol{n}+v: v \in V\}$,
    \item $B_n^*:=\{\boldsymbol{n}_k+v : v \in V\}$. 
    \end{itemize}
    \begin{comment}
\begin{figure}[h!]
    \includegraphics[width=10 cm]{Figure_1.pdf}
    \caption{We prove Theorem \ref{main_theorem_1}(1) by an averaging method. We consider rectangles $A_n,B_n,V,A_n^*,B_n^*$ around $-\boldsymbol{n},-\boldsymbol{n_k},\boldsymbol{0},\boldsymbol{n},\boldsymbol{n_k}$ of equal size whose length along the time direction is $\frac{n}{100k}$ and length along space direction $n^{2/3}$. We want to find an upper bound for the expected number of points in $v \in V$ such that it lies in the intersection of $\Gamma_{w_1,w_1'}$ and $\Gamma_{w_2,w_2'}$ for some $w_1 \in A_n,w_1' \in A_n^*,w_2 \in B_n$ and $w_2' \in B_n^*$.}
    \label{fig: basic_picture}
\end{figure}
\end{comment}
\begin{comment}We explain how this method helps us to prove Theorem \ref{main_theorem_1}(1). 
\end{comment}
For all $v \in V$, we denote the unique geodesic from $-\boldsymbol{n}+v$ to  $\boldsymbol{n}+v$ by $\Gamma_n^v$. Similarly, $\Gamma_{n_k}^{v}$ denote the unique geodesics from $-\boldsymbol{n}_k+v$ to $\boldsymbol{n}_k+v$. We have for all $v \in V$
\begin{equation}
   \label{first_equation}
   \mathbb{P}(\boldsymbol{0} \in \Gamma_n^{\boldsymbol{0}} \cap \Gamma_{n_k}^{\boldsymbol{0}})=\mathbb{P}(v \in \Gamma_n^v \cap \Gamma_{n_k}^{v}).
\end{equation}
Further, 
\begin{equation}
\label{second_equation}
    \frac{n^{5/3}}{100 k}\mathbb{P}(\boldsymbol{0} \in \Gamma_n^{\boldsymbol{0}} \cap \Gamma_{n_k}^{\boldsymbol{0}})=\sum_{v \in V} \mathbb{P}(v \in \Gamma_n^v \cap \Gamma_{n_k}^{v})=\mathbb{E} \left(\sum_{v \in V} \mathbbm{1}_{\{ v \in \Gamma_n^v \cap \Gamma_{n_k}^{v} \} } \right) \leq \mathbb{E}(N^V),
\end{equation}
where $N^V$ is defined as follows.
\begin{displaymath}
    N^V:=\sum_{v \in V} \mathbbm{1}_{\{ v \in \Gamma_n^v \cap \Gamma_{n_k}^{v} \}}.
\end{displaymath} then it is sufficient to prove the following.
\begin{comment}Here $N^V$ is the random variable defined as the number of points $v \in V$ such that $\exists  w_1 \in A_n, w_2 \in A_n^*$ and $w_3 \in B_n$ and $w_4 \in B_n^*$ such that $w \in \Gamma_{w_{1,}w_{2}} \cap \Gamma_{w_{3},w_{4}}$. Hence we want to find an upper bound of $\mathbb{E}(N^V)$. In particular, we will prove 
\end{comment}
\begin{equation}
\label{third_equation}
    \mathbb{E}(N^V) \leq C n \log k/k^3.
\end{equation}
\eqref{second_equation} and \eqref{third_equation} together prove Proposition \ref{general_upper_bound}(i).
\\As discussed earlier, we will prove \eqref{third_equation} as follows. Let us first define the following random variables. Let $w_1,w_2 \in A_n$ and $w_1',w_2' \in A_n^*$. We say $(w_1,w_1')\sim (w_2,w_2')$ if the geodesics $\Gamma_{w_{1},w_{1}'}$ and $\Gamma_{w_{2},w_{2}'}$ coincide inside $V$. Let $M_n^A$ denote the number of equivalence classes. Similarly we define $M_n^B$ for $B_n$ and $B_n^*$. For two geodesics $\Gamma_n$  and $\Gamma_n'$ starting from $A_n$ and ending at $A_n^*$ (resp.\ starting at $B_n$ and ending at $B_n^*$) let $I_{\Gamma_n,\Gamma_n'}$ denote the number of intersection points of these two geodesics inside $V$. 
Recall that we had following two inequalities.
\begin{equation}
    N^V \leq M_n^A M_n^B \max I_{\Gamma_n,\Gamma_n'},
\end{equation} 
where the maximum is taken over all pairs of geodesics starting from $A_n$ (resp.\ $B_n$) and ending at $A_n^*$ (resp.\ $B_n^*$). Further,
\begin{displaymath}
\mathbb{P}(N^V \geq \ell n \log k/k^3) \leq \mathbb{P}(M_n^A \geq \ell^{1/3})+\mathbb{P}(M_n^B \geq \ell^{1/3})+\mathbb{P}(\max I_{\Gamma_n,\Gamma_n'} \geq \ell^{1/3}n \log k/k^3).
\end{displaymath}
We will prove following two lemmas.
\begin{lemma}
\label{first_lemma}
In the above setup, there exist constants $C,c>0$ (depending on $\epsilon$) such that for sufficiently large $n$ and $\ell < n^{0.01}$ sufficiently large
\begin{displaymath}\mathbb{P}(M_n^A \geq \ell) \leq Ce^{-c \ell^{1/128}}.
\end{displaymath}
Same bound holds for $M_n^B$.
\end{lemma}
\begin{lemma}
\label{second_lemma}
In the above setup for sufficiently large $k$, there exists constant $C,c>0$ (depending on $\epsilon$) such that for all $n$ and $\ell$
\begin{displaymath}\mathbb{P}(\max I_{\Gamma_n,\Gamma_n'} \geq \ell n \log k/k^3) \leq Ce^{-c\ell},
\end{displaymath}
where the maximum is taken over all pairs of geodesics starting from $A_n$ (resp.\ $B_n$) and ending at $A_n^*$ (resp.\ $B_n^*$).
\end{lemma}
\begin{comment}
Recall that in Section \ref{Outline} we had following two inequalities.
\begin{equation}
    N^V \leq M_n^A M_n^B \max I_{\Gamma_n,\Gamma_n'},
\end{equation} 
where the maximum is taken over all pairs of geodesics starting from $A_n$ (resp.\ $B_n$) and ending at $A_n^*$ (resp.\ $B_n^*$).
\begin{displaymath}
\mathbb{P}(N^V \geq \ell n \log k/k^3) \leq \mathbb{P}(M_n^A \geq \ell^{1/3})+\mathbb{P}(M_n^B \geq \ell^{1/3})+\mathbb{P}(\max I_{\Gamma_n,\Gamma_n'} \geq \ell^{1/3}n \log k/k^3).
\end{displaymath}
\end{comment}
So, using Lemma \ref{first_lemma} and Lemma \ref{second_lemma} we have for $\ell<n^{0.03}$ sufficiently large
\begin{displaymath}
    \mathbb{P}(M_n^A \geq \ell^{1/3})+\mathbb{P}(M_n^B \geq \ell^{1/3})+\mathbb{P}(\max I_{\Gamma_n,\Gamma_n'} \geq \ell^{1/3}n \log k/k^3) \leq 2Ce^{-c \ell^{\frac{1}{384}}}+Ce^{-c\ell^{1/3}}.
\end{displaymath}
Using a same argument as used in proof of Theorem \ref{first_theorem} (i) upper bound we get that for sufficiently large $n$ expectation of the random variable $\frac{N^V k^3}{n\log k}$ is uniformly bounded by a constant. This proves Proposition \ref{general_upper_bound}. \qed \\
We prove Lemma \ref{first_lemma} and Lemma \ref{second_lemma} now.
\begin{figure}[t!]
    \includegraphics[width=13 cm]{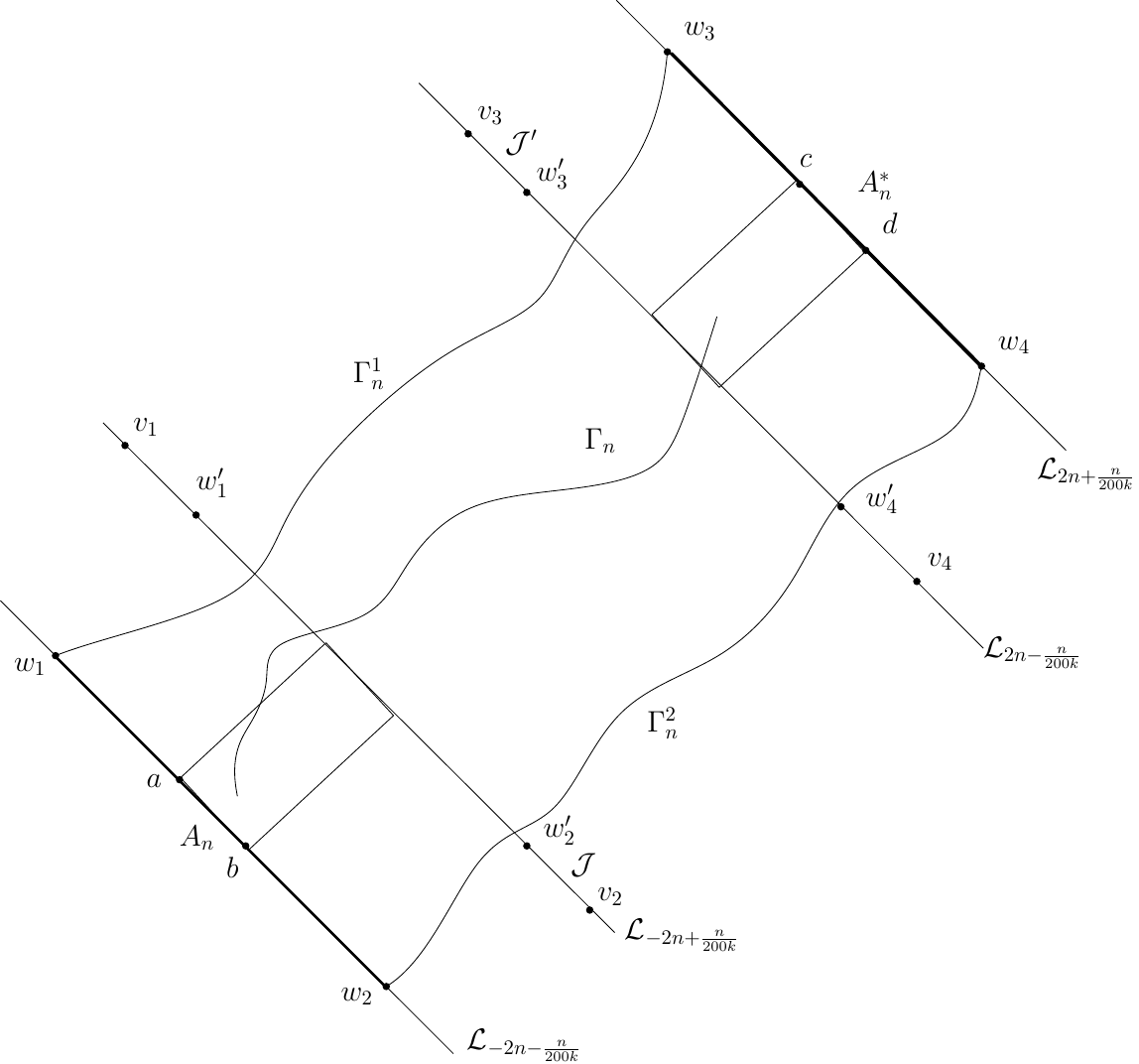}
    \caption{The idea of the proof of Lemma \ref{first_lemma} is to apply Proposition \ref{coalescence_theorem}. To apply Proposition \ref{coalescence_theorem} we need to consider geodesics starting from a line and ending at another. To do this we take two deterministic points $w_1,w_2$ (resp.\ $w_3,w_4$) on $\mathcal{L}_{-2n-\frac{n}{200k}}$ (resp.\ $\mathcal{L}_{2n+\frac{n}{200k}}$) so that $\Gamma_{w_{1},w_{3}}$ and $\Gamma_{w_{2},w_{4}}$ have very small probability of  entering the rectangles due to transversal fluctuation. By using planarity of geodesics any geodesic starting from $A_n$ and ending at $A_n^*$ will be sandwiched between these two geodesics with high probability. Using transversal fluctuation again we can get two extended line segments $\mathcal{J}$ and $\mathcal{J}'$ so that with high probability $\Gamma_{w_{1},w_{3}}$ and $\Gamma_{w_{2},w_{4}}$ do not go outside $\mathcal{J}$ and $\mathcal{J}'$. Hence we can restrict to the event that all geodesics from $A_n$ to $A_n^*$ intersect $\mathcal{L}_{-2n+\frac{n}{200k}}$ and $\mathcal{L}_{2n-\frac{n}{200k}}$ only on $\mathcal{J}$ and $\mathcal{J}'$. Now we can apply Proposition \ref{coalescence_theorem} on the line segments $\mathcal{J}$ and $\mathcal{J}'$. }
    \label{fig: box to box}
\end{figure}
\begin{proof}[Proof of Lemma \ref{first_lemma}]
 Let us consider $w_1,w_2$ (resp.\ $w_3,w_4$) on the line $\mathcal{L}_{-2n-\frac{n}{200k}}$ (resp.\ \\$\mathcal{L}_{2n+\frac{n}{200k}}$) such that $\psi(w_1) \leq \psi(a)$ and $\psi(a)-\psi(w_1)=\ell^{1/32}(n/k)^{2/3}$ and $\psi(b )\leq \psi(w_2)$ and $\psi(w_2)-\psi(b)=\ell^{1/32}(n/k)^{2/3}$ (resp.\ $\psi(w_3) \leq \psi(c))$ and $\psi(c)-\psi(w_3)=\ell^{1/32}(n/k)^{2/3}$ and $\psi(d) \leq \psi(w_4)$ and $\psi(w_4)-\psi(d)=\ell^{1/32}(n/k)^{2/3}$), where $a,b$ (resp.\ $c,d$) are the two bottom (resp.\ up) corners of $A_n$ (resp.\ $A_n^*$) (see Figure \ref{fig: box to box}). We denote the unique geodesic from $w_1$ to $w_3$ by $\Gamma_n^1$ and the unique geodesic from $w_2$ to $w_4$ by $\Gamma_n^2$. Further, consider $w_1',w_3'$, the intersection points of $\mathcal{L}_{-2n+\frac{n}{200k}}$ (resp.\ $\mathcal{L}_{2n-\frac{n}{200k}}$)  and the straight lines joining $w_1,w_3$. Similarly we define $w_2',w_4'.$ Consider $v_1$ (resp.\ $v_2$) on $\mathcal{L}_{-2n+\frac{n}{200k}}$ such that $\psi(v_1) \leq \psi(w_1'), \psi(v_3) \leq \psi(w_3')$ and $\psi(w_1')-\psi(v_1)=\psi(w_3')-\psi(v_3)=\ell^{1/32}(n/k)^{2/3}$. We consider similar points $v_2,v_4$ (see Figure \ref{fig: box to box}). We have 
\begin{equation}
\label{coalescence_union_bound}
    \mathbb{P}(M_n^A \geq \ell) \leq \mathbb{P}(\{M_n^A \geq \ell\} \cap \mathcal{B})+\mathbb{P}(\mathcal{B}^{c}).
\end{equation} Here $\mathcal{B}$ is the event defined as,
\begin{displaymath}
\mathcal{B}:=\{\psi(\Gamma_n^1(t)) \leq \psi(\Gamma_n(t)) \leq \psi(\Gamma_n^2(t)): \forall t \in [-2n-\tfrac{n}{200k},2n+\tfrac{n}{200k}] \cap\{\phi(\Gamma_n(s))\},\ \forall \Gamma_n \,\text{from}\, A_n\, \text{to} \,A_n^*\}.
\end{displaymath}
By planarity of geodesics $\mathcal{B}^{c}$ can happen only four ways. Either $\Gamma_n^1(t) \in A_n$ for some $t$ (denoted by $\mathcal{B}_1$) or $\Gamma_n^2(t) \in A_n$ for some $t$ (denoted by $\mathcal{B}_2$) or $\Gamma_n^1(t) \in A_n^*$ for some $t$ (denoted by $\mathcal{B}_3$) or $\Gamma_n^2(t) \in A_n^*$ for some $t$ (denoted by $\mathcal{B}_4$). So, by Proposition \ref{transversal_fluctuation}(ii) we have  
\begin{displaymath}
\mathbb{P}(\mathcal{B}_1) \leq \mathbb{P}\left(\sup \{|\psi(\Gamma_n^1(t))-\psi(v(t))|: 0 \leq t \leq \frac{n}{100k}\text{ and } t \in \mathbb{Z} \} > \ell^{1/32}(n/k)^{2/3}\right) \leq e^{-c\ell^{3/32}}.
\end{displaymath} We get similar bounds for $\mathcal{B}_2, \mathcal{B}_3, \mathcal{B}_4$. Hence,
\begin{displaymath}
    \mathbb{P}(\mathcal{B}^c) \leq 4e^{-c \ell^{3/32}}.
\end{displaymath}\\ For the first event in (\ref{coalescence_union_bound}) we define the following events. Let,
\begin{itemize}
\item $\mathcal{C}_1:=\{\psi(v_1) < \psi(\Gamma_n(-2n+\frac{n}{200k})) < \psi(v_2)$, for all $\Gamma_n$ from $A_n$ to $A_n^*\}$;
   \item $\mathcal{C}_2:=\{\psi(v_3) < \psi(\Gamma_n(2n-\frac{n}{200k}) <\psi(v_4)$ for all $\Gamma_n$ from $A_n$ to  $A_n^*\}$;
   \item  $\mathcal{C}:=\mathcal{C}_1 \cap \mathcal{C}_2$.
\end{itemize}
Now
\begin{equation}
\label{union}
\mathbb{P}(\{M_n^A \geq \ell\} \cap \mathcal{B}) \leq \mathbb{P}(\{M_n^A \geq \ell\} \cap \mathcal{B} \cap \mathcal{C})+\mathbb{P}(\mathcal{B} \cap \mathcal{C}^{c}).
\end{equation} Also $M_n^A\mathbbm{1}_{\mathcal{B} \cap \mathcal{C}} \leq M_n$ where $M_n$ is defined as in Proposition \ref{coalescence_theorem} for the line segments $[v_1,v_2](\mathcal{J})$, $[v_3,v_4](\mathcal{J}')$. So, the first term in (\ref{union}) is upper bounded by $\mathbb{P}(M_n \geq \ell).$ Then for sufficiently large $n,\ell$ with $\ell < n^{0.01}$ we have from Proposition \ref{coalescence_theorem}
\begin{displaymath}
\mathbb{P}(M_n\geq \ell)\leq e^{-c\ell^{1/128}}.
\end{displaymath}
For the second event in (\ref{union}) we define
\begin{itemize}
    \item $\mathcal{B}_1':=\{\psi(\Gamma_n^1(-2n+\frac{n}{200k}) < \psi(v_1)\}$;
    \item $\mathcal{B}_2':=\{\psi(\Gamma_n^2(-2n+\frac{n}{200k}) > \psi(v_2)\}$;
    \item $\mathcal{B}_3':=\{\psi(\Gamma_n^1(2n-\frac{n}{200k})) < \psi(v_3)\};$
    \item $\mathcal{B}_4':=\{\psi(\Gamma_n^2(2n-\frac{n}{200k})) > \psi(v_4)\}.$
\end{itemize}
We have by planarity 
\begin{displaymath}
\mathcal{B} \cap \mathcal{C}^{c} \subset \mathcal{B}_1' \cup \mathcal{B}_2'\cup  \mathcal{B}_3' \cup \mathcal{B}_4'.
\end{displaymath}
Due to Proposition \ref{transversal_fluctuation}(i) each term on the right side is upper bounded by $e^{-c\ell^{3/32}}.$
Combining the arguments we get 
\begin{displaymath}
\mathbb{P}(M_n^A \geq \ell) \leq 4e^{-c\ell^{3/32}}+e^{-c\ell^{1/128}}+ 4e^{-c \ell^{3/32}} \leq 9e^{-c \ell^{1/128}}
\end{displaymath}
This proves Lemma \ref{first_lemma}.\\
\end{proof}
\begin{proof}[Proof of Lemma \ref{second_lemma}]
 We will assume $1< k<(1-\epsilon)n^{1/3}$ sufficiently large.
 \begin{comment}
 \begin{figure}[h!]
     \centering
     \begin{subfigure}[b]{0.4 \textwidth}
         \centering
         \includegraphics[width=5 cm]{Figure_6.pdf}
         \caption{}
         \label{fig: discretization}
     \end{subfigure}
     \hfill
     \begin{subfigure}[b]{0.4 \textwidth}
         \centering
         \includegraphics[width=5 cm]{Figure_7.pdf}
         \caption{}
         \label{fig: magnification}
     \end{subfigure}
     \caption{Figure (\ref{fig: discretization}) shows the discretization of $V$ done in proof of Proposition \ref{second_lemma}. Smaller rectangles inside $V$ has length $\ell n/2k^3$ in the time direction and $\ell n^{2/3}/8k^2$ in the space direction. So, if the intersection length of two geodesics is at least $\ell n/k^3$ inside $V$, then they will enter exactly one of these smaller rectangles $R$ through a single point and leave the rectangle through a single point. We then take a union bound all such smaller rectangles.\\
     Figure (\ref{fig: magnification}) shows a magnified version of the rectangle $R$. Proposition \ref{second_lemma} reduces to the event that there exists a point $u$ on $[w_1,w_2]$ such that $\Gamma_{u,v'}$ and $\Gamma_{u,v''}$ leave $R$ through a single point $u'$. By planarity of geodesics this event will imply either of $\Gamma_{w_1,v_1}$ or $\Gamma_{w_2,v_2}$ will have large transversal fluctuation on $[w_3,w_4]$. Then applying Theorem \ref{transversal_fluctuation} we prove Proposition \ref{second_lemma}}
     \label{fig: figure_for_second_lemma}
     \end{figure}
     \end{comment}
Recall that we want to estimate the maximum size of intersection of geodesics that start from $A_n$ (resp.\ $B_n$) and end at $A_n^*$ (resp.\ $B_n^*$). Note that following the proof of Lemma \ref{first_lemma} we can restrict to the event that the geodesics are ending at two line segments $[v_1',v_2']$ and $[v_3',v_4']$ of length $\ell^{1/3}n^{2/3}$ (see Figure \ref{fig: subfigure}) on the line $\mathcal{L}_{2n-\frac{n}{200k}}$. Using Proposition \ref{transversal_fluctuation} complement of this event has probability bounded by $e^{-c \ell}.$\\
\begin{comment}\begin{figure}[t!]
     \centering
     \begin{subfigure}[7a]{0.4 \textwidth}
         \centering
         \includegraphics[width=5 cm]{Figure_6.pdf}
         \caption{}
         \label{fig: discretization}
     \end{subfigure}
     \begin{subfigure}[7b]{0.4 \textwidth}
         \centering
         \includegraphics[width=5 cm]{Figure_7.pdf}
         \caption{}
         \label{fig: magnification}
     \end{subfigure}
     \caption{Figure (\ref{fig: discretization}) shows the discretization of $V$ done in proof of Lemma \ref{second_lemma}. Smaller rectangles inside $V$ has length $\ell n/2k^3$ in the time direction and $\ell n^{2/3}/8k^2$ in the space direction. So, if the intersection length of two geodesics is at least $\ell n/k^3$ inside $V$, then they will enter exactly one of these smaller rectangles $R$ through a single point and leave the rectangle through a single point. We then take a union bound over all such smaller rectangles.\\
     Figure (\ref{fig: magnification}) shows a magnified version of the rectangle $R$. Lemma \ref{second_lemma} reduces to the event that there exists a point $u$ on $[w_1,w_2]$ such that $\Gamma_{u,v}$ and $\Gamma_{u,v'}$ leave $R$ through a single point $u'$. By planarity of geodesics this event will imply either of $\Gamma_{w_1,v_1}$ or $\Gamma_{w_2,v_4}$ will have large transversal fluctuation on $[w_3,w_4]$. Then applying Proposition \ref{transversal_fluctuation} we prove Lemma \ref{second_lemma}}
     \label{fig: figure_for_second_lemma}
     \end{figure}
     \end{comment}
     \begin{figure}[t!]
    \includegraphics[width=13 cm]{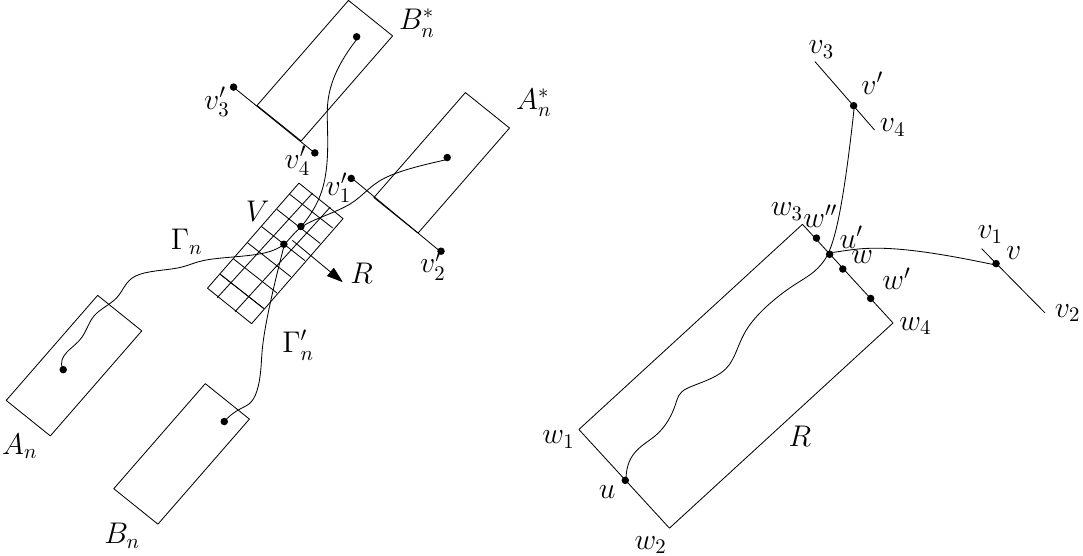}
    \caption{{Left hand side figure shows the discretization of $V$ done in proof of Lemma \ref{second_lemma}. Smaller rectangles inside $V$ have length $\ell n/2k^3$ in the time direction and $\ell n^{2/3}/8k^2$ in the space direction. So, if the intersection length of two geodesics is at least $\ell n/k^3$ inside $V$, then they will enter at least one of these smaller rectangles $R$ through a single point and leave the rectangle through a single point. We then take a union bound over all such smaller rectangles.\\
     Right hand side figure shows a magnified version of the rectangle $R$. Lemma \ref{second_lemma} reduces to the event that there exists a point $u$ on $[w_1,w_2]$ such that $\Gamma_{u,v}$ and $\Gamma_{u,v'}$ leave $R$ through a single point $u'$. By planarity of geodesics this event will imply either of $\Gamma_{w_1,v_1}$ or $\Gamma_{w_2,v_4}$ will have large transversal fluctuation on $[w_3,w_4]$. Then applying Proposition \ref{transversal_fluctuation} we prove Lemma \ref{second_lemma}}}
    \label{fig: subfigure}
\end{figure}
 First we define few events. For two geodesics $\Gamma_n$ and $\Gamma_n'$ starting from $A_n$ and $B_n$  and ending at $A_n^*$ and $B_n^*$ respectively, let $\mathcal{A}_{\Gamma_n,\Gamma_n'}$ denote the event that the intersection size of $\Gamma_n$ and $\Gamma_n'$ inside $V$ is at least $\ell n/k^3$. Then in notation of Lemma \ref{second_lemma} we have 
 \begin{displaymath}
     \{\max I_{\Gamma_n,\Gamma_n'} \geq \ell n/k^3\} \subset \bigcup \mathcal{A}_{\Gamma_n,\Gamma_n'}
 \end{displaymath}
 Here the union and the maximum are over all the pairs of geodesics described above. \\
 We discretize $V$ in both space and time direction. In time direction and in space direction we divide $V$ into smaller rectangles whose two sides are parallel to the time axes and are of length $\ell n/2k^3$ and other two sides are parallel to the space axes and are of length $\ell n^{2/3}/8k^2$ (Figure \ref{fig: subfigure}). If the size of the intersection of some $\Gamma_n$ and $\Gamma_n'$ is at least $\ell n/k^3$ then in at least one of these rectangles they will enter the rectangle through a single point on the side parallel to the space axes and go out of the opposite side of that through a single point (Figure \ref{fig: subfigure}).\\ For a rectangle $R$ with corner points $w_1,w_2,w_3,w_4$ let $\mathcal{B}_R$ denote the event that there exist points $u \in [w_1,w_2], u' \in [w_3,w_4], v \in [v_1',v_2'], v' \in [v_3',v_4']$  such that $\Gamma_{u,v}$ and $\Gamma_{u,v'}$ coincide inside $R$ and leave the rectangle $R$ at point $u'$. We have 
 \begin{equation}
     \label{second_lemma_first_equation}
     \{\max I_{\Gamma_n,\Gamma_n'} \geq \ell n/k^3\} \subset \bigcup \mathcal{B}_R,
 \end{equation}
 where the union is over all the rectangles inside $V$. Note that there are $\mathcal{O}(k^4/\ell^2)$ rectangles inside $V$. So, if we can estimate $\mathbb{P}(\mathcal{B}_R)$ then taking union bound over all smaller rectangles will give the result. Also observe that $\ell$ can be at most $k^2$. So, for sufficiently large $k$, $\ell^{1/3} < \frac{k}{4}.$ We find an upper bound for $\mathbb{P}(\mathcal{B}_R)$ using Lemma \ref{intersection size lemma} below.\\
 So, we have 
 \begin{displaymath}
     \mathbb{P}(\mathcal{B}_R) \leq 2e^{-c \ell}.
 \end{displaymath}
 Hence by (\ref{second_lemma_first_equation}) taking union bound and observing that there are $\mathcal{O}(k^4/\ell^2)$ rectangles we have for some constant $M>0,$
 \begin{equation}
     \label{second_lemma_sixth_equation}
     \mathbb{P}(\max I_{\Gamma_n,\Gamma_n'} \geq \ell n/k^3) \leq Mk^4 e^{-c \ell}/\ell^{2}.
 \end{equation}
 Hence 
 \begin{displaymath}
     \mathbb{P}(\max I_{\Gamma_n,\Gamma_n'} \geq \ell n \log k/k^3) \leq \frac{Mk^4e^{-c \ell \log k}}{\ell^{2}} \leq Ce^{-c \ell}
 \end{displaymath}
 as we have assumed $k$ and $\ell$ to be sufficiently large. This completes the proof of Lemma \ref{second_lemma}.\\
\end{proof} 
We state and prove Lemma \ref{intersection size lemma} now. Let us consider a line segment $\mathcal{J}$ of length $\frac{\ell n^{2/3}}{16k^2}$ (with $\ell <k^2$) on $\mathcal{L}_0$ with midpoint $\boldsymbol{0}$. Let $[v_1,v_2]$ (resp.\ $[v_3,v_4]$) be two line segments on $\mathcal{L}_{2n}$ with midpoints $\boldsymbol{n}$ (resp.\ $\boldsymbol{n}_k$) each of length $\ell^{1/3}n^{2/3}$. We consider the following event.
\begin{itemize}
    \item $\mathsf{Int}:=\{\exists \Gamma_1, \Gamma_2$ starting from $\mathcal{J}$ and ending at $[v_1,v_2]$ (resp.\ $[v_3,v_4]$) such that $|\Gamma_1 \cap \Gamma_2| \geq \frac{\ell n}{2k^3}$\}.
\end{itemize}
\begin{lemma}
\label{intersection size lemma}
    In the above setup there exist constants $C,c>0$ (depending on $\epsilon$) such that 
    \begin{displaymath}
        \mathbb{P}(\mathsf{Int}) \leq Ce^{-c \ell}.
    \end{displaymath}
\end{lemma}
In particular, as defined above we have for some constant $C>0$ and sufficiently large $n$ and $k$
\begin{displaymath}
    \mathbb{E}\left(\max I_{\Gamma_n,\Gamma_n'}\right) \leq \frac{C n \log k}{k^3},
\end{displaymath}
where the maximum is taken over all pair of geodesics starting from $A_n$ (resp.\ $B_n$) and ending at $A_n^*$ (resp.\ $B_n^*$).
\begin{proof}
Let $w_1,w_2$ be the end points of $\mathcal{J}$ with $\psi(w_1) \leq \psi(w_2).$ We consider the event $\mathsf{Int}$. Let $\Gamma_1,\Gamma_2$ be two geodesics starting from $\mathcal{J}$ and ending at $[v_1,v_2]$ (resp.\ $[v_3,v_4]$) such that $|\Gamma_1 \cap \Gamma_2| \geq \frac{\ell n}{2k^3}.$ Let $\Gamma_1$ and $\Gamma_2$ both intersect $\mathcal{L}_{\frac{\ell n}{2k^3}}$ at a point $u'$. By planarity 
 \begin{comment}of geodesic $\Gamma_{u,v}$ and $\Gamma_{u,v'}$ is sandwiched between $\Gamma_{w_1,v_1}$ and $\Gamma_{w_2,v_4}$ (see Figure \ref{fig: magnification}). Hence almost surely
 \end{comment}
 we have (see also Figure \ref{fig: subfigure})
 \begin{equation}
 \label{second_lemma_second_equation}
     \psi(\Gamma_{w_1,v_1}(\ell n/2k^3)) \leq \psi(u');
 \end{equation}
 \begin{equation}
 \label{second_lemma_third_equation}
     \psi(\Gamma_{w_2,v_4}(\ell n/2k^3)) \geq \psi(u').
 \end{equation}
 Consider the straight lines $\mathcal{J}_{w_1,v_1}$ (resp.\ $\mathcal{J}_{w_2,v_4}$) joining $w_1$ and $v_1$ (resp.\ $w_2$ and $v_4$). Further, let $v$ denote the intersection point of $\mathcal{J}_{w_1,v_1}$ and $\mathcal{J}_{w_2,v_4}$. As midpoint of $[v_1,v_2]$ and midpoint of $[v_3,v_4]$ are $kn^{2/3}$ distance apart from each other and each interval is of length $\ell^{1/3}n^{2/3}$, for sufficiently large $\ell$ we have 
 \begin{displaymath}
     \phi(v)\leq \frac{\ell n}{2k^3}.
 \end{displaymath} Hence, if  $\mathcal{J}_{w_1,v_1}$ (resp.\ $\mathcal{J}_{w_2,v_4}$) intersects $\mathcal{L}_{\frac{\ell n}{2 k^3}}$ at $w'$ (resp.\ $w''$), then $\psi(w')-\psi(w'') \geq K\ell n^{2/3}/k^2$ for some fixed constant $K$. Let $w$ be the midpoint of $[w',w'']$ (see Figure \ref{fig: subfigure}).
 We have 
 \begin{displaymath}
     \mathsf{Int} \subset (\mathsf{Int} \cap \{\psi(u') \leq \psi(w)\}) \cup (\mathsf{Int} \cap \{\psi(u') \geq \psi(w) \}).
 \end{displaymath}
 By (\ref{second_lemma_second_equation})
 \begin{equation}
 \label{second_lemma_forth_equation}
    \mathsf{Int} \cap \{\psi(u') \leq \psi(w)\} \subset \{\psi(\Gamma_{w_1,v_1}(\ell n/2k^3))-\psi(w') \leq -K\ell n^{2/3}/2k^2\}
 \end{equation}
 and by \eqref{second_lemma_third_equation}  
 \begin{equation}
     \label{second_lemma_fifth_equation}
     \mathsf{Int} \cap \{\psi(u') \geq \psi(w) \} \subset \{\psi(\Gamma_{w_2,v_4}(\ell n/2k^3))-\psi(w'') \geq K \ell n^{2/3}/2k^2\}.
 \end{equation}
 By Proposition \ref{transversal_fluctuation} we have both the events (\ref{second_lemma_forth_equation}) and (\ref{second_lemma_fifth_equation}) have probability upper bounded by $e^{-c \ell}$.
    \end{proof}
\subsection{Proof of Proposition \ref{general_upper_bound}(ii)}
\begin{figure}[t!]
    \includegraphics[width=12 cm]{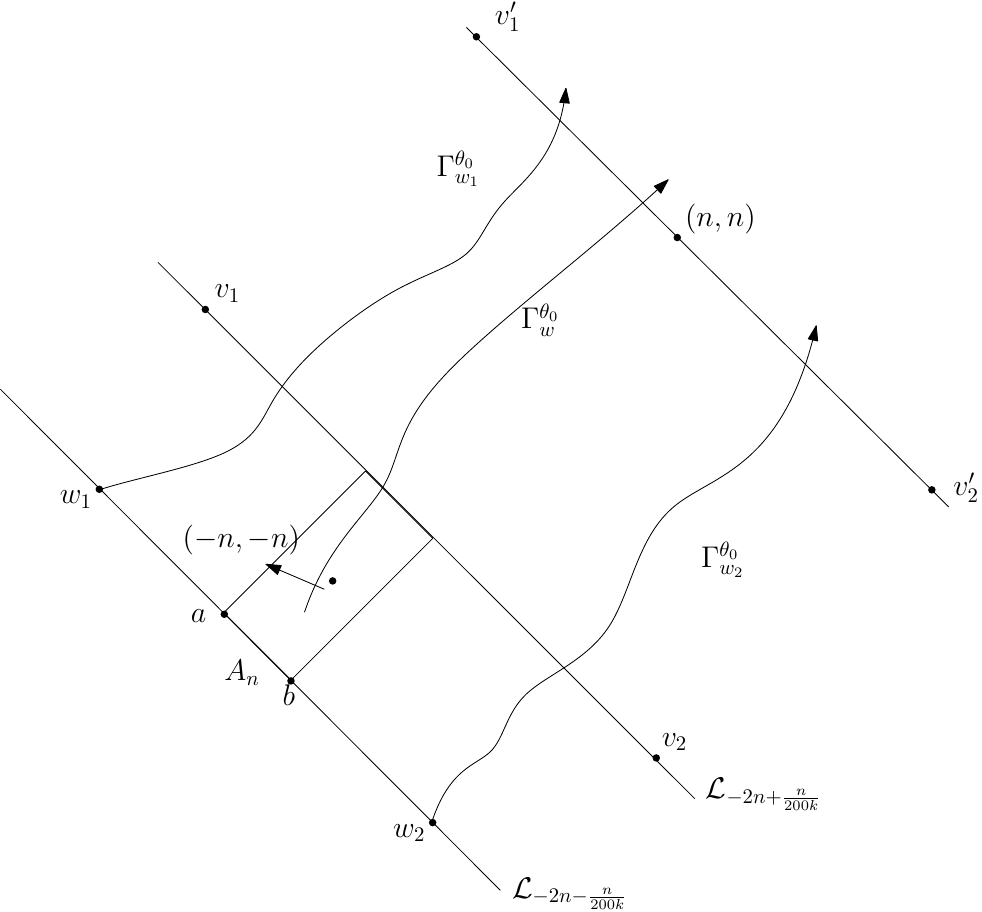}
    \caption{To prove Theorem \ref{main_theorem_1}(ii) we first apply a similar technique to $A_n$ and $B_n$ as in the proof of Lemma \ref{first_lemma} to restrict to the event that the geodesics from $A_n$ (resp. $B_n$) in the direction $\theta_0$ (resp. $\theta_k$) intersect $\mathcal{L}_{-2n+\frac{n}{200k}}$ (resp. $\mathcal{L}_{2n}$) on two extended line segments $[v_1,v_2]$ and $[v_1',v_2']$ (resp. $[v_3,v_4]$ and $[v_3',v_4']$). This reduces Theorem \ref{main_theorem_1}(2) to the finite case.}
    \label{fig: proof_for_semi_infinite}
\end{figure}
Proof for the semi-infinite case can be reduced to the finite setup using a transversal fluctuation argument. We briefly outline the idea here.\\
Same as before it suffices to prove for all $n$ sufficiently large there exists $C>0$ such that
\begin{displaymath}
     \mathbb{E}(\widehat{N^V}) \leq C n \log k/k^3,
\end{displaymath}where
\begin{comment}Similar to the Theorem \ref{main_theorem_1}(1) for $v \in V$ consider the semi-infinite geodesic $\Gamma^v$ (resp. $\Gamma^{'v}$) starting from $-\boldsymbol{n}+v$ (resp. $-\boldsymbol{n_k}+v$) in the direction $\theta_{0}$ (resp. $\theta_k$). Using translation invariance we have for all $v$ the following holds.
\begin{displaymath}
    \mathbb{P}(v \in \Gamma^v \cap \Gamma^{'v})=\mathbb{P}(\boldsymbol{0} \in \Gamma^{\boldsymbol{0}} \cap \Gamma^{\boldsymbol{0}'}).
\end{displaymath}
We have 
\begin{displaymath}
     \sum_{v \in V} \mathbb{P}(v \in \Gamma^v \cap \Gamma^{v'})=\mathbb{E} (\sum_{v \in V} \mathbbm{1}_{\{ v \in \Gamma^v \cap \Gamma^{v'} \} } ) \leq \mathbb{E}(\widehat{N^V}).
\end{displaymath}
\end{comment}
$\widehat{N^V}$ is the number of $v \in V$ such that there exists $w_1 \in A_n, w_2 \in B_n$ such that $\Gamma_{w_1}^{\theta_0}$ and $\Gamma_{w_2}^{\theta_k}$ intersect at $v$. 
\begin{comment}Arguing similarly as Theorem \ref{main_theorem_1} to prove Theorem \ref{main_theorem_1}(2) it suffices to show the following.
\begin{displaymath}
     \mathbb{E}(\widehat{N^V}) \leq C n \log k/k^3.
\end{displaymath}
\end{comment}
\\
We define $w_1,w_2$ on the line $\mathcal{L}_{-2n-\frac{n}{200k}}$ such that $\psi(w_1) \leq \psi(a)$ and $\psi(a)-\psi(w_1)=\ell^{1/32}(n/k)^{2/3}$ and $\psi(b )\leq \psi(w_2)$ and $\psi(w_2)-\psi(b)=\ell^{1/32}(n/k)^{2/3}$, where $a,b$ are the two bottom corners of $A_n$. 
\begin{comment}Consider two sequence of angles $\alpha_n=\pi/4+\delta_n$ and $\beta=\pi/4-\delta_n$ for some fixed $\delta>0$ small. 
\end{comment} 
Consider two semi-infinite geodesics $\Gamma_{w_1}^{\theta_0}$ and $\Gamma_{w_2}^{\theta_0}.$ We have the following events.
\begin{itemize}
    \item $\mathcal{B}_1:=\{\Gamma_{w_1}^{\theta_0}$ enters $A_n$\}.
    \item $\mathcal{B}_2:=\{\Gamma_{w_2}^{\theta_0}$ enters $A_n$\}.
    \item $\mathcal{B}:=\mathcal{B}_1 \cup \mathcal{B}_2$.
\end{itemize}
Further, if $c,d$ with $\psi(c) \leq \psi(d)$ denote the bottom corner points of $B_n$ then we define similarly points $w_3,w_4$ on the line $\mathcal{L}_{-2n-\frac{n}{200k}}$ such that $\psi(w_3) \leq \psi(c)$ and $\psi(c)-\psi(w_3)=\ell^{1/32}(n/k)^{2/3}$ and $\psi(d)\leq \psi(w_4)$ and $\psi(w_4)-\psi(b)=\ell^{1/32}(n/k)^{2/3}$. Similarly, if we define the following events
\begin{itemize}
    \item $\mathcal{B}'_1:=\{\Gamma_{w_3}^{\theta_k}$ enters $B_n$\}.
    \item $\mathcal{B}'_2:=\{\Gamma_{w_4}^{\theta_k}$ enters $B_n$\}.
    \item $\mathcal{B}':=\mathcal{B}'_1 \cup \mathcal{B}'_2$.
\end{itemize}
then by Proposition \ref{transversal_fluctuation_of_semi_infinite_geodesic}(ii) we have 
\begin{displaymath}
     \mathbb{P}(\mathcal{B} \cup \mathcal{B}') \leq 4C_1e^{-c_1\ell^{3/32}}.
\end{displaymath}
On the event $\mathcal{B}^{c}$, by planarity, for all $w \in A_n$ (resp.\ for $w \in B_n$), $\Gamma_{w}^{\theta_0}$ (resp.\ $\Gamma_w^{\theta_k}$) is sandwiched between $\Gamma_{w_1}^{\theta_0}$ and $\Gamma_{w_2}^{\theta_0}$ (resp.\ $\Gamma_{w_3}^{\theta_k}$ and $\Gamma_{w_4}^{\theta_k}$). Let us consider two line segments of length $4\ell^{1/32}n^{2/3}, [v_1,v_2]$ (resp.\ $[v_1',v_2']$) on $\mathcal{L}_{-2n+\frac{n}{200k}}$ (resp.\ $\mathcal{L}_{2n}$) with midpoints $-\boldsymbol{n+\frac{n}{400k}}$ (resp.\ $\boldsymbol{n}$). We do a similar calculation for $B_n$ to get two more line segments $[v_3,v_4]$ (resp.\ $[v_3',v_4']$) on $\mathcal{L}_{-2n+\frac{n}{200k}}$ (resp. $\mathcal{L}_{2n}$) (see Figure \ref{fig: proof_for_semi_infinite}). Let $\mathcal{C}$ (resp.\ $\mathcal{C}'$) denote the event that for any $w \in A_n$ (resp.\ $w \in B_n$) $\Gamma_w^{\theta_0}$ (resp.\ $\Gamma_w^{\theta_k}$) do not intersect $\mathcal{L}_{-2n+\frac{n}{200k}}$ and $\mathcal{L}_{2n}$ outside $[v_1,v_2]$ and $[v_1',v_2']$ (resp.\ $[v_3,v_4]$ and $[v_3',v_4']$). Using Proposition \ref{transversal_fluctuation_of_semi_infinite_geodesic} again we have 
\begin{displaymath}
    \mathbb{P}(\mathcal{C} \cup \mathcal{C}') \leq 8C_1e^{-c_1\ell^{3/32}}.
\end{displaymath}
Hence, we can restrict ourselves to the event $(\mathcal{B} \cup \mathcal{B}' \cup \mathcal{C} \cup \mathcal{C}')^c$ . But note that now we are in the finite setup. We consider $M_n^A$ (resp.\ $M_n^B$) as defined in Proposition \ref{coalescence_theorem} for the line segments $[v_1,v_2]$ and $[v_1',v_2']$ (resp.\ $[v_3,v_4]$ and $[v_3',v_4']$). Also consider for all $w_1 \in A_n, w_2 \in B_n, I_{\Gamma_{w_1}^{\theta_0},\Gamma_{w_2}^{\theta_k}}$ as defined in the Lemma \ref{second_lemma}. We have 
\begin{displaymath}
     \widehat{N^V} \leq M_n^A M_n^B \max I_{\Gamma_{w_1}^{\theta_0},\Gamma_{w_2}^{\theta_k}}.
\end{displaymath}
Now, arguing as Proposition \ref{general_upper_bound}(i) we get there exists $C>0$ such that for all sufficiently large $n $
\begin{displaymath}
     \mathbb{E}(\widehat{N^V}) \leq C n \log k/k^3.
\end{displaymath}
This proves Proposition \ref{general_upper_bound}(ii).

\section{Lower bound for finite geodesic intersections}
%Proof of Theorem \ref{main_theorem_1} Lower Bounds}
\label{proof_for_lower_bound}
This section is divided into following subsections. As mentioned earlier we will prove Theorem \ref{main_theorem_1}(i) lower bound using Proposition \ref{positive_probability_for_point_to_line}. In the first subsection we will prove Theorem \ref{main_theorem_1}(i) lower bound assuming Proposition \ref{positive_probability_for_point_to_line}. In the second subsection we will 
We will prove Proposition \ref{positive_probability_for_point_to_line}. Proposition \ref{small_probability_lemma} will be proved in the last subsection.
\begin{comment}Similar to Theorem \ref{main_theorem_1} to prove Theorem \ref{lower_bound_theorem_1} also we apply an averaging argument. We construct rectangles around $-\boldsymbol{n},-\boldsymbol{n_k},\boldsymbol{0},\boldsymbol{n},\boldsymbol{n_k}$ as before, but this time with extended line segments. Specifically, we define the following rectangles. \\
\end{comment}
\subsection{Proof of Theorem \ref{main_theorem_1}(i) lower bound}We fix $M,k$ large enough (these will be chosen later). We define the following rectangles (see also Figure \ref{fig: proof_for_lower_bound}).
\begin{itemize}
  \item  $V_{M}:=\{v=(v_1,v_2) \in \mathbb{Z}^2: |\psi(v)|<\frac{M}{2}n^{2/3}$ and $|\phi(v)|< \frac{n}{100k}\}$.
    \item $A_{M}:=\{-\boldsymbol{n}+v : v \in V_{M}\}$,
    \item $B_{M}:=\{-\boldsymbol{n}_k+v : v \in V_{M}\}$,
    \item $A_{M}^*:=\{\boldsymbol{n}+v: v \in V_{M}\}$,
    \item $B_{M}^*:=\{\boldsymbol{n}_k+v : v \in V_{M}\}$. 
    \end{itemize}
    \begin{comment}The idea is to construct a positive probability event $\mathcal{E}$, where $\mathcal{E}$ is the event that all geodesics starting from $\widehat{A_n}$ (resp. $\widehat{B_n}$) and ending at $\widehat{A_n^*}$ (resp. $\widehat{B_n^*}$) intersect at a single point inside $\widehat{V_{\widetilde{M}}}$. The precise definition of the event $\mathcal{E}$ is the following.
    \end{comment}
    Let us define the following events.
\begin{itemize}
  \item $\mathcal{E}_1:=\{\Gamma_{a_{1},a_1'}(t)=\Gamma_{a_{2},a_2'}(t), \forall a_1,a_2 \in A_{M}, a_1',a_2' \in A_{M}^*, \forall t \in [-\frac{n}{100k},\frac{n}{100k}] \cap \{\phi(\Gamma_{a_{1},a_1'}(s))\}$\};
   \item $\mathcal{E}_2:=\{\Gamma_{b_1,b_1'}(t)=\Gamma_{b_2,b_2'}(t), \forall b_1,b_2 \in B_{M}, b_1',b_2' \in B_{M}^*, \forall t \in [-\frac{n}{100k},\frac{n}{100k}] \cap \{\phi(\Gamma_{b_{1},b_1'}(s))\}\}$;
    \item $\mathcal{E}_3$:= \{$\Gamma_{a,a'} \cap \Gamma_{b,b'} \subset V_M, \forall a \in A_{M}, b \in B_{M}, a' \in A_{M}^*, b' \in B_{M}^*$\}.
    \item $\mathcal{E}=\mathcal{E}_1 \cap \mathcal{E}_2 \cap \mathcal{E}_3.$
\end{itemize}
Using Proposition \ref{positive_probability_for_point_to_line} we have the following Lemma.
\begin{lemma}
\label{positive_probability_lemma}
In the above setup there exist $c>0, n_0 \in \mathbb{N}$ (depending only on the choice of $\epsilon$) such that for all $n \geq n_0$ we have 
\begin{displaymath}
     \mathbb{P}(\mathcal{E}) \geq c.
\end{displaymath}
\end{lemma}
\begin{comment}\begin{figure}[h!]
    \includegraphics[width=10 cm]{Figure_12.pdf}
    \caption{To prove Lemma \ref{positive_probability_lemma} we consider line segments $P_n,P_n^*,Q_n,Q_n^*$ of length $2Mn^{2/3}$. We can chose $k$ large enough so that using planarity and transversal fluctuation the event $\mathcal{T}^c$ will have probability less than $Ce^{-cM^{12}}.$ The event $\mathcal{T}$ ensures that the geodesics don't go outside the corresponding line segments $P_n,P_n^*,Q_n,Q_n^*$. Now, we can use Lemma \ref{line_to_line_positive_probability_lemma} on the line segments $P_n,P_n^*,Q_n,Q_n^*$ to find a lower bound for $\mathcal{E}$.}
    \label{fig: proof_for_lower_bound} 
\end{figure}
\end{comment}
We already had
\begin{displaymath}
     \sum_{v \in V_{M}} \mathbb{P}(v \in \Gamma_n^v \cap \Gamma_{n_k}^{v})=\mathbb{E} \left(\sum_{v \in V_{M}} \mathbbm{1}_{\{ v \in \Gamma_n^v \cap \Gamma_{n_k}^{v} \} } \right).
\end{displaymath}
Note that on the event $\mathcal{E}$, any geodesic starting from $A_n$ (resp.\ $B_n$) and ending at $A_n^*$ (resp.\ $B_n^*$) will coincide inside $V_M$ and also their intersection lie inside $V_M$. So, there is at least one $v_0 \in V_{M}$ such that $v_0 \in \Gamma_n^{v_0} \cap \Gamma_{n_k}^{v_0}$. Hence on the event $\mathcal{E}$ we have 
\begin{displaymath}
      \sum_{v \in V_{M}} \mathbbm{1}_{\{ v \in \Gamma_n^v \cap \Gamma_{n_k}^{v} \} } \geq 1.
\end{displaymath}
By Lemma \ref{positive_probability_lemma} we have 
\begin{displaymath}
     \mathbb{E} \left(\sum_{v \in V_{M}} \mathbbm{1}_{\{ v \in \Gamma_n^v \cap \Gamma_{n_k}^{v} \} } \right) \geq c.
\end{displaymath}
This in turn implies 
\begin{displaymath}
      \frac{Mn^{5/3}}{100k} \mathbb{P}(\boldsymbol{0} \in \Gamma_n^{\boldsymbol{0}} \cap \Gamma_{n_k}^{\boldsymbol{0}}) \geq c.
\end{displaymath}
This completes the proof of Theorem \ref{main_theorem_1}(i) lower bound. \qed
\\We only need to prove Lemma \ref{positive_probability_lemma}.
\begin{proof}[Proof of Lemma \ref{positive_probability_lemma}]
Consider the line segments $P_n$ (resp.\ $P_n^*$) on $\mathcal{L}_{-2n+\frac{2n}{100k}}$ (resp.\ $\mathcal{L}_{2n-\frac{2n}{100k}}$) of length $2Mn^{2/3}$ with midpoints $a$ (resp.\ $a'$), where $a$ (resp.\ $a'$) are the vertices on $\mathcal{L}_{-2n+\frac{n}{100k}}$ (resp.\ $\mathcal{L}_{2n-\frac{n}{100k}}$) with $\psi(a)=0$ (resp.\ $\psi(a')=0$). Similarly, we define two more line segments $Q_n$ (resp.\ $Q_n^*$)  on $\mathcal{L}_{-2n+\frac{2n}{100k}}$ (resp.\ $\mathcal{L}_{2n-\frac{2n}{100k}}$) of length $2Mn^{2/3}$ with midpoints $b$ (resp.\ $b'$), where $b$ (resp.\ $b'$) are the vertices on $\mathcal{L}_{-2n+\frac{n}{100k}}$ (resp.\ $\mathcal{L}_{2n-\frac{n}{100k}}$) with $\psi(b)=\psi(-\boldsymbol{n}_k)$ (resp.\ $\psi(a')=\psi(\boldsymbol{n}_k)$) (see Figure \ref{fig: proof_for_lower_bound}). We consider the following events.
\begin{itemize}
    \item $\widehat{\mathcal{E}_1}:=\{\Gamma_{a_1,a_1'}(t)=\Gamma_{a_2,a_2'}(t), \forall a_1,a_2 \in P_n, a_1',a_2' \in P_n^*, \forall t \in [-\frac{n}{200k},\frac{n}{200k}]$\};
    \item $\widehat{\mathcal{E}_2}:=\{\Gamma_{b_1,b_1'}(t)=\Gamma_{b_2,b_2'}(t), \forall b_1,b_2 \in Q_n, b_1',b_2' \in Q_n^*, \forall t \in [-\frac{n}{200k},\frac{n}{200k}]$\};
   \item $\widehat{\mathcal{E}_3} := \{\Gamma_{a,a'}\cap \Gamma_{b,b'} \subset V_{M}, \forall a \in P_n, b \in Q_n, a' \in P_n^*, b' \in Q_n^*$\}.
    \item $\widehat{\mathcal{E}}=\widehat{\mathcal{E}_1} \cap \widehat{\mathcal{E}_2} \cap \widehat{\mathcal{E}_3}.$
    \end{itemize}
    Using Proposition \ref{positive_probability_for_point_to_line} we get that in the above setup there exist sufficiently large $M$ (depending on $\epsilon$) and $c_2>0, n_0 \in \mathbb{N}$ (depending on $M$) such that for all $n \geq n_0$,
\begin{equation}
\label{line_to_line_equation}
     \mathbb{P}(\widehat{\mathcal{E}}) \geq c_2.
\end{equation}
\begin{figure}[t!]
    \includegraphics[width=12 cm]{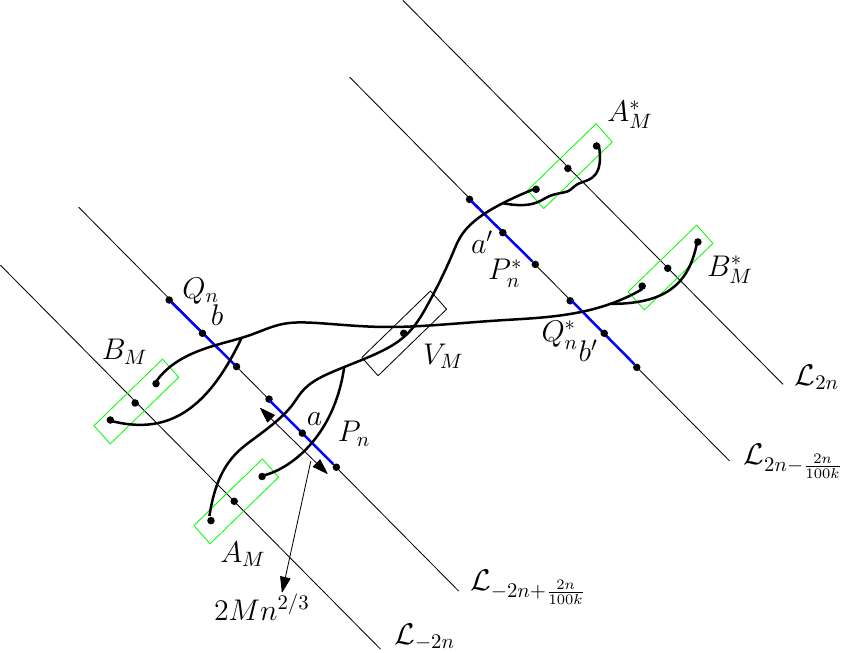}
    \caption{To prove Lemma \ref{positive_probability_lemma} we consider line segments $P_n,P_n^*,Q_n,Q_n^*$ of length $2Mn^{2/3}$. We can chose $k$ large enough so that using planarity and transversal fluctuation the event $\mathsf{TF}^c$ will have probability less than $c_2$, where $c_2$ is obtained from \eqref{line_to_line_equation}. The event $\mathsf{TF}$ ensures that the geodesics don't go outside the corresponding line segments $P_n,P_n^*,Q_n,Q_n^*$. Now, we can use \eqref{line_to_line_equation} on the line segments $P_n,P_n^*,Q_n,Q_n^*$ to find a lower bound for $\mathcal{E}$.}
    \label{fig: proof_for_lower_bound} 
\end{figure}
Fix $M$ so that \eqref{line_to_line_equation} holds. We chose $k$ large enough  so that $16C_1e^{-c_1k^2M^3} < c_2$, where $C_1,c_1$ are the constants obtained from Proposition \ref{transversal_fluctuation}. We apply a similar technique as we did in the proof of Lemma \ref{first_lemma}.
\begin{comment}Consider two points $a,b$ on $\mathcal{L}_{-2n+\frac{2n}{k}}$ such that $\psi(a)=\psi(-\boldsymbol{n}), \psi(b)=\psi(-\boldsymbol{n}_k)$. Similarly, we define $a',b'$ on $\mathcal{L}_{2n-\frac{2n}{k}}$ such that $\psi(a')=\psi(\boldsymbol{n})$ and $\psi(b')=\psi(\boldsymbol{n}_k).$ Now we consider 2 line segments $P_n,Q_n$ of length $2Mn^{2/3}$ on $\mathcal{L}_{-2n+2\frac{n}{k}}$ with mid-points $a,b$ respectively. Similarly, we define line segments $P_n^*,Q_n^*$ of length $2Mn^{2/3}$ on $\mathcal{L}_{2n-2\frac{n}{k}}$ with mid-points $a'$ and $b'$.
\end{comment}
We want to apply \eqref{line_to_line_equation} to the line segments $P_n,Q_n,P_n^*,Q_n^*$. First we define the following event (see Figure \ref{fig: proof_for_lower_bound}).
\begin{itemize}
    \item $\mathsf{TF}$ is the event that any geodesic $\Gamma$ starting from $A_{M}$ and ending at $A_{M}^*$ will have $|\psi(\Gamma(-2n+\frac{2n}{100k}))-\psi(a)| \leq Mn^{2/3}, |\psi(\Gamma(2n-\frac{2n}{100k}))-\psi(a')| \leq Mn^{2/3}$ and any geodesic $\Gamma'$ starting from $B_{M}$ and ending at $B_{M}^*$ will have $|\psi(\Gamma'(-2n+\frac{2n}{100k}))-\psi(b)| \leq Mn^{2/3}, |\psi(\Gamma'(2n-\frac{2n}{100k}))-\psi(b')| \leq Mn^{2/3}.$
\end{itemize}
Note that we can bound $\mathsf{TF}^c$ similar to the proof of Lemma \ref{first_lemma}. i.e., we take 8  deterministic points away from $A_{M},A_{M}^*,B_{M},B_{M}^*$ and consider the geodesics between them as before. Now using planarity and transversal fluctuation of these geodesics we can bound $\mathbb{P}(\mathsf{TF}^c)$ (see Figure \ref{fig: box to box}). In particular, we have 
\begin{equation}
\label{positive_probability_first_equation}
    \mathbb{P}(\mathsf{TF}^c) \leq 16C_1e^{-c_1k^2M^3} < c_2.
\end{equation}
\begin{comment}Now consider the event $\mathcal{E}^0$ corresponding to  the line segments $P_n,P_n^*,Q_n,Q_n^*$ as defined in \eqref{line_to_line_equation}.
\end{comment}
Note that on the event $\mathsf{TF}$ any geodesic starting from $A_{M}$ (resp.\ $B_{M}$) and ending at $A_{M}^*$ (resp.\ $B_{M}^*$) will not go outside the line segments $P_n,P_n^*$ (resp.\ $Q_n,Q_n^*$). Further, on the event $\widehat{\mathcal{E}}$ any geodesic starting from $P_n$ (resp.\ $Q_n$) and ending at $P_n^*$ (resp.\ $Q_n^*$) will coincide inside $V_M$ and also the intersection lies inside $V_{M}$. Hence, we have
\begin{displaymath}
    \widehat{\mathcal{E}} \cap \mathsf{TF} \subset \mathcal{E}.
\end{displaymath}
Now, using \eqref{positive_probability_first_equation}
\begin{displaymath}
    \mathbb{P}(\widehat{\mathcal{E}} \cap \mathsf{TF}) \geq c_2 - \mathbb{P}(\mathsf{TF}^c)=\tilde{c}>0.
\end{displaymath}
This completes the proof of Lemma \ref{positive_probability_lemma}.
\end{proof}
\subsection{Proof of Proposition \ref{positive_probability_for_point_to_line}}
Proof of this proposition will be a modification of the proof of \cite[Proposition 1.2]{BB21}. Recall the definitions of the variants of last passage times as defined in the notations. Also, here we state an important inequality regarding expected last passage time which we will use frequently in the proof below (see \cite[Theorem 2]{LR10}).
For each $\delta >0$, there exists a positive constant $C_2$ (depending only on $\delta$) such that for all $m,n$ sufficiently large with $\delta < \frac{m}{n} < \delta^{-1}$ we have
\begin{equation}
\label{expected_passage_time_estimate}
      |\mathbb{E}T_{0,(m,n)}-(\sqrt{m}+\sqrt{n})^2| \leq C_2n^{1/3}.
\end{equation}
$(\sqrt{m}+\sqrt{n})^2$ will be called the \textit{time constant} in the direction $(m,n)$. The above inequality also holds for passage times when defined including the initial vertex.\\ To prove Proposition \ref{positive_probability_for_point_to_line} we first fix the following notations. Let $u_1,u_2$ (resp.\ $u_1',u_2'$) denote the end points of $P_n$ (resp.\ $P_n^*$). Similarly, let $w_1,w_2$ (resp.\ $w_1',w_2'$) denote the end points of $Q_n$ (resp.\ $Q_n^*$) (see Figure \ref{fig: favourable_environment}). We construct some large probability events. First consider the following parallelograms. For $\Delta>0$ define $U_n^\Delta$ to be the parallelogram whose two pairs of opposite sides are line segments of length $\Delta n^{2/3}$ on $\mathcal{L}_{-2n}$ (resp.\ $\mathcal{L}_{2n}$) with midpoints $-\boldsymbol{n}$ (resp.\ $\boldsymbol{n}$). We similarly define the parallelogram $U_{n_k}^\Delta$ with $\boldsymbol{n}$ replaced by $\boldsymbol{n}_k$ (see Figure \ref{fig: favourable_environment}). We will also define similar parallelograms around different midpoints $u \in \mathcal{L}_{-2n}, v \in \mathcal{L}_{2n}$. To denote such parallelograms we will use the notation $U_{u,v}^{\Delta}$. We will frequently use these parallelograms with different values of $\Delta$.\\
We define a collection of large probability events. These events will ensure that on them the corresponding geodesics do not have large transversal fluctuation. For $u \in P_n$ let us define the following events. 
\begin{itemize}
    \item $\boldsymbol{\mathrm{Restr}}_{{u},\Delta}:=$\{All $\gamma:u \rightarrow v$ for some $v \in P_n^*$ and satisfying $\gamma \cap \left(U_n^{\Delta} \right)^c \neq \emptyset$ have $\ell(\gamma) \leq \mathbb{E}(T_{u,v})-c_3\Delta^2n^{1/3}$\}.
    \end{itemize}
    Further, define 
    \begin{itemize}
    \item $\boldsymbol{\mathrm{Restr}}_{n,\Delta}:=$\{All $\gamma:-\boldsymbol{n} \rightarrow \boldsymbol{n}$ satisfying $\gamma \cap \left(U_n^{\Delta} \right)^c \neq \emptyset$ have $\ell(\gamma) \leq \mathbb{E}(T_{-\boldsymbol{n},\boldsymbol{n}})-c_3\Delta^2n^{1/3}\}$.
    \end{itemize}
    Also, for all $u \in Q_n$ we define 
    \begin{comment}
    \item $\boldsymbol{\mathrm{Restr}}_{{u_2},2M^2}:=$\{All $\gamma: u_2 \rightarrow \mathcal{L}_{2n}$ satisfying $\gamma \cap \left (U_n^{2M^2} \right )^c \neq \emptyset$ have $\ell(\gamma) \leq \mathbb{E}(T_{u_2,u_2'})-c_1M^4n^{1/3}$\};
    \item $\boldsymbol{\mathrm{Restr}}_{n,2M^{7/8}}:=$\{All $\gamma: -\boldsymbol{n} \rightarrow \mathcal{L}_{2n}$ satisfying $\gamma \cap \left (U_{n}^{2M^{7/8}} \right)^c \neq \emptyset$ have $\ell(\gamma) \leq \mathbb{E}(T_{-\boldsymbol{n},\boldsymbol{n}})-c_1M^{7/4}n^{1/3}$\};
    \end{comment}
    \begin{itemize}
    \item $\boldsymbol{\mathrm{Restr}}_{{u},\Delta}^k:=$\{All $\gamma: u \rightarrow v$ for some $v \in Q_n^*$ and satisfying $\gamma \cap \left (U_{n_k}^{\Delta} \right )^c \neq \emptyset$ have $\ell(\gamma) \leq \mathbb{E}(T_{u,v})-c_3\Delta^2n^{1/3}$\}.
    \end{itemize}
    Further, define
    \begin{itemize}
    \item $\boldsymbol{\mathrm{Restr}}^k_{n,\Delta}:=$\{All $\gamma:-\boldsymbol{n}_k \rightarrow \boldsymbol{n}_k$ satisfying $\gamma \cap \left(U_{n_k}^{\Delta} \right)^c \neq \emptyset$ have $\ell(\gamma) \leq \mathbb{E}(T_{-\boldsymbol{n}_k,\boldsymbol{n}_k})-c_3\Delta^2n^{1/3}\}$.
    \end{itemize}
    
    \begin{comment}
    \item $\boldsymbol{\mathrm{Restr}}_{{w_2},2M^2}:=$\{All $\gamma: w_2 \rightarrow \mathcal{L}_{2n}$ satisfying $\gamma \cap \left (U_{n_k}^{2M^{2}} \right )^c \neq \emptyset$ have $\ell(\gamma) \leq \mathbb{E}(T_{w_2,w_2'})-c_1M^4n^{1/3}$\};
    \item $\boldsymbol{\mathrm{Restr}}_{n_k,2M^{7/8}}:=$\{All $\gamma: -\boldsymbol{n}_k \rightarrow \mathcal{L}_{2n}$ satisfying $\gamma \cap \left (U_{n_k}^{2M^{7/8}} \right)^c \neq \emptyset$ have $\ell(\gamma) \leq \mathbb{E}(T_{-\boldsymbol{n}_k,\boldsymbol{n}_k})-c_1M^{7/4}n^{1/3}$\};
    \end{comment}
    Finally, define (see Figure \ref{fig: favourable_environment})\\
$\boldsymbol{\mathrm{Restr}}:=\boldsymbol{\mathrm{Restr}}_{{u_1},2M^2} \cap \boldsymbol{\mathrm{Restr}}_{{u_2},2M^2} \cap \boldsymbol{\mathrm{Restr}}_{n,2M^{7/8}} \cap \boldsymbol{\mathrm{Restr}}_{{w_1},2M^2}^k \cap \boldsymbol{\mathrm{Restr}}_{{w_2},2M^2}^k \cap \boldsymbol{\mathrm{Restr}}_{n,2M^{7/8}}^k.$

\cite[Proposition C.8]{BGZ21} says that there exists a constant $c_3$ (depending on $\epsilon$) such that for sufficiently large $\Delta$ (depending on $\epsilon$) all the above events are large probability events. For a fixed $\delta>0$ (to be chosen later) we consider the following parallelograms (see Figure \ref{fig: favourable_environment}).
\begin{itemize}
      \item $R_1:=U_n^{2M^2} \cap \{v \in \mathbb{Z}^2: -2n+\frac{\delta}{4} n \leq \phi(v) \leq -2n+\frac{3\delta}{4}n\}$;
      \item $R_2:=U_n^{2M^2} \cap \{v \in \mathbb{Z}^2: 2n-\frac{3\delta}{4} n \leq \phi(v) \leq 2n-\frac{\delta}{4} n\}$;
      \item $R_1^{k}:=U_{n_k}^{2M^{2}} \cap \{v \in \mathbb{Z}^2: -2n+\frac{\delta}{4} n \leq \phi(v) \leq -2n+\frac{3\delta}{4}n\}$;
      \item $R_2^{k}:=U_{n_k}^{2M^{2}} \cap \{v \in \mathbb{Z}^2: 2n-\frac{3\delta}{4} n \leq \phi(v) \leq 2n-\frac{\delta}{4} n\}$.
\end{itemize}
These parallelograms will act as ``barriers" (see the grey shaded regions in Figure \ref{fig: favourable_environment}). We will chose $M$ sufficiently large in the proof. We also chose $k$ sufficiently large and $\delta>0$ sufficiently small so that the above defined parallelograms are disjoint from each other. This is an important fact, as we will use this to show certain events are independent of each other.\\
Define the following line segments which are boundaries of the barrier regions (see also Figure \ref{fig: favourable_environment}).
\begin{itemize}
    \item $\underline{R_1}:=R_1 \cap \mathcal{L}_{-2n+\frac{\delta n}{4}}$;
    \item $\overline{R_1}:=R_1 \cap \mathcal{L}_{-2n+\frac{3\delta n}{4}}$;
    \item $\underline{R_2}:=R_2 \cap \mathcal{L}_{2n-\frac{3\delta n}{4}}$;
    \item $\overline{R_2}:=R_2 \cap \mathcal{L}_{2n-\frac{\delta n}{4}}$;
    \item $\underline{R_1^{k}}:=R_1^{k} \cap \mathcal{L}_{-2n+\frac{\delta n}{4}}$;
    \item $\overline{R_1^{k}}:=R_1^{k} \cap \mathcal{L}_{-2n+\frac{3\delta n}{4}}$;
    \item $\underline{R_2^{k}}:=R_2^{k} \cap \mathcal{L}_{2n-\frac{3\delta n}{4}}$;
    \item $\overline{R_2^{k}}:=R_2^{k} \cap \mathcal{L}_{2n-\frac{\delta n}{4}}$.
\end{itemize}
\begin{figure}[t!]
    \includegraphics[width=15 cm]{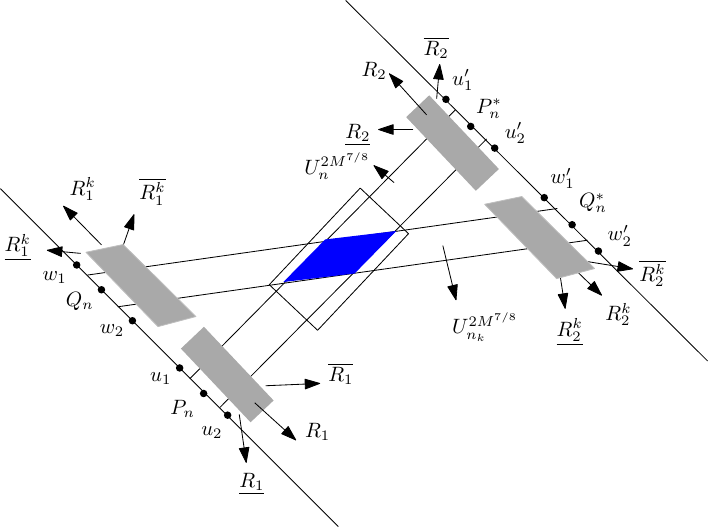}
    \caption{As a preparation for proving Proposition \ref{positive_probability_for_point_to_line} we first construct some large probability events. The events $\mathcal{H}, \mathcal{I}, \widetilde{\mathcal{I}}$ are about typical weights on the paths inside the parallelograms. The event $\boldsymbol{\mathrm{Restr}}$ ensures typical transversal fluctuation of geodesics. The parallelograms $R_1, R_2, R_1^{k}, 
    R_2^{k}$ will act as barriers. In these barriers we can decrease weight of paths with positive probability to force geodesics to coincide on a single path. The barriers and the intersection of the two parallelograms are coloured in the figure.}
    \label{fig: favourable_environment} 
\end{figure}
We define some more large probability events. 
\begin{itemize}
    \item $\mathcal{H}_1:=\{ \forall u \in \mathcal{L}_{-2n} \cap U_n^{2M^2}$ and $ \forall v \in \underline{R_1}$, $|\widetilde{T_{u,v}}| \leq Mn^{1/3}$\};
    \item $\mathcal{H}_2:=\{ \forall u \in \overline{R_1}$ and $\forall v \in \underline{R_2}, |\widetilde{\underline{T_{u,v}}}| \leq Mn^{1/3}\}$;
    \item $\mathcal{H}_3:= \{\forall u \in \overline{R_2}$ and $\forall v \in \mathcal{L}_{2n} \cap U_n^{2M^2}, |\widetilde{\underline{T_{u,v}}}| \leq Mn^{1/3}$\};
    \item $\mathcal{H}_1^k:=\{ \forall u \in \mathcal{L}_{-2n} \cap U_{n_k}^{2M^2}$ and $\forall v \in \underline{R_1^{k}}, |\widetilde{T_{u,v}}| \leq Mn^{1/3}$\};
    \item  $\mathcal{H}_2^{k}:=\{ \forall u \in \overline{R_1^{k}}$ and $\forall v \in \underline{R_2^{k}}, |\widetilde{\underline{T_{u,v}}}| \leq Mn^{1/3}$\};
    \item $\mathcal{H}_3^{k}:=\{ \forall u \in \overline{R_2^{k}}$ and $\forall v \in \mathcal{L}_{2n} \cap U_{n_k}^{2M^2}, |\widetilde{\underline{T_{u,v}}}| \leq Mn^{1/3}$\}.
\end{itemize}
The above events ensure typical centred last passage time (note that here we have used both variants of the last passage time) on the rectangles $U_n^{2M^2}$ and $U_{n_k}^{2M^2}$ except the barrier regions. We exclude the barrier regions for the following reason. All the events defined above does not depend on the vertex weights inside the barriers. This will be required later while we condition on the above events. For the barrier regions we consider the following events which ensure centred passage time is typical in the barrier regions. 
\begin{itemize}
    \item $\mathcal{I}_1:=\{ \forall u \in \underline{R_1}$ and $\forall v \in \overline{R_1},|\widetilde{\underline{T_{u,v}}}| \leq Mn^{1/3}$\};
    \item $\mathcal{I}_2:=\{u \in \underline{R_2}$ and $\forall v \in \overline{R_2}, |\widetilde{\underline{T_{u,v}}}| \leq Mn^{1/3}$\};
    \item $\mathcal{I}_1^{k}:=\{\forall u \in \underline{R_1^{k}}$ and $ \forall v \in \overline{R_1^{k}} , |\widetilde{\underline{T_{u,v}}}| \leq Mn^{1/3}$\};
    \item $\mathcal{I}_2^{k}:=\{ \forall u \in \underline{R_2^{k}}$ and $\forall v \in \overline{R_2^{k}},|\widetilde{\underline{T_{u,v}}}| \leq Mn^{1/3}$\};
\end{itemize}
We require some more large probability events which are about typical constrained passage time in the barriers. The only difference is the following are decreasing events.
\begin{itemize}
    \item $\widetilde{\mathcal{I}_1}:=\{u \in \underline{R_1}$ and $\forall v \in \overline{R_1},\widetilde{\underline{T_{u,v}}} \leq Mn^{1/3}$\};
    \item $\widetilde{\mathcal{I}_2}:= \{ \forall u \in \underline{R_2}$ and $\forall v \in \overline{R_2}, \widetilde{\underline{T_{u,v}}} \leq Mn^{1/3}$\}.
    \item $\widetilde{\mathcal{I}_1^{k}}:=\{ \forall u \in \underline{R_1^{k}}$ and $\forall v \in \overline{R_1^{k}}, \widetilde{\underline{T_{u,v}}} \leq Mn^{1/3}$\};
    \item $\widetilde{\mathcal{I}_2^{k}}:=\{ \forall u \in \underline{R_2^{k}}$ and $\forall v \in \overline{R_2^{k}}$ we have $\widetilde{\underline{T_{u,v}}} \leq Mn^{1/3}$\};
\end{itemize}
We define
\begin{displaymath}
\mathcal{H} :=\mathcal{H}_1 \cap \mathcal{H}_2 \cap \mathcal{H}_3
\cap \mathcal{H}_1^{k} \cap \mathcal{H}_2^{k} \cap \mathcal{H}_3^{k};
\end{displaymath}
\begin{displaymath}
\mathcal{I}:=\mathcal{I}_1 \cap \mathcal{I}_2
\cap \mathcal{I}_1^{k} \cap \mathcal{I}_2^{k};
\end{displaymath}
\begin{displaymath}
\widetilde{\mathcal{I}}:=\widetilde{\mathcal{I}_1} \cap \widetilde{\mathcal{I}_2} 
\cap \widetilde{\mathcal{I}_1^{k}} \cap 
\widetilde{\mathcal{I}_2^{k}}.
\end{displaymath}
We prove the following lemma.
\begin{lemma}
\label{restr_large_probability_lemma}
For all sufficiently large $M$ (depending on $\epsilon$) and sufficiently large $n$ (depending $M$) we have
\begin{displaymath}
\mathbb{P}(\boldsymbol{\mathrm{Restr}}) \geq 0.99.
\end{displaymath}
\end{lemma}
\begin{proof}
We will apply \cite[Proposition C.8]{BGZ21}. We will show 
\begin{equation}
\label{restr_lower_bound_equation}
\mathbb{P} \left(\boldsymbol{\mathrm{Restr}}_{{u_1},2M^2} \right) \geq 0.998.
\end{equation}
For the rest of the events we can get the same lower bound and then taking a union bound gives us the desired result for $\boldsymbol{\mathrm{Restr}}$. Therefore, we will just show \eqref{restr_lower_bound_equation}. We divide $P_n^*$ into line segments of length $2n^{2/3}$ and apply \cite[Proposition C.8]{BGZ21} to each of these line segments and then apply a union bound. So, we consider the setup of \cite[Proposition C.8]{BGZ21} for $\psi=1-\epsilon/2$ and choose $M$ large enough so that it satisfies $Me^{-c_4 (M^2-M)^3} \leq 0.002$ ($c_4$ is the constant $c_2$ in the statement of \cite[Proposition C.8]{BGZ21}). Then \cite[Proposition C.8]{BGZ21} implies for sufficiently large $n$
\begin{displaymath}
\mathbb{P}\left(\left(\boldsymbol{\mathrm{Restr}}_{{u_1},2M^2}\right)^c \right) \leq 0.002.
\end{displaymath}
This implies \eqref{restr_lower_bound_equation}.
\end{proof}
Using \cite[Theorem 4.2]{BGZ21} we now have the following lemma.
\begin{lemma}
\label{H_I_are_large_probability_events}
For all sufficiently large $M$ (depending on $\epsilon$) and sufficiently large $n$ (depending on $M$) we have $\mathbb{P}(\mathcal{H}) \geq 0.99$ and $\mathbb{P}(\mathcal{I}) \geq 0.99, \mathbb{P}(\widetilde{\mathcal{I}}) \geq 0.99$.
\end{lemma}
\begin{proof}
We will show $\mathcal{H}$ is intersection of large probability events $\mathcal{H}_i$'s and $\mathcal{H}_i^k$'s. Then same as before taking a union bound we conclude the lemma.\\ 
For $\mathcal{H}_1$, we divide $\mathcal{L}_{-2n} \cap U_{n}^{2M^2}$ (resp.\ $\underline{R_1}$) into line segments $A_i$ (resp.\ $B_i$) each of length $n^{2/3}$. We chose $n$ large enough (depending only on $\epsilon$) so that to each pair of these sub-intervals we can apply \cite[Theorem 4.2]{BGZ21}. Hence there exists $C,c>0$ such that for sufficiently large $n$  (depending only on $\epsilon$) 
\begin{displaymath}
\mathbb{P}((\mathcal{H}_1)^c) \leq CM^4e^{-cM^{3/2}}.
\end{displaymath}
In the beginning we fix $M$ large enough so that the right hand side above is smaller than $0.0016$. We get similar bounds by applying Proposition \cite[Theorem 4.2]{BGZ21} again for $(\mathcal{H}_2)^c, (\mathcal{H}_3)^c$ and $(\mathcal{H}_i^{k})^c$ for each $1 \leq i \leq 3$. Finally taking union bound we have $\mathbb{P}(\mathcal{H}) \geq 0.99$.\\
We get $\mathbb{P}(\mathcal{I}) \geq 0.99$ and $\mathbb{P}(\widetilde{\mathcal{I}}) \geq 0.99$ by arguing similarly and using \cite[Theorem 4.2]{BGZ21}. This concludes the lemma.
\end{proof}
After constructing the favourable events we now focus on the barrier regions. The parallelograms $R_1,R_2,
R_1^{k}, R_2^{k}$ will act as barriers. We will show for fixed paths $\gamma_1 \subset R_1, \gamma_2 \subset R_2, 
\gamma_1^{k} \subset R_1^{k}, \gamma_2^{k} \subset R_2^{k}$ there is positive probability of the event where all the disjoint paths in the respective parallelograms have much smaller weight. But before that we show that on typical environment the point to point geodesics $\Gamma_{-\boldsymbol{n}, \boldsymbol{n}}$ and $\Gamma_{-\boldsymbol{n}_k, \boldsymbol{n}_k}$ shows some typical behaviours in these rectangles with large probability.\\
 Let us first define the following deterministic set of quadruples of paths.\\
    $\mathcal{L}:= \{(\gamma_1, \gamma_2,
    \gamma_1^{k}, \gamma_2^{k}): 
    \gamma_1$ is a path from  $\underline{R_1} \text{ to } \overline{R_1}, \gamma_2$ is a path from $\underline{R_2}$ to $\overline{R_2},
    \gamma_1^{k}$ is a path from $\underline{R_1^{k}}$ to $\overline{R_1^{k}} , \gamma_2^{k}$ is a path from $\underline{R_2^{k}}$ to $\overline{R_2^{k}} \}.
    $\\Consider the parallelograms $U_n^{2M^{3/4}}$ and $U_{n_k}^{2M^{3/4}}$. A path $\gamma$ is called M-typical if $|\widetilde{\underline{\ell(\gamma)}}| \leq Mn^{1/3}$. Let $\mathcal{K}$ denote the random subset of $\mathcal{L}$ which consists of all quadruples such that all $\gamma_1,\gamma_2,
    \gamma_1^{k}, \gamma_2^{k}
    $ are M-typical and $\gamma_1,\gamma_2$ lie inside $U_n^{2M^{3/4}}$ and $\gamma_1^{k},\gamma_2^{k}$ lie inside $U_{n_k}^{2M^{3/4}}$. We have the following lemma.
\begin{lemma}
\label{typical_geodesic_lemma}
For sufficiently large $M$ (depending on $\epsilon$) and sufficiently large $n$ (depending on $M$) we have
\begin{displaymath}
\mathbb{P}\left(\bigl\{(\Gamma_{-\boldsymbol{n},\boldsymbol{n}}|_{R_1}, \Gamma_{-\boldsymbol{n},\boldsymbol{n}}|_{R_{2}}, 
\Gamma_{-\boldsymbol{n}_k, \boldsymbol{n}_k}|_{R_1^{k}},\Gamma_{-\boldsymbol{n}_k,\boldsymbol{n}_k}|_{R_{2}^{k}})\in \mathcal{K}\bigr\}\right) \geq 0.98.
\end{displaymath}
\end{lemma}
\begin{proof}
Consider the line $x=y$ and let $v_1,v_2,v_3, v_4$ be intersection points of $\underline{R_1}, \overline{R_1}, \underline{R_2}, \overline{R_2}$ with $x=y$ respectively. We have on $\mathcal{H} \cap \mathcal{I}$ \\
$T_{-\boldsymbol{n},\boldsymbol{n}} \geq T_{-\boldsymbol{n},v_1}+\underline{T_{v_1,v_2}}+\underline{T_{v_2,v_3}}+\underline{T_{v_3,v_4}}+\underline{T_{v_4, \boldsymbol{n}}} \\ \geq \mathbb{E}(T_{-\boldsymbol{n},v_1})+\mathbb{E}(T_{v_1,v_2})+\mathbb{E}(T_{v_2,v_3})+\mathbb{E}(T_{v_3,v_4})+\mathbb{E}(T_{v_4, \boldsymbol{n}})-5Mn^{1/3}\\
\geq \mathbb{E}(T_{-\boldsymbol{n},\boldsymbol{n}})-(5M+6C_2)n^{1/3}$.\\
The second inequality comes from the definition of the events $\mathcal{H}, \mathcal{I}$ and the last inequality is a consequence of \eqref{expected_passage_time_estimate}.
    Now, on $\boldsymbol{\mathrm{Restr}}_{n,2M^{3/4}}$ 
    any path $\gamma$ from $-\boldsymbol{n}$ to $\boldsymbol{n}$ going out of $U_{n}^{2M^{3/4}}$ will have 
   \begin{equation}
    \label{Lemma_7.12_equation_2}
       \ell(\gamma) \leq \mathbb{E}(T_{-\boldsymbol{n},\boldsymbol{n}})-c_3 M^{3/2}n^{1/3}.
    \end{equation}
    As $c_3M^{3/2}> (5M+6C_2)$ for sufficiently large $M$,
    from the above argument and \eqref{Lemma_7.12_equation_2} we conclude on $\boldsymbol{\mathrm{Restr}}_{n,2M^{3/4}} \cap \mathcal{H} \cap \mathcal{I}, \Gamma_{-\boldsymbol{n},\boldsymbol{n}}$ lies inside $U_{n}^{2M^{3/4}}$. From this it is clear that on $\boldsymbol{\mathrm{Restr}}_{n,2M^{3/4}} \cap \mathcal{H} \cap \mathcal{I}, \Gamma_{-\boldsymbol{n},\boldsymbol{n}}|_{R_1}$, $\Gamma_{-\boldsymbol{n},\boldsymbol{n}}|_{R_2}$ are M-typical. Similar argument shows that on $\boldsymbol{\mathrm{Restr}}^k_{n,2M^{3/4}} \cap \mathcal{H} \cap \mathcal{I}, \Gamma_{-\boldsymbol{n}_k,\boldsymbol{n}_k}$ lies inside $U_{n_k}^{2M^{3/4}}$ and $\Gamma_{-\boldsymbol{n}_k,\boldsymbol{n}_k}|_{R_1^{k}}$, $\Gamma_{-\boldsymbol{n}_k,\boldsymbol{n}_k}|_{R_2^{k}}$ are M-typical. So, we have
\begin{displaymath}
    \boldsymbol{\mathrm{Restr}}_{n,2M^{3/4}} \cap \boldsymbol{\mathrm{Restr}}^k_{n,2M^{3/4}} \cap \mathcal{H} \cap \mathcal{I}  \subset \bigl\{(\Gamma_{-\boldsymbol{n},\boldsymbol{n}}|_{R_1}, \Gamma_{-\boldsymbol{n},\boldsymbol{n}}|_{R_{2}}, 
\Gamma_{-\boldsymbol{n}_k, \boldsymbol{n}_k}|_{R^{k}},\Gamma_{-\boldsymbol{n}_k,\boldsymbol{n}_k}|_{R_{2}^{k}})\in \mathcal{K}\bigr\} .
\end{displaymath}
Now, from Lemma \ref{restr_large_probability_lemma} and Lemma \ref{H_I_are_large_probability_events} we can show that for sufficiently large $M$ and $n$, each of the events $\boldsymbol{\mathrm{Restr}}_{n,2M^{3/4}}, \boldsymbol{\mathrm{Restr}}^k_{n,2M^{3/4}}, \mathcal{H}, \mathcal{I}$ have probability bigger than $0.997$, taking $M,n$ sufficiently large as before and taking a union bound we have\\ 
    $\mathbb{P}\left (\bigl\{(\Gamma_{-\boldsymbol{n},\boldsymbol{n}}|_{R_1}, \Gamma_{-\boldsymbol{n},\boldsymbol{n}}|_{R_{2}}, 
\Gamma_{-\boldsymbol{n}_k, \boldsymbol{n}_k}|_{R_1^{k}},\Gamma_{-\boldsymbol{n}_k,\boldsymbol{n}_k}|_{R_{2}^{k}})\in \mathcal{K}\bigr\} \right)\\ \geq \mathbb{P} \left(\boldsymbol{\mathrm{Restr}}_{n,2M^{3/4}} \cap \boldsymbol{\mathrm{Restr}}^k_{n,2M^{3/4}} \cap \mathcal{H} \cap \mathcal{I} \right) \geq 0.98$.\\
This completes the proof of Lemma \ref{typical_geodesic_lemma}.
\end{proof} 
Finally, we have the following lemma. 
\begin{comment}Now from Lemma \ref{restr_large_probability_lemma}, Lemma \ref{H_I_are_large_probability_events}, Lemma \ref{typical_geodesic_lemma} and a union bound gives us the following Lemma.
\end{comment}
\begin{lemma}
\label{high_probability_environment}
    For sufficiently large $M$ (depending on $\epsilon$) and sufficiently large $n$ (depending on $M$)
    \begin{displaymath}
        \mathbb{P}\left(\mathcal{H} \cap \widetilde{\mathcal{I}} \cap \boldsymbol{\mathrm{Restr}} \cap \bigl\{(\Gamma_{-\boldsymbol{n},\boldsymbol{n}}|_{R_{1}}, \Gamma_{-\boldsymbol{n},\boldsymbol{n}}|_{R_{2}}, \Gamma_{-\boldsymbol{n}_k,\boldsymbol{n}_k}|_{R_{1}^{k}},\Gamma_{-\boldsymbol{n}_k,\boldsymbol{n}_k}|_{R_{2}^{k}}  )\in \mathcal{K}\bigr\}\right) \geq 0.95.
    \end{displaymath}
\end{lemma}
\begin{proof}
    It follows from Lemma \ref{restr_large_probability_lemma}, Lemma \ref{H_I_are_large_probability_events}, Lemma \ref{typical_geodesic_lemma} and a union bound.
\end{proof}
Now we define following events.\\
For any path $\gamma_1$ from $\underline{R_1}$ to $\overline{R_1}$ , $\gamma_2$ from $\underline{R_2}$ to $\overline{R_2}$, $\gamma_1^{k}$ from $\underline{R_1^{k}}$ to $\overline{R_1^{k}}$ , $\gamma_2^{k}$ from $\underline{R_2^{k}}$ to $\overline{R_2^{k}}$ consider the following events.
\begin{itemize}
    \item $\mathcal{P}_{\gamma_1}$ := the event that any path $\gamma$ from $\underline{R_1}$ to $\overline{R_1}$ contained in $R_1$ and disjoint from $\gamma_1$ satisfies $\widetilde{\underline{\ell(\gamma)}} \leq -M^4n^{1/3}$;
    \item $\mathcal{P}_{\gamma_2}$ := the event that any path $\gamma$ from $\underline{R_2}$ to $\overline{R_2}$ contained in $R_2$ and disjoint from $\gamma_2$ satisfies $\widetilde{\underline{\ell(\gamma)}} \leq -M^4n^{1/3}$;
    \item $\mathcal{P}_{\gamma_1^{k}}$ := the event that any path $\gamma$ from $\underline{R_1^{k}}$ to $\overline{R_1^{k}}$ contained in $R_1^k$ and disjoint from $\gamma_1^{k}$ satisfies $\widetilde{\underline{\ell(\gamma)}} \leq -M^4n^{1/3}$;
    \item $\mathcal{P}_{\gamma_2^{k}}$ := the event that any path $\gamma$ from $\underline{R_2^{k}}$ to $\overline{R_2^{k}}$ contained in $R_2^k$ and disjoint from $\gamma_2^{k}$ satisfies $\widetilde{\underline{\ell(\gamma)}} \leq -M^4n^{1/3}$;
    \item $\mathcal{P}_{\gamma_1,\gamma_2,\gamma_1^{k},\gamma_2^{k}}:=\mathcal{P}_{\gamma_1} \cap \mathcal{P}_{\gamma_2} \cap \mathcal{P}_{\gamma_1^{k}} \cap \mathcal{P}_{\gamma_2^{k}}$.
    \end{itemize}
    Using Proposition \ref{small_probability_lemma} we have the following lemma. 
    \begin{lemma}
    \label{barrier_events_lemma}
    For any path $\gamma_1$ from $\underline{R_1}$ to $\overline{R_1}$ , $\gamma_2$ from $\underline{R_2}$ to $\overline{R_2}$, $\gamma_1^{k}$ from $\underline{R_1^{k}}$ to $\overline{R_1^{k}}$ , $\gamma_2^{k}$ from $\underline{R_2^{k}}$ to $\overline{R_2^{k}}$ and fixed $\epsilon>0$ there exists a positive constant $c_5>0$ such that for all $M$ large enough $n$ large enough (depending only on $M, \epsilon$) we have 
    \begin{displaymath}
    \mathbb{P}(\mathcal{P}_{\gamma_1,\gamma_2,\gamma_1^{k},\gamma_2^{k}}) \geq c_5.
    \end{displaymath}
    \end{lemma}
\begin{remark}
\label{independence_remark}
As pointed out before we will chose $k$ large enough so that the parallelograms $R_1,R_2,R_1^{k},R_2^{k}$ have disjoint set of vertices. Also we have that the events $\mathcal{P}_{\gamma_1},\mathcal{P}_{\gamma_2},\mathcal{P}_{\gamma_1^{k}},\mathcal{P}_{\gamma_2^{k}}$  depend only on the vertex weights of $R_1,R_2,R_1^{k},R_2^{k}$ respectively. Hence they are independent. So, if we can show that probability of each of the events has a positive lower bound  then their intersection will also have positive lower bound.
\end{remark}
\begin{proof}
For $\mathcal{P}_{\gamma_1}$ we have 
\begin{displaymath}
    \Bigg \{\sup_{u \in \underline{R_1} ,v \in \overline{R_1}} \widetilde{\underline{T_{u,v}}} \leq -M^4n^{1/3} \Bigg \} \subset \mathcal{P}_{\gamma_1}.
\end{displaymath}
Hence using Proposition \ref{small_probability_lemma} we get $\mathbb{P}(\mathcal{P}_{\gamma_1^n}) \geq c$ for some constants $c$. Similar argument holds for $\mathcal{P}_{\gamma_2}, \mathcal{P}_{\gamma_1^k}$ and  $\mathcal{P}_{\gamma_2^{k}}$. Following Remark \ref{independence_remark} we get $\mathbb{P}(\mathcal{P}_{\gamma_1,\gamma_2,\gamma_1^{k},\gamma_2^{k}}) \geq c_5$ for some $c_5>0$ as stated in Lemma \ref{barrier_events_lemma}.
\end{proof}
Our next goal is to force geodesics to coalesce. For a fixed $(\gamma_1,\gamma_2, \gamma_1^{k}, \gamma_2^{k}) \in \mathcal{L} \cap R_1 \times R_2 \times R_1^k \times R_2^k$ we define the following events. Recall that $u_1,u_2$ (resp.\ $u_1',u_2'$) are the end points of $P_n$ (resp.\ $P_n^*$) and $w_1,w_2$ (resp.\ $w_1',w_2'$) are the end points of $Q_n$ (resp.\ $Q_n^*$).
\begin{itemize}
\item $\mathcal{R}_{\gamma_{1},\gamma_2,\gamma_1^{k},\gamma_2^{k}}:=\{\Gamma_{-\boldsymbol{n},\boldsymbol{n}}|_{R_1}=\gamma_1, \Gamma_{-\boldsymbol{n},\boldsymbol{n}}|_{R_2}=\gamma_2,\Gamma_{-\boldsymbol{n}_k,\boldsymbol{n}_k}|_{R_1^{k}}=\gamma_1^{k}, \Gamma_{-\boldsymbol{n}_k,\boldsymbol{n}_k}|_{R_2^{k}}=\gamma_2^{k}\}$;
\item $\mathcal{C}_{\gamma_{1},\gamma_2,\gamma_1^{k},\gamma_2^{k}}:=\{\Gamma_{u_1,u_1'} \text{ and } \Gamma_{u_2,u_2'} \text{ will meet } \gamma_1 \text{ and }\gamma_2, \Gamma_{w_1,w_1'} \text{ , }  \Gamma_{w_2,w_2'} \text{ will meet }\\ \gamma_1^{k} \text{ and }\gamma_2^{k}\};$
\item 
    $\mathcal{C}':=\{\Gamma_{-\boldsymbol{n}, \boldsymbol{n}} \text{ and }\Gamma_{-\boldsymbol{n}_k, \boldsymbol{n}_k} \text{ will intersect inside } V_{\frac{M}{2}}\}$,
\end{itemize}
\begin{comment}\begin{displaymath}
\mathcal{C}'':=\{\Gamma_{u_1,P_n^*}, \Gamma_{u_2,P_n^*}, \Gamma_{w_1,Q_n^*}, \Gamma_{w_2,Q_n^*} \text{ will intersect inside } \widehat{V_{M}}\},
\end{displaymath}
\end{comment}
Note that for a fixed $(\gamma_1,\gamma_2, \gamma_1^{k}, \gamma_2^{k}) \in \mathcal{L},$ on $\mathcal{R}_{\gamma_{1},\gamma_2,\gamma_1^{k},\gamma_2^{k}} \cap \mathcal{C}_{\gamma_{1},\gamma_2,\gamma_1^{k},\gamma_2^{k}} \cap \mathcal{C}'$, $\mathcal{E}$ happens (recall $\mathcal{E}$ as defined in Proposition \ref{positive_probability_for_point_to_line}).\\
We finally have the following proposition.
\begin{proposition}
\label{forcing_coealsance}
In the above setup we have for sufficiently large $M$ (depending on $\epsilon)$ and sufficiently large $n$ (depending on $M$) we have,
\begin{displaymath}
\{(\gamma_1,\gamma_2,\gamma_1^{k}, \gamma_2^{k}) \in \mathcal{K}\} \cap \mathcal{P}_{\gamma_1,\gamma_2,\gamma_1^{k},\gamma_2^{k}} \cap \mathcal{H} \cap \mathcal{\widetilde{I}} \cap \boldsymbol{\mathrm{Restr}} \subset \mathcal{C}_{\gamma_{1},\gamma_2,\gamma_1^{k},\gamma_2^{k}} \cap \mathcal{C'}.
\end{displaymath}
\end{proposition}
\begin{figure}[t!]
    \includegraphics[width=15 cm]{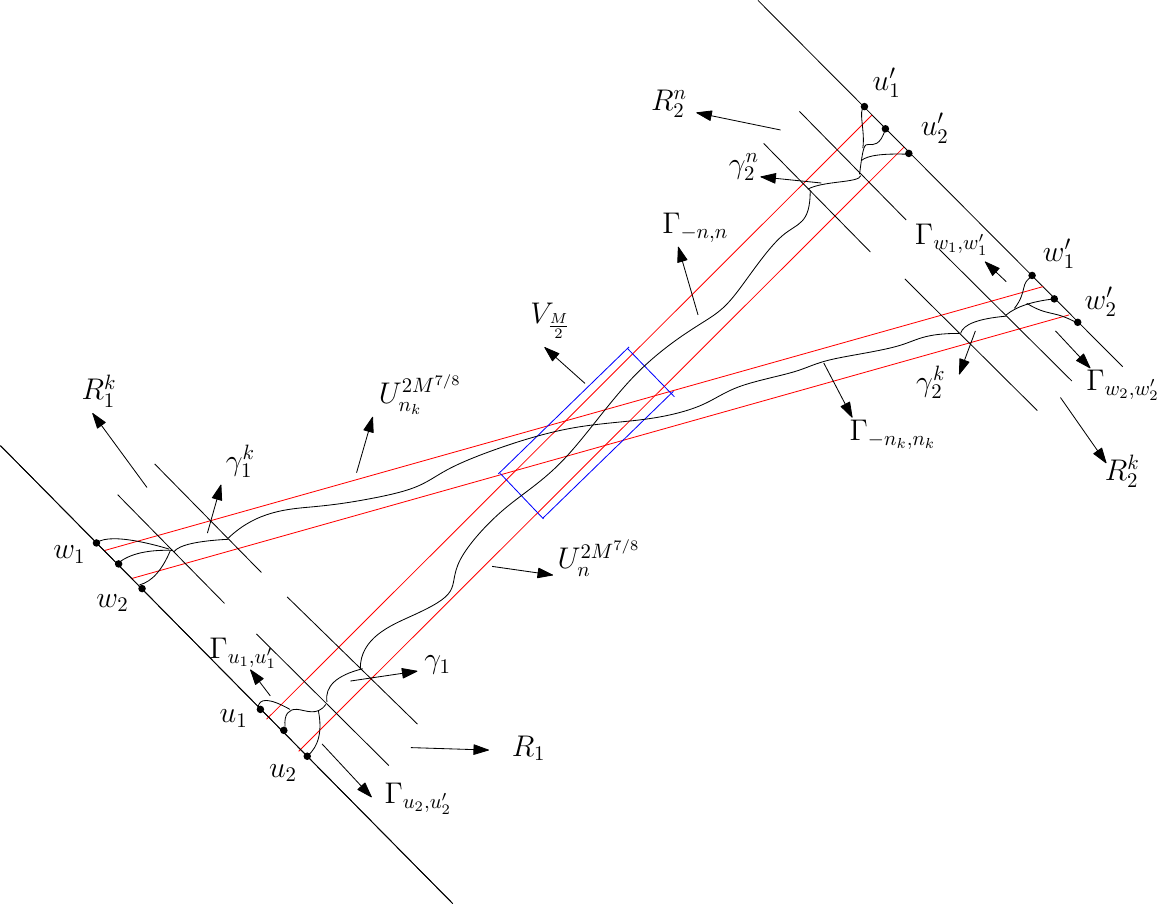}
        \caption{In Proposition \ref{forcing_coealsance}, we consider the event $\{(\gamma_1,\gamma_2,\gamma_1^{k}, \gamma_2^{k}) \in \mathcal{K}\} \cap \mathcal{P}_{\gamma_1,\gamma_2,\gamma_1^{k},\gamma_2^{k}} \cap \mathcal{H} \cap \widetilde{\mathcal{I}} \cap \boldsymbol{\mathrm{Restr}}$. On this event, comparing weights of different paths we show that $\Gamma_{u_1,u_1'},\Gamma_{u_2,u_2'}$ (resp.\ $\Gamma_{w_1,w_1'}, \Gamma_{w_2,w_2'}$) intersect $\gamma_1$ and $\gamma_2$ (resp.\ $\gamma_1^{k}$ and $\gamma_2^{k}$). $\mathcal{H} \cap \widetilde{\mathcal{I}} \cap \boldsymbol{\mathrm{Restr}}$ ensures that $\Gamma_{-\boldsymbol{n},\boldsymbol{n}}$ (resp.\ $\Gamma_{-\boldsymbol{n}_k,\boldsymbol{n}_k}$) lie inside $U_{n}^{2M^{7/8}}$ (resp.\ $U_{n_k}^{2M^{7/8}}$). Hence, in the setup of Proposition \ref{forcing_coealsance} all geodesics starting from $P_n$ (resp.\ $Q_n$) to $P_n^*$ (resp.\ $Q_n^*$) follow unique paths inside $V_{\frac{M}{2}}$ and these two paths intersect inside $V_{\frac{M}{2}}$. Finally a conditioning argument proves Proposition \ref{positive_probability_for_point_to_line}}
    \label{fig:forcing_geodesics}
\end{figure}
\begin{proof}
First we will consider the line segment $P_n$ and $P_n^*$. The idea is to show that on
\[
 \{(\gamma_1,\gamma_2,\gamma_1^{k}, \gamma_2^{k}) \in \mathcal{K}\} \cap \mathcal{P}_{\gamma_1,\gamma_2,\gamma_1^{k},\gamma_2^{k}} \cap \mathcal{H} \cap \widetilde{\mathcal{I}} \cap \boldsymbol{\mathrm{Restr}},
\]
$\Gamma_{u_1,u_1'}$ and $\Gamma_{u_2,u_2'}$ (and hence all the geodesics starting from $P_n$ and ending at $P_n^*$) will intersect $\gamma_1$ and $\gamma_2$. We will further show that, on $\mathcal{H} \cap \widetilde{\mathcal{I}} \cap \boldsymbol{\mathrm{Restr}}, \Gamma_{-\boldsymbol{n},\boldsymbol{n}}$ will lie inside $U_{n}^{2M^{7/8}}$. 
\begin{comment}This will ensure $\mathcal{C'}$.
\end{comment}
We proceed to prove the above claims now.

\begin{comment}Let $v \in P_n^*$
be any element.
\end{comment}
First using $\gamma_1$ and $\gamma_2$ we construct a path with large weight. We do this in the following way. Let the endpoints of $\gamma_1$ be $v_1$ and $v_2$ and end points of $\gamma_2$ be $v_3$ and $v_4$. Note that, $v_1,v_2,v_3,v_4 \in  U_{n}^{2M^{3/4}} \subset U_{u_1,u_1'}^{2M}$. We consider the path $\chi_1$ from $u_1$ to $u_1'$ which is a concatenation of $\Gamma_{u_1,v_1}, \gamma_1, \Gamma_{v_2,v_3}, \gamma_2, \Gamma_{v_4, u_1'}.$ We have
\begin{equation}
\label{a path with large weight}
\ell(\chi_1)=\ell(\Gamma_{u_1,v_1})+\underline{\ell(\gamma_1)}+\underline{\ell(\Gamma_{v_2,v_3})}+\underline{\ell(\gamma_2)}+\underline{\ell(\Gamma_{v_4,u_1'})}
\end{equation}
Now on $\mathcal{H} \cap \widetilde{\mathcal{I}} \cap \{(\gamma_1,\gamma_2,\gamma_1^{k}, \gamma_2^{k}) \in \mathcal{K}\} $ by definition we have \\
    $\ell(\chi_1) \geq \mathbb{E}(T_{u_1,v_1})+\mathbb{E}(T_{v_1,v_2})+\mathbb{E}(T_{v_2,v_3})+\mathbb{E}(T_{v_3,v_4})+\mathbb{E}(T_{v_4,u_1'})-5Mn^{1/3} \\ \geq \mathbb{E}(T_{u_1,u_1'})-cM^2n^{1/3}-Cn^{1/3}-5Mn^{1/3}$
\\for some constants $c,C>0.$ The last inequality comes from \eqref{expected_passage_time_estimate} and using an easy calculus argument. \\So, on $\mathcal{H} \cap \widetilde{\mathcal{I}} \cap \{(\gamma_1,\gamma_2,\gamma_1^{k}, \gamma_2^{k}) \in \mathcal{K}\}$ we have 
\begin{equation}
\label{large_last_passage_time}
    T_{u_1,u_1'} \geq \mathbb{E}(T_{u_1,u_1'})-cM^2n^{1/3}-Cn^{1/3}-5Mn^{1/3}.
\end{equation}
Now, on the event $\boldsymbol{\mathrm{Restr}}$ any path $\gamma$ from $u_1$ to $u_1'$ going out of $U_{n}^{2M^2}$ have 
\begin{equation}
\label{bad_paths_are_penalised}
    \ell(\gamma) < \mathbb{E}(T_{u_1,u_1'})-c_3M^4n^{1/3}.
\end{equation}
From the above argument it is clear that if we chose $M$ sufficiently large $\Gamma_{u_1,u_1'}$ lie inside $U_{n}^{2M^2}$. Similar argument shows $\Gamma_{u_2,u_2'}$ lie inside $U_{n}^{2M^2}$.\\
The next step is to show that on $\{(\gamma_1,\gamma_2,\gamma_1^{k}, \gamma_2^{k}) \in \mathcal{K}\} \cap \mathcal{P}_{\gamma_1,\gamma_2,\gamma_1^{k},\gamma_2^{k}} \cap \mathcal{H} \cap \widetilde{\mathcal{I}} \cap \boldsymbol{\mathrm{Restr}}$ if any path $\gamma \in U_{n}^{M^2}$ from $u_1$ to $u_1'$ does not intersect either $\gamma_1$ or $\gamma_2$ then that path will be penalised heavily. Assume that $\gamma$ is from $u_1$ to $u_1'$ and $\gamma$ does not intersect $\gamma_1$. Then on $\{(\gamma_1,\gamma_2,\gamma_1^{k}, \gamma_2^{k}) \in \mathcal{K}\} \cap \mathcal{P}_{\gamma_1,\gamma_2,\gamma_1^{k},\gamma_2^{k}} \cap \mathcal{H} \cap \widetilde{\mathcal{I}} \cap \boldsymbol{\mathrm{Restr}}$ by definition we have
\begin{equation}
\label{penalised_heavily_if_does_not_intersect_gamma_1}
    \ell(\gamma) \leq \mathbb{E}(T_{u_1,u_1'})+4Mn^{1/3}-M^4n^{1/3}.
\end{equation}
Choosing $M$ sufficiently large we have $\ell(\gamma)< \ell(\chi_1)$. Hence, $\gamma$ can not be a geodesic. Similar argument shows that $\Gamma_{u_1,u_1'}$ intersects $\gamma_2$ also and $\Gamma_{u_2,u_2'}$ intersects both $\gamma_1$ and $\gamma_2$.\\
So, the only thing remains to show is that $\Gamma_{-\boldsymbol{n},\boldsymbol{n}}$ lies within $U_{n}^{2M^{7/8}}.$ We construct a path using $\gamma_1$ and $\gamma_2$ with large weight as before. Consider the path $\chi_0$ from $-\boldsymbol{n}$ to $\boldsymbol{n}$ which is a concatenation of $\Gamma_{-\boldsymbol{n},v_1}, \gamma_1, \Gamma_{v_2,v_3}, \gamma_2, \Gamma_{v_4, \boldsymbol{n}}.$ Considering the parts of this path as before we have
\begin{equation}
\label{another_path_with_large_weight}
    \ell(\chi_0) \geq \mathbb{E}(T_{-\boldsymbol{n},\boldsymbol{n}})-cM^{3/2}n^{1/3}-Cn^{1/3}-5Mn^{1/3}.
\end{equation}
This inequality again comes from \eqref{expected_passage_time_estimate}, an easy calculus argument and the fact that $\gamma_1$ and $\gamma_2$ lie inside $U_{n}^{M^{3/4}}$.\\
On $\boldsymbol{\mathrm{Restr}}_{n,2M^{7/8}}$ using the above inequality we have $\Gamma_{-\boldsymbol{n}, \boldsymbol{n}}$ lie in $U_{n}^{2M^{7/8}}$.\\
Similarly, on $\{(\gamma_1,\gamma_2,\gamma_1^{k}, \gamma_2^{k}) \in \mathcal{K}\} \cap \mathcal{P}_{\gamma_1,\gamma_2,\gamma_1^{k},\gamma_2^{k}} \cap \mathcal{H} \cap \widetilde{\mathcal{I}} \cap \boldsymbol{\mathrm{Restr}}$ $\Gamma_{w_1,w_1'}$ and $\Gamma_{w_2,w_2'}$ intersect both $\gamma_1^{k}$ and $\gamma_2^{k}$ and $\Gamma_{-\boldsymbol{n}_k,\boldsymbol{n}_k}$ lie inside $U_{n_k}^{2M^{7/8}}.$ One final thing we observe that intersection of $U_{n}^{2M^{7/8}}$ and $U_{n_k}^{2M^{7/8}}$ lie inside $V_{\frac{M}{2}}$. Hence, we have
\begin{displaymath}
\mathcal{P}^{\gamma_1,\gamma_2,\gamma_1^{k},\gamma_2^{k}} \cap \mathcal{H} \cap \widetilde{\mathcal{I}} \cap \boldsymbol{\mathrm{Restr}} \cap \subset \mathcal{C}_{\gamma_{1},\gamma_2,\gamma_1^{k},\gamma_2^{k}} \cap \mathcal{C'}.
\end{displaymath}
This completes the proof.
\end{proof}
As a final step we apply a conditioning argument. For this, first we prove the following lemma. Before that we define a deterministic set of quadruple of paths. 
\begin{itemize}
\item $\mathcal{J}:=\{(\gamma_1,\gamma_2, \gamma_1^{k},\gamma_2^{k}) \in \mathcal{L}: \gamma_1$ and $\gamma_2$ are contained in $U_{n}^{2M^{3/4}}, \gamma_1^{k}$ and  $\gamma_2^{k}$ are contained in $U_{n_k}^{2M^{3/4}}\}$.
\end{itemize}
\begin{lemma}
\label{Final_Conditioning_lemma}
For any $(\gamma_1,\gamma_2,\gamma_1^{k},\gamma_2^{k}) \in \mathcal{J}$ and for sufficiently large $M$ (depending on $\epsilon$) and sufficiently large $n$ (depending on $M$) we have
\begin{displaymath} 
\mathbb{P} \left (\mathcal{P}_{\gamma_1,\gamma_2, \gamma_1^{k}, \gamma_2^{k}}|\mathcal{H} \cap \widetilde{\mathcal{I}} \cap \boldsymbol{\mathrm{Restr}} \cap \mathcal{R}_{\gamma_1,\gamma_2, \gamma_1^{k}, \gamma_2^{k}} \cap \{(\gamma_1,\gamma_2, \gamma_1^{k}, \gamma_2^{k}) \in \mathcal{K}\}\right ) \geq c_5.
\end{displaymath}
\end{lemma}
\begin{proof}
    In the above set-up we have $\mathcal{P}_{\gamma_1,\gamma_2, \gamma_1^{k}, \gamma_2^{k}}, \widetilde{\mathcal{I}}, \boldsymbol{\mathrm{Restr}}$ are decreasing events. Consider the following region 
    \begin{displaymath}
    \mathsf{R}:=\left(R_1 \setminus (\gamma_1 \cup \overline{R_1}) \cup R_2 \setminus (\gamma_2 \cup \overline{R_2}) \cup R_1^{k} \setminus (\gamma_1^{k} \cup \overline{R_1^{k}}) \cup R_2^{k} \setminus (\gamma_2^{k} \cup \overline{R_2^{k}}) \right)^c.
    \end{displaymath}
     If we fix a configuration of $\mathsf{R}$ and condition on this configuration then $\mathcal{R}_{\gamma_1,\gamma_2, \gamma_1^{k}, \gamma_2^{k}}$ is a decreasing event. Further, note that if $\mathcal{F}_{\gamma_1,\gamma_2, \gamma_1^{k},\gamma_2^{k}}$ is the $\sigma-$algebra generated by the vertex weights of the above region then $\mathcal{P}_{\gamma_1,\gamma_2, \gamma_1^{k}, \gamma_2^{k}}$ is independent of $\mathcal{F}_{\gamma_1,\gamma_2, \gamma_1^{k},\gamma_2^{k}}$ and $\mathcal{H} \cap \{(\gamma_1,\gamma_2, \gamma_1^{k}, \gamma_2^{k}) \in \mathcal{K}\}$ is measurable with respect to $\mathcal{F}_{\gamma_1,\gamma_2, \gamma_1^{k},\gamma_2^{k}}$. Hence, using the FKG inequality to each fix configuration of $\mathsf{R}$ and integrating we have\\
$\mathbb{P}\left (\mathcal{P}_{\gamma_1,\gamma_2, \gamma_1^{k}, \gamma_2^{k}} \cap \widetilde{\mathcal{I}} \cap \boldsymbol{\mathrm{Restr}} \cap \mathcal{R}_{\gamma_1,\gamma_2, \gamma_1^{k}, \gamma_2^{k}} | \mathcal{H} \cap \{(\gamma_1,\gamma_2, \gamma_1^{k}, \gamma_2^{k} ) \in \mathcal{K}\} \right)\\ \geq \mathbb{P} \left (\mathcal{P}_{\gamma_1,\gamma_2, \gamma_1^{k}, \gamma_2^{k}}| \mathcal{H} \cap \{(\gamma_1,\gamma_2, \gamma_1^{k}, \gamma_2^{k} ) \in \mathcal{K} \right ) \times \\ \mathbb{P} \left(\widetilde{\mathcal{I}} \cap \boldsymbol{\mathrm{Restr}} \cap \mathcal{R}_{\gamma_1,\gamma_2, \gamma_1^{k}, \gamma_2^{k}}| \mathcal{H} \cap \{(\gamma_1,\gamma_2, \gamma_1^{k}, \gamma_2^{k} ) \in \mathcal{K}\} \right)\\=\mathbb{P} \left (\mathcal{P}^{\gamma_1^{n},\gamma_2^{n}, \gamma_1^{k}, \gamma_2^{k}} \right) \mathbb{P} \left(\widetilde{\mathcal{I}} \cap \boldsymbol{\mathrm{Restr}} \cap \mathcal{R}_{\gamma_1,\gamma_2, \gamma_1^{k}, \gamma_2^{k}} | \mathcal{H} \cap \{(\gamma_1,\gamma_2, \gamma_1^{k}, \gamma_2^{{k}} ) \in \mathcal{K}\} \right) \\ \geq c_5 \mathbb{P} \left(\widetilde{\mathcal{I}} \cap \boldsymbol{\mathrm{Restr}} \cap \mathcal{R}_{\gamma_1,\gamma_2, \gamma_1^{k}, \gamma_2^{k}} | \mathcal{H} \cap \{(\gamma_1,\gamma_2, \gamma_1^{k}, \gamma_2^{k} ) \in \mathcal{K}\} \right).$
\\ The last inequality comes from Lemma \ref{barrier_events_lemma}. Combining above we get the lemma. 
\end{proof}
Finally we prove Proposition \ref{positive_probability_for_point_to_line}.
\begin{proof}[Proof of Proposition \ref{positive_probability_for_point_to_line}]
We have in the notation of Proposition \ref{positive_probability_for_point_to_line}\\
$\mathbb{P}(\mathcal{E}) \geq \sum_{(\gamma_1,\gamma_2,\gamma_1^{k},\gamma_2^{k}) \in \mathcal{J}}\mathbb{P} \left (\mathcal{E} \cap \mathcal{R}_{\gamma_1,\gamma_2, \gamma_1^{k}, \gamma_2^{k}} \right ) \\ \geq \sum_{(\gamma_1,\gamma_2,\gamma_1^{k},\gamma_2^{k}) \in \mathcal{J}}\mathbb{P} \left (\mathcal{C}_{\gamma_1,\gamma_2, \gamma_1^{k}, \gamma_2^{k}} \cap \mathcal{C}' \cap \mathcal{R}_{\gamma_1,\gamma_2, \gamma_1^{k}, \gamma_2^{k}} \right ) \\ \geq \sum_{(\gamma_1,\gamma_2, \gamma_1^{k}, \gamma_2^{k}) \in \mathcal{J}} \mathbb{P} \left (\mathcal{P}_{\gamma_1,\gamma_2, \gamma_1^{k}, \gamma_2^{k}} \cap \mathcal{H} \cap \widetilde{\mathcal{I}} \cap \boldsymbol{\mathrm{Restr}} \cap \mathcal{R}_{\gamma_1,\gamma_2, \gamma_1^{k}, \gamma_{2^k}} \cap \{(\gamma_1,\gamma_2, \gamma_1^{k}, \gamma_2^{k}) \in \mathcal{K}\}\right ).$\\ The last two inequalities come from Proposition \ref{forcing_coealsance}.\\
Now we have for a fixed $(\gamma_1,\gamma_2, \gamma_1^{k}, \gamma_2^{k}) \in \mathcal{J}\\$
$\mathbb{P} \left (\mathcal{P}^{\gamma_1,\gamma_2, \gamma_1^{k}, \gamma_2^{k}} \cap \mathcal{H} \cap \widetilde{\mathcal{I}} \cap \boldsymbol{\mathrm{Restr}} \cap \mathcal{R}_{\gamma_1,\gamma_2, \gamma_1^{k}, \gamma_2^{k}} \cap \{(\gamma_1,\gamma_2, \gamma_1^{k}, \gamma_2^{k}) \in \mathcal{K}\}\right )\\ =\mathbb{P} \left (\mathcal{P}^{\gamma_1,\gamma_2, \gamma_1^{k}, \gamma_2^{k}}|\mathcal{H} \cap \widetilde{\mathcal{I}} \cap \boldsymbol{\mathrm{Restr}} \cap \mathcal{R}_{\gamma_1,\gamma_2, \gamma_1^{k}, \gamma_2^{k}} \cap \{(\gamma_1,\gamma_2, \gamma_1^{k}, \gamma_2^{k}) \in \mathcal{K}\}\right ) \times \\ \mathbb{P} \left (\mathcal{H} \cap \widetilde{\mathcal{I}} \cap \boldsymbol{\mathrm{Restr}} \cap \mathcal{R}_{\gamma_1,\gamma_2, \gamma_1^{k}, \gamma_2^{k}} \cap \{(\gamma_1,\gamma_2, \gamma_1^{k}, \gamma_2^{k}) \in \mathcal{K}\} \right )\\
\geq c_5\mathbb{P} \left (\mathcal{H} \cap \widetilde{\mathcal{I}} \cap \boldsymbol{\mathrm{Restr}} \cap \mathcal{R}_{\gamma_1,\gamma_2, \gamma_1^{k}, \gamma_2^{k}} \cap \{(\gamma_1,\gamma_2, \gamma_1^{k}, \gamma_2^{k}) \in \mathcal{K}\} \right ).$\\
The last inequality comes from Lemma \ref{Final_Conditioning_lemma}.
Hence, summing over all $(\gamma_1, \gamma_2, \gamma_1^k, \gamma_2^k) \in \mathcal{J}$ and using the definition of $\mathcal{R}_{\gamma_1,\gamma_2, \gamma_1^{k}, \gamma_2^{k}}$ we have the following\\
$\mathbb{P}(\mathcal{E}) \geq \sum_{(\gamma_1,\gamma_2, \gamma_1^{k}, \gamma_2^{k}) \in \mathcal{J}}\mathbb{P} \left (\mathcal{H} \cap \widetilde{\mathcal{I}} \cap \boldsymbol{\mathrm{Restr}} \cap \mathcal{R}_{\gamma_1,\gamma_2, \gamma_1^{k}, \gamma_2^{k}} \cap \{(\gamma_1,\gamma_2, \gamma_1^{k}, \gamma_2^{k}) \in \mathcal{K}\} \right )\\ \geq c_5\mathbb{P} \left (\mathcal{H} \cap \widetilde{\mathcal{I}} \cap \boldsymbol{\mathrm{Restr}} \cap \{(\Gamma_{-\boldsymbol{n},\boldsymbol{n}}|_{R_1},\Gamma_{-\boldsymbol{n},\boldsymbol{n}}|_{R_2},\Gamma_{-\boldsymbol{n}_k,\boldsymbol{n}_k}|_{R_1^{k}}, \Gamma_{-\boldsymbol{n}_k,\boldsymbol{n}_k}|_{R_2^{k}})\in \mathcal{K}\} \right ) \\ \geq 0.95c_5.$\\
The last inequality comes from Lemma \ref{high_probability_environment}.
This completes the proof of Proposition \ref{positive_probability_for_point_to_line}.
\end{proof}
\subsection{Proof of Proposition \ref{small_probability_lemma}}Recall the notation of Proposition \ref{small_probability_lemma}. The idea of the proof is essentially same as \cite[Lemma 4.10]{BGZ21}, except for the fact that here the result is about general directions. First we fix $m, \Delta$ and $x$. As described in Figure {\ref{fig: general_direction_proof}}, we will consider points on $\mathcal{L}_{-2\mu n}$ and $\mathcal{L}_{2(\mu+1)n}$ for small enough $\mu$ (to be chosen later) and apply Proposition \cite[Theorem 4.2]{BGZ21}. To apply Proposition \cite[Theorem 4.2]{BGZ21}, we divide $U_{m, \Delta}^0$ and $U_{m,\Delta}^{2n}$ into line segments of length $(\mu n)^{2/3}$. The line segments on $U_{m, \Delta}^0$ (resp.\ $U_{m,\Delta}^{2n}$) will be denoted by $A_i$ (resp.\ $B_i$) for $1 \leq i \leq \frac{\Delta}{\mu^{2/3}}$ (see Figure \ref{fig: general_direction_proof}). We wish to show for any $i,j$ satisfying $1 \leq i \leq \frac{\Delta}{\mu^{2/3}}, 1 \leq j \leq \frac{\Delta}{\mu^{2/3}}$ there exist $C,c>0$ depending on $\Delta,\phi$ (as defined in the statement of the proposition) such that for all sufficiently large $n$ (depending on $\phi$)
\begin{equation}
\label{sub_intervals_are_positive_probability}
    \mathbb{P}\left(\sup_{u \in A_i, v \in B_j} \widetilde{T_{u,v}} \leq -xn^{1/3} \right) \geq Ce^{-cx^3}.
\end{equation}
\begin{figure}[t!]
    \includegraphics[width=13 cm]{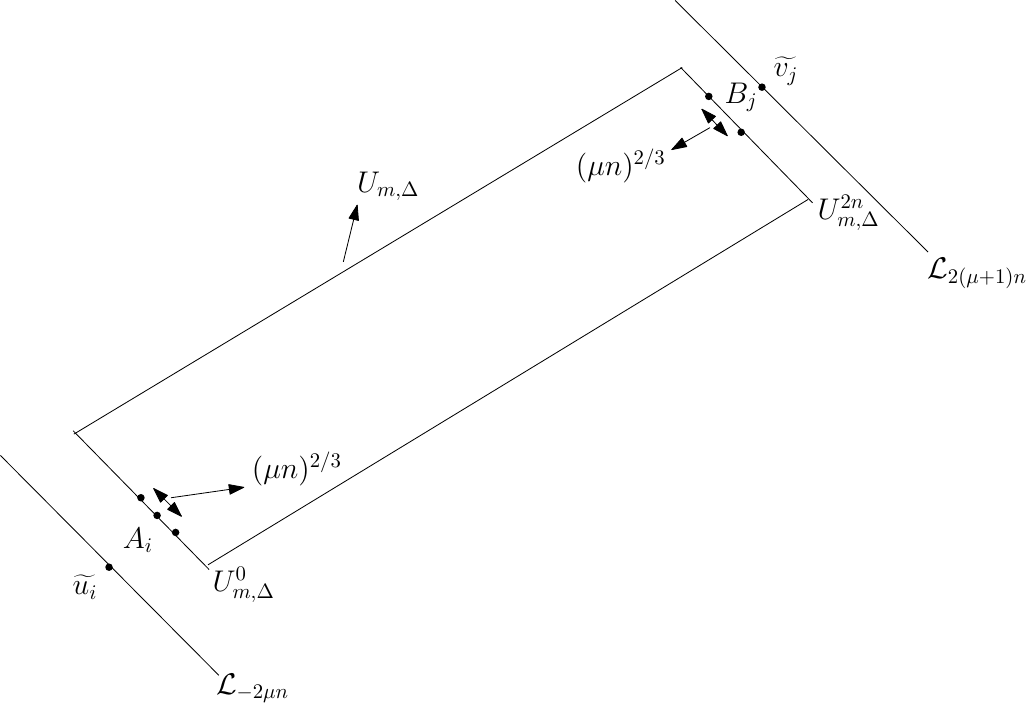}
    \caption{To prove Proposition \ref{small_probability_lemma} we take points $\widetilde{u_i},\widetilde{v_j}$ outside of $U_{m,\Delta}$. Using \cite[Theorem 2]{BGHK21} we have, between any two  $\tilde{u_i}$ and $\tilde{v_j}$ passage time can be arbitrarily small with positive probability and in our setup this probability is lower bounded by a uniform constant for general direction. Now, on the event where passage times from $\widetilde{u_i}$ to $U_{m, \Delta}^{0}$ and $U_{m, \Delta}^{2n}$ to $\widetilde{v_j}$ are not too small and passage time between $\widetilde{u_i}$ to $\widetilde{v_j}$ are small, $\sup_{u \in U_{m, \Delta}^0, v \in U_{m, \Delta}^{2n}}\widetilde{T_{u,v}}$ is small.}
    %Now, taking $\mu$ small enough we can make $\mathbb{P}(\mathcal{B}_i)$ and $\mathbb{P}(\mathcal{B}_j)$ large using \cite[Theorem 4.2]{BGZ21}. For this, we need to divide $U_{m, \Delta}^0$ and $U_{m, \Delta}^{2n}$ into intervals of length $(\mu n)^{2/3}$. Combining these we get the lower bound in Proposition \ref{small_probability_lemma}.
    \label{fig: general_direction_proof} 
\end{figure}
Let $\mathcal{B}_{i,j}$ denote the event $\{\sup_{u \in A_i, v \in B_j} \widetilde{T_{u,v}} \leq -xn^{1/3}\}$. Let, $\mathcal{B}$ denote the event \\ $\{\sup_{u \in U_{m, \Delta}^0, v \in U_{m, \Delta}^{2n}}\widetilde{T_{u,v}} \leq -xn^{1/3}\}$. Clearly,
\begin{displaymath}
    \mathcal{B}=\bigcap_{1 \leq i \leq \frac{\Delta}{\mu^{2/3}}, 1 \leq j \leq \frac{\Delta}{\mu^{2/3}}} \mathcal{B}_{i,j}.
\end{displaymath}
Each $\mathcal{B}_{i,j}$ are decreasing events. Hence, if we can show \eqref{sub_intervals_are_positive_probability}, then using the FKG inequality we can conclude
\begin{displaymath}
    \mathbb{P}(\mathcal{B}) \geq (Ce^{-cx^3})^{(\frac{\Delta}{\mu^{2/3}})^2}.
\end{displaymath}
This will conclude the proposition. We prove \eqref{sub_intervals_are_positive_probability} now.\\
Let $1 \leq i \leq \frac{\Delta}{\mu^{2/3}}$ and $1 \leq j \leq \frac{\Delta}{\mu^{2/3}}$ be fixed. $u_i$ (resp.\ $v_j$) denote the mid points of $A_i$ (resp.\ $B_j$). Consider the lines $\mathcal{L}_{-2\mu n}$ and $\mathcal{L}_{2(1+\mu) n}$. We choose $\widetilde{u_i}$ (resp.\ $\widetilde{v_j}$) on $\mathcal{L}_{-2\mu n}$ (resp.\ $\mathcal{L}_{2n+\mu n}$) such that $\psi(u_i)=\psi(\widetilde{u_i})$ and $\psi(v_j)=\psi(\widetilde{v_j})$ (see Figure \ref{fig: general_direction_proof}). First we observe that for $u \in U_{m, \Delta}^0$ and $v \in U_{m, \Delta}^{2n}$, using \eqref{expected_passage_time_estimate} and an easy calculus argument we have, there exist $c,C>0$  (only depending on $\phi, \Delta$) such that 
\begin{equation}
\label{compairing_expectations}
    \mathbb{E}(T_{\widetilde{u_i},u})+ \mathbb{E}(T_{u,v})+ \mathbb{E}(T_{v,\widetilde{v_j}}) \geq \mathbb{E}(T_{\widetilde{u_i},\widetilde{v_j}})-c\Delta^2 n^{1/3}-Cn^{1/3}.
\end{equation}
We define the following events.
\begin{itemize}
    \item $\mathcal{C}_i:=\{\inf_{u \in U_{m, \Delta}^0} \widetilde{T_{\widetilde{u_i},u}} \geq -xn^{1/3}\}$;
    \item $\mathcal{C}_j:=\{\inf_{v \in U_{m, \Delta}^{2n}} \widetilde{T_{v,\widetilde{v_j}}} \geq -xn^{1/3}\}$;
    \item $\mathcal{C}_{i,j}:=\{\widetilde{T_{\widetilde{u_i},\tilde{v_j}}} \leq -(c \Delta^2+C+3x)n^{1/3}\}.$
\end{itemize}
We observe that using \eqref{compairing_expectations}
\begin{equation}
\label{small_probability_second_equation}
\mathcal{C}_{i} \cap \mathcal{C}_{i,j} \cap \mathcal{C}_{j} \subset \mathcal{B}_{i,j}.
\end{equation}
First we want to find a lower bound for $\mathcal{C}_{i,j}$. Note that 
\begin{displaymath}
    \{T_{\widetilde{u_i},\widetilde{v_j}}-c_{i,j}n\leq -(c \Delta^2+C+3x+C_2)n^{1/3}\} \subset \mathcal{C}_{i,j},
\end{displaymath}
where $c,C$ are as in \eqref{compairing_expectations}, $C_2$ is as in \eqref{expected_passage_time_estimate} and $c_{i,j}$'s are respective time constants (i.e.,\linebreak$\lim_{n \to \infty}\frac{\mathbb{E}(T_{\widetilde{u_i},\widetilde{v_j}})}{n}=c_{i,j}$) and they are uniformly bounded (the bounds depend only on $\phi$). We will find lower bound for 
    $\mathbb{P}(T_{\widetilde{u_i},\widetilde{v_j}}-c_{i,j}n\leq -(c \Delta^2+C+3x+C_2)n^{1/3})$. In \cite[Theorem 2]{BGHK21} we take 
    \begin{displaymath}
        \epsilon=\frac{(c \Delta^2+C+3x+C_2)n^{1/3}}{c_{i,j}n}.
    \end{displaymath}
    Recall that in \cite[Theorem 2]{BGHK21} we get that there exists $c'$ such that for $0 < \epsilon <c'$, \cite[(11)]{BGHK21} holds. Further, when $0 < \epsilon \leq c'' \frac{\sqrt{n}}{\sqrt{m}}$ then the second inequality in  \cite[(11)]{BGHK21} holds. Note that if we take $n$ sufficiently large then in the setup of Proposition \ref{small_probability_lemma} $\epsilon < c'$ and  and $0 < \epsilon < c''\frac{\sqrt{n}}{\sqrt{m}}$. So, we have 
    \begin{displaymath}
        \mathbb{P}(T_{\widetilde{u_i},\widetilde{v_j}}-c_{i,j}n\leq -(c \Delta^2+C+3x+C_2)n^{1/3}) \geq C_0e^{-\kappa (c \Delta^2+C+3x+C_2)^3} \geq Ce^{-cx^3}.
    \end{displaymath} where $C_0$ is as in \cite[Theorem 2]{BGHK21}, $\kappa$ is a constant depending on $\phi$ and $C,c$ are constants depending on $\Delta$ and $\phi$.\\
Now we consider $\mathcal{C}_i$ and $\mathcal{C}_j$. Using \cite[Theorem 4.2(i)]{BGZ21} we have $\mathbb{P}((\mathcal{C}_i)) \leq e^{-\frac{c x^3}{\mu}}$. So, we can chose $\mu$ small enough so that $\mathbb{P}((\mathcal{C}_i)^c) \leq \frac{Ce^{-cx^3}}{4}$. Similarly, we can chose small $\mu$ small enough so that $\mathbb{P}((\mathcal{C}_j)^c) \leq \frac{Ce^{-cx^3}}{4}$. So, for this choice of $\mu $, $\mathbb{P}(\mathcal{C}_{i} \cap \mathcal{C}_{i,j} \cap \mathcal{C}_{j}) \geq \frac{Ce^{-cx^3}}{2}.$ Hence, $\mathbb{P}(\mathcal{B}_{i,j}) \geq Ce^{-cx^3}$. This completes the proof. \qed
\section{Lower bound for semi-infinite geodesic intersections}
\label{Lower Bound for Coalescence of Semi-Infinite Geodesics}
In this Section we will give an outline of the lower bounds in Theorem \ref{second_theorem} and Theorem \ref{main_theorem_1}(ii). The details will be similar to the proof of Proposition \ref{positive_probability_for_point_to_line}. Hence, we will not prove the lower bounds in details. First we prove Lemma \ref{busemann function and semi-infinite geodesic lemma}.
\begin{comment}state a result which shows how Busemann functions are related to semi-infinite geodesics. Let $\alpha \in (0, \frac{\pi}{2})$ and $r \in \mathbb{Z}$ are fixed and $\Gamma^{\alpha}_{u,\mathcal{L}_{r}}$ denotes the geodesic under the boundary condition $\{B^{\alpha}_{v}\}_{v \in \mathcal{L}_r}$. Further, as mentioned before $\Gamma_u^{\alpha}$ denotes the semi-infinite geodesic starting from $u$ in the direction $\alpha.$ Then the following equality holds.
and the definition of semi-infinite geodesics implies the following. Any semi-infinite geodesic starting from $u \in {\{v \in \mathbb{Z}^2: \phi(v)<r\}}$ in the direction $\alpha$ will intersect $\mathcal{L}_r$ at a vertex $v$ such that $T_{u,v}+B^{\alpha}_v$ has the maximum value over $\mathcal{L}_r.$ Hence, we have the following.
\begin{lemma}
In the above notations
\begin{equation}
\label{semi-infinite_geodesic_and_Busemann_Function}
    \Gamma^{\alpha}_{u,\mathcal{L}_{r}}=\Gamma_u^{\alpha}|_{\{v \in \mathbb{Z}^2: \phi(u) \leq \phi(v) \leq r\}}.
\end{equation}
\end{lemma}
\end{comment}
\begin{proof}[Proof of Lemma \ref{busemann function and semi-infinite geodesic lemma}]
    For a fixed $\alpha$ we first fix a probability 1 set on which we will prove the lemma. Let $\Omega_1^{\alpha}$ denote the probability 1 set on which unique semi-infinite geodesics in the direction $\alpha$ exist. Let $\Omega^{\alpha}$ denote the probability 1 set as defined in the definition of two valued Busemann function (i.e., \eqref{Busemann_Function_Defintion} holds on $\Omega^{\alpha}$). Let $\Omega=\Omega_1^{\alpha} \cap \Omega^{\alpha}$. Let $\omega \in \Omega$ and consider $\Gamma_{u}^{\alpha}(\omega).$ We will work on this probability 1 set. By definition $\Gamma_{u}^{\alpha}$ is a random sequence $\{v_1, v_2,....\}$ such that 
    \begin{displaymath}
        \lim_{n \rightarrow \infty} \frac{v_n}{\|v_n\|}=\alpha.
    \end{displaymath}
    So, for all $v \in \mathcal{L}_r$
    \begin{displaymath}
       B^{\alpha}_{v}:=\lim_{n \rightarrow \infty}[T_{v,v_n}-T_{r_{\alpha},v_n}].
  \end{displaymath}
  Now, we have  $\Gamma^{B_\alpha}_{u,\mathcal{L}_{r}}$ (this is the point to line geodesic under the boundary condition $\{B^{\alpha}_v \}_{v \in \mathcal{L}_r}$) intersects $\mathcal{L}_r$ at $v_0$ such that
  \begin{displaymath}
      T_{u,v_0}+B^{\alpha}_{v_0}=\max_{v \in \mathcal{L}_r} \left(T_{u,v}+B^{\alpha}_v \right).
  \end{displaymath}
Also, note that $\Gamma^{\alpha}_{u,\mathcal{L}_{r}}$ is a geodesic between $u$ and $v_0$ in exponential last passage percolation without any boundary condition. So, if we can prove that $\Gamma_{u}^{\alpha} \cap \mathcal{L}_r=\{v_0\}$ then the lemma will follow. We prove this now. By contradiction assume that there exists $v'$ with $T_{u,v'}+B^{\alpha}_{v'} <  T_{u,v_0}+B^{\alpha}_{v_0}$ such that 
\begin{displaymath}
    \Gamma_{u}^{\alpha} \cap \mathcal{L}_r=\{v'\}.
\end{displaymath}
By \eqref{Busemann_Function_Defintion} there exists $n \in \mathbb{N}$ such that 
\begin{equation}
\label{smaller_Busemann_Value}
    T_{u,v'}+ T_{v',v_n}<  T_{u,v_0}+T_{v_0,v_n} \leq T_{u,v_n}.
\end{equation}
 As we have assumed that $v', v_n \in \Gamma_{u}^{\alpha}$ the left hand side of \eqref{smaller_Busemann_Value} is $T_{u,v_n}$. Hence, we get a contradiction. This proves the lemma.
\end{proof}
\subsection{Proof of Theorem \ref{second_theorem} lower bound}As mentioned before it is enough to prove Proposition \ref{colealscence_probability_for_general_direction}. Also, to prove Proposition \ref{colealscence_probability_for_general_direction}, it is enough to show Proposition \ref{Busemann_functions_coealsce} due to Lemma \ref{busemann function and semi-infinite geodesic lemma}. We outline the proof of Proposition \ref{Busemann_functions_coealsce} now. Recall that Proposition \ref{Busemann_functions_coealsce} says that any Busemann geodesic (in direction $\alpha$) starting from an $Mn^{2/3}$ (for sufficiently large $M$) interval around $-n_{\alpha}$ (recall the definition of $n_{\alpha}$ as we defined before stating Theorem \ref{second_theorem}) on $\mathcal{L}_{-2n}$ will coalesce at a single point on $\mathcal{L}_{0}$ with positive probability.
\\ For the Busemann geodesic (in direction $\alpha$) in general condition we need to consider the boundary condition to be the Busemann function in general direction, $\{B^{\alpha}_v\}_{v \in \mathcal{L}_{2n}}$. 
\begin{comment}Only difficulty in this case is that, when $\alpha \neq \frac{\pi}{4}$ then the collection $\{B^{\alpha}_v\}_{v \in \mathbb{Z}^2}$ satisfies \nameref{assumption_1} but not \nameref{assumption_2}. 
\end{comment}
In this case, on $\mathcal{L}_{2n}$, Busemann functions are two sided random walk (see \eqref{busemann_function_increment}). Let $\alpha \in (\epsilon, \frac{\pi}{2}-\epsilon)$ and consider the point $n_{\alpha}:=\left( \frac{2n}{1+\tan \alpha}, \frac{2n \tan \alpha}{1+\tan \alpha} \right) \in \mathcal{L}_{2n}$. From \eqref{expected_passage_time_estimate}, we have for $\epsilon$ there exists $C_2>0$ (depending only on $\epsilon$) such that 
\begin{displaymath}
    |\mathbb{E}(T_{\boldsymbol{0},n_{\alpha})}-\frac{2n(1+\sqrt{\tan \alpha})^2}{(1+\tan \alpha)}| \leq C_2n^{1/3}.
\end{displaymath}
$\frac{2n(1+\sqrt{\tan \alpha})^2}{(1+\tan \alpha)}$ will be called the \textit{time constant in the direction $\alpha$} and will be denoted by $m_{\alpha}n$ (note that for the $45^{\circ}$ direction case $m_{\alpha}n=4n).$ We wish to analyse how the expected passage time behaves if on $\mathcal{L}_{2n}$ we move $xn^{2/3}$ distance away from $n_{\alpha}$. So, for $x \in \mathbb{Z}^2$ let $w_x:=\left( \frac{2n}{1+\tan \alpha}+xn^{2/3}, \frac{2n \tan \alpha}{1+\tan \alpha}-xn^{2/3} \right)$. Then using \eqref{expected_passage_time_estimate}, a Taylor series argument we have 
\begin{equation}
\label{weight_loss_in_general_direction}
    |\mathbb{E}(T_{\boldsymbol{0},w_x})-m_{\alpha}+x\left(\frac{1-\tan \alpha}{\sqrt{\tan \alpha}}\right)n^{2/3}+x^2n^{1/3}| \leq Cn^{1/3}
\end{equation}
for some constant $C>0$ depending on $\epsilon$. Note that for the $\frac{\pi}{4}$ case the $n^{2/3}$ term vanishes and we just have $x^2n^{1/3}$ loss from the time constant as we moved $xn^{2/3}$ distance away from $\boldsymbol{n}$ on $\mathcal{L}_{2n}$. Also observe that from \eqref{busemann_function_increment} (also see \cite[Theorem 4.2]{Sep17})
\begin{equation}
\label{Busemann_Function_Expectation}
    \mathbb{E}(B^{\alpha}_{w_x})=x\left(\frac{1-\tan \alpha}{\sqrt{\tan \alpha}}\right)n^{2/3}.
\end{equation}
When $\alpha=\frac{\pi}{4},$ this expectation is $0$.
Now, for some small $\delta$, let us define the event $\mathcal{B}^{\alpha}_M$.
\begin{itemize}
    \item $\mathcal{B}^{\alpha}_M:=\{$ for all $v \in \mathcal{L}_{2n},$ for all $1 \leq i \leq n^{1/3}$, such that $|\psi(v)-\psi(n_{\alpha})| \leq iMn^{2/3}$, we have  $|B^{\alpha}_v-\mathbb{E}(B^{\alpha}_v)| \leq i^{1+\delta}M^{1+\delta}n^{1/3}\}$.
\end{itemize}
Then a direct application of Kolmogorov's maximal inequality we have the following lemma.

\begin{comment}For some $\delta>0$ define the following events.
    \begin{itemize}
    \item $B^{\alpha}_{+}:=\{$for all $1 \leq i \leq [\frac{n^{1/3}}{M}]$ and for all $v \in \mathcal{L}_{2n}$ such that $\psi(v) \in [0, iMn^{2/3}], |B^{\alpha}_v-\mathbb{E}(\mathcal{B}^\alpha_v)|)< i^{1+\delta}M^{1+\delta}n^{1/3}$\};
    \item $B^{\alpha}_{-}:=\{$for all $1 \leq i \leq [\frac{n^{1/3}}{M}]$ and for all $v \in \mathcal{L}_{2n}$ such that $\psi(v) \in [-iMn^{2/3},0], |B^{\alpha}_v-\mathbb{E}(\mathcal{B}_v^{\alpha})| < i^{1+\delta}M^{1+\delta}n^{1/3}$\};
    \item $B^{\alpha}:=B^{\alpha}_{+} \cap B^{\alpha}_{-}.$
    \end{itemize}
    \end{comment}
   \begin{lemma}
    \label{Busemann_Function_loss}
   For sufficiently large $M$ and sufficiently large $n$ (depending on $M$) we have 
   \begin{displaymath}
       \mathbb{P}(\mathcal{B}^{\alpha}_M) \geq 0.98.
 \end{displaymath}
 \end{lemma}
 Further, in this case we define the event $\boldsymbol{\mathrm{Restr}}$ as follows. Let $P_n^{\alpha}$ be a line segment with length $Mn^{2/3}$ and midpoint $-n_{\alpha}$ then for all $u \in P_n^{\alpha}$ we define
 \begin{itemize}
      \item $\boldsymbol{\mathrm{Restr}}_{u, \Delta}^{\alpha}:=\{$ all $\gamma:u \rightarrow v$, for some $v \in \mathcal{L}_{2n}$ and for all $1 \leq i \leq n^{1/3}$ with $\gamma \cap (U_{n_\alpha}^{iM})^c \neq \emptyset, \ell(\gamma) \leq \mathbb{E}(T_{u,v})-ci^2M^2n^{1/3}$\}
 \end{itemize}
 Using the above we construct a large probability event $\boldsymbol{\mathrm{Restr}}^{\alpha}$ as we have constructed before in the proof of Proposition \ref{positive_probability_for_point_to_line}.
Combining all the above, let us consider any path $\gamma$ starting from a point $u$ on $Mn^{2/3}$ line segment around $-n_{\alpha}$ ending on $ v \in \mathcal{L}_{2n}$. If it goes out of the $Mn^{2/3}$ rectangle around the straight line joining $-n_{\alpha}$, then $\ell(\gamma)+B^{\alpha}_v$ will suffer a loss of order at least $(M^2-M^{1+\delta})n^{1/3}$ from $m_{\alpha}n$ with large probability (the key point here is that let $\gamma$ ends at $w_x$ for some $x$. When we are taking the sum $\ell(\gamma)+B^{\alpha}_{w_x}$, the $x\left(\frac{1-\tan \alpha}{\sqrt{\tan \alpha}}\right)n^{2/3}$ term gets cancelled with $\mathbb{E}(B^{\alpha}_{w_x})$ due to \eqref{weight_loss_in_general_direction} and \eqref{Busemann_Function_Expectation}. So on $\boldsymbol{\mathrm{Restr}}^{\alpha} \cap \mathcal{B}_M^{\alpha},$ the overall penalty is still of order at least $(M^2-M^{1+\delta})n^{1/3}$ as before). Hence, we again construct barrier regions, large probability events ($\mathcal{H}^{\alpha}, \mathcal{I}^{\alpha}, \widetilde{\mathcal{I}}^{\alpha}$) and compare weights of different paths as done in \eqref{a path with large weight}, \eqref{large_last_passage_time}, \eqref{bad_paths_are_penalised}, \eqref{penalised_heavily_if_does_not_intersect_gamma_1}, \eqref{another_path_with_large_weight} to force the geodesics to coalesce. One final thing to note is that, as $\mathcal{B}^{\alpha}_M$ is independent of $\{\tau_v\}_{\{v \in \mathbb{Z}^2: \phi(v) <2n\}}$, we can intersect $\mathcal{B}_M^{\alpha}$ with the large probability events to get positive probability events and finally using a conditioning argument similar to Proposition \ref{positive_probability_for_point_to_line} we prove the proposition \ref{Busemann_functions_coealsce}. \qed

\subsection{Proof of \ref{main_theorem_1}(ii) lower bound}To prove the lower bound, same as before, by an averaging argument the following lemma will be sufficient. Consider the parallelograms $A_M,B_M$ and $V_{\frac M2}$ as defined in the proof of Theorem \ref{main_theorem_1}(i) lower bound.
\begin{comment}
Consider two line segments $P_n,Q_n$ on $\mathcal{L}_{-2n}$ each of length $2Mn^{2/3}$ ($M$ will be fixed later) with midpoints $-\boldsymbol{n},-\boldsymbol{n}_k$ respectively. Let
$V_{M}$ denote the parallelogram $\{v \in \mathbb{Z}^2: |\phi(v)| < \frac{n}{100k}, |\psi(v)| < Mn^{2/3}\}$.
\end{comment}
Now consider the following events.
\begin{comment}Let $u_1,u_2$ (resp. $u_1',u_2'$) denote the end points of $P_n$ (resp. $P_n^*$) and $w_1,w_2$ (resp. $w_1',w_2'$) denote the end points of $Q_n$ (resp. $Q_n^*$). 
For any point $u \in \mathbb{Z}^2$ and a line segment $T_{u,A}$ attaining path will be denoted by $\Gamma_{u,A}$.
\end{comment}
\begin{itemize}
    \item $\widetilde{\mathcal{E}_1}:=\{\Gamma^{\theta_0}_{a_1}(t)=\Gamma^{\theta_0}_{a_2}(t)$,  $\forall a_1,a_2 \in A_M$ and $ \forall t \in [-\frac{n}{200k},\frac{n}{200k}] \cap \{\phi(\Gamma^{\theta_0}_{a_1}(s))\}$\};
    \item $\widetilde{\mathcal{E}_2}:=\{\Gamma^{\theta_k}_{b_1}(t)=\Gamma_{b_2}^{\theta_k}(t)$, $ \forall b_1,b_2 \in B_M$ and $\forall t \in [-\frac{n}{200k},\frac{n}{200k}] \cap \{\phi(\Gamma^{\theta_k}_{b_1}(s))\}$\};
   \item $\widetilde{\mathcal{E}_3}:=\{ \Gamma^{\theta_0}_{a} \cap \Gamma^{\theta_k}_{b} \subset V_{\frac{M}{2}}, \forall a \in A_M, b \in B_M$\}
    \item $\widetilde{\mathcal{E}}=\widetilde{\mathcal{E}_1} \cap \widetilde{\mathcal{E}_2} \cap \widetilde{\mathcal{E}_3}.$
    \end{itemize}
\begin{lemma}
\label{positive_probability_for_point_to_line_for_semi-infinite}
    For sufficiently large $M$ (depending only on $\epsilon$) there exists a constant $c$ (depending on $M$) such that for all sufficiently large $n$ (depending on $M$) we have 
    \begin{displaymath}
        \mathbb{P}(\widetilde{\mathcal{E}}) \geq c.
    \end{displaymath}
\end{lemma}
Similar as before, due to Lemma \ref{busemann function and semi-infinite geodesic lemma}, if we can prove a coalesce result similar to Lemma \ref{positive_probability_for_point_to_line_for_semi-infinite}, for Busemann geodesics, then Lemma \ref{positive_probability_for_point_to_line_for_semi-infinite} will follow. In this case we need to consider Busemann functions in two different directions. We define the following events. 
\begin{itemize}
    \item $\mathcal{E}_1^{B}:=\{\Gamma^{B_{\theta_0}}_{a_1,\mathcal{L}_0}(t)=\Gamma^{B_{\theta_0}}_{a_2, \mathcal{L}_0}(t)$,  $\forall a_1,a_2 \in A_M$ and $ \forall t \in [-\frac{n}{200k},\frac{n}{200k}] \cap \{\phi(\Gamma^{B_{\theta_0}}_{a_1,\mathcal{L}_0}(s))\}\}$;
    \item $\mathcal{E}_2^{B}:=\{\Gamma^{B_{\theta_k}}_{b_1, \mathcal{L}_0}(t)=\Gamma^{B_{\theta_k}}_{b_2, \mathcal{L}_0}(t)$,  $\forall b_1,b_2 \in B_M$ and $ \forall t \in [-\frac{n}{200k},\frac{n}{200k}] \cap \{\phi(\Gamma^{B_{\theta_k}}_{b_1, \mathcal{L}_0}(s))\}\}$;
    \item $\mathcal{E}_3^{B}:=\{ \Gamma^{B_{\theta_0}}_{a, \mathcal{L}_0} \cap \Gamma^{B_{\theta_k}}_{b, \mathcal{L}_0} \subset V_{\frac{M}{2}}, \forall a \in A_M, b \in B_M$\}
    \item $\mathcal{E}^B=\mathcal{E}_1^{B} \cap \mathcal{E}_2^{B} \cap \mathcal{E}_3^{B}.$
    \end{itemize}
    We have the following lemma. 
    \begin{lemma}
\label{positive_probability_for_point_to_line_for_Busemann_geodesic}
    For sufficiently large $M$ (depending only on $\epsilon$) there exists a constant $c$ (depending on $M$) such that for all sufficiently large $n$ (depending on $M$) we have 
    \begin{displaymath}
        \mathbb{P}(\mathcal{E}^{B}) \geq c.
    \end{displaymath}
\end{lemma}
\paragraph{\textit{Sketch of proof}} By similar argument as in Lemma \ref{positive_probability_lemma} we can restrict ourselves to semi-infinite geodesics starting from line segments $P_n$(resp.\ $Q_n$) in the direction $\theta_0$(resp.\ $\theta_k$), where $P_n,Q_n$ are line segments on $\mathcal{L}_{-2n}$ each of length $2Mn^{2/3}$ ($M$ will be fixed later) with midpoints $-\boldsymbol{n},-\boldsymbol{n}_k$ respectively. As described in the proof of Theorem \ref{second_theorem} lower bound we construct large probability events $\mathcal{H}^0, \mathcal{I}^0, \widetilde{\mathcal{I}}^0, \boldsymbol{\mathrm{Restr}}^0, \mathcal{H}^{k}, \mathcal{I}^{k}, \widetilde{\mathcal{I}}^{k} ,\boldsymbol{\mathrm{Restr}}^{k}.$ We consider the collections of Busemann functions $\{B^0_{v}\}_{\{v \in \mathcal{L}_{2n}\}}$ and $\{B^k_{v}\}_{\{v \in \mathcal{L}_{2n}\}}$. Further, we consider the events $\mathcal{B}^0_M$ and $\mathcal{B}_M^{k}$. Note that, as we have already shown before both of these events are large probability events, their intersection has positive probability and the intersection is independent of $\{\tau_v\}_{\{v \in \mathbb{Z}^2: \phi(v)<2n\}}$. As done in the proof of Proposition \ref{positive_probability_for_point_to_line} we construct the barrier regions. We also consider the positive probability event as done in Lemma \ref{barrier_events_lemma}. Rest of the proof is similar to what we did in the proof of Theorem \ref{second_theorem} lower bound. As argued before, using \eqref{weight_loss_in_general_direction}, \eqref{Busemann_Function_Expectation}, on $\mathcal{H}^0\cap \widetilde{\mathcal{I}}^0 \cap \mathcal{H}^{k} \cap  \widetilde{\mathcal{I}}^{k} \cap \boldsymbol{\mathrm{Restr}}^0 \cap \boldsymbol{\mathrm{Restr}}^k \cap \mathcal{B}_M^0 \cap \mathcal{B}_M^k$, we can compare weights of different paths and get similar conclusions as in \eqref{a path with large weight}, \eqref{large_last_passage_time}, \eqref{bad_paths_are_penalised}, \eqref{penalised_heavily_if_does_not_intersect_gamma_1}, \eqref{another_path_with_large_weight}. Finally, applying a conditioning argument as done in Lemma \ref{Final_Conditioning_lemma} we prove the lemma. \qed
\bibliography{intersection}
\bibliographystyle{plain}
\end{document}